\documentclass{amsart}

\usepackage{graphicx}
\usepackage{color}
\usepackage{framed}
\usepackage{amsthm}
\usepackage{amscd}
\usepackage{amsmath}
\usepackage{mathrsfs}
\usepackage{amsfonts}
\usepackage{amssymb}
\usepackage{psfrag}
\usepackage{amstext}
\usepackage{constants} 
\usepackage[colorlinks]{hyperref}

\hypersetup{ colorlinks=true, pdftitle={Solenoidal attractors with bounded combinatorics are shy}, pdfsubject={Solenoidal attractors with bounded combinatorics are shy}, pdfauthor={Daniel Smania},pdfkeywords={ renormalization, rigidity,  universality, multimodal, hyperbolicity, hyperbolic set, solenoidal attractor, shy set}}

\newconstantfamily{c}{
symbol=\lambda,
format=\arabic,
}
\newconstantfamily{e}{
symbol=\theta,
format=\arabic,
}

\begin{document}

\author{Daniel Smania }

\address{Departamento de Matem\'atica \\
   Instituto de Ci\^encias Matem\'aticas e de Computa\c{c}\~ao, Universidade de S\~ao Paulo-Campus de S\~ao Carlos (ICMC/USP/S\~ao Carlos), Caixa Postal 668, CEP 13560-970 \\ S\~ao Carlos-SP, Brazil}
\email{smania@icmc.usp.br} \urladdr{\href{http://conteudo.icmc.usp.br/pessoas/smania/}{http://conteudo.icmc.usp.br/pessoas/smania/}}

\date{\today}
\dedicatory{ Dedicated to the memory of  \\ Welington de Melo (1946-2016)} 
\title[Solenoidal attractors with  bounded combinatorics are shy]{Solenoidal attractors with bounded combinatorics are shy}

\begin{abstract}  We show that in a generic   finite-dimensional real-analytic family  of real-analytic multimodal maps, the subset of parameters on which   the corresponding map has a  solenoidal attractor with bounded combinatorics is a set with  zero Lebesgue measure.  \end{abstract}


\thanks{We thank the referees for their careful reading and suggestions. D.S. was partially supported by CNPq 430351/2018-6, 307617/2016-5, 470957/2006-9, 310964/2006-7,  472316/03-6,  303669/2009-8, 305537/2012-1 and FAPESP 03/03107-9,  2008/02841-4, 2010/08654-1, 2017/06463-3.}

\maketitle
\newcommand{\co}{\mathbb{C}}
\newcommand{\incl}[1]{i_{U_{#1}-Q_{#1},V_{#1}-P_{#1}}}
\newcommand{\inclu}[1]{i_{V_{#1}-P_{#1},\co-P}}
\newcommand{\func}[3]{#1\colon #2 \rightarrow #3}
\newcommand{\norm}[1]{\left\lVert#1\right\rVert}
\newcommand{\norma}[2]{\left\lVert#1\right\rVert_{#2}}
\newcommand{\hiper}[3]{\left \lVert#1\right\rVert_{#2,#3}}
\newcommand{\hip}[2]{\left \lVert#1\right\rVert_{U_{#2} - Q_{#2},V_{#2} -
P_{#2}}}
\newtheorem{prop}{Proposition}[section]

\newtheorem{lem}[prop]{Lemma}

\newtheorem{rem}[prop]{Remark}

\newtheorem*{mth}{Main Theorem}
\newtheorem{thm}{Theorem}

\newtheorem{cor}[prop]{Corollary}

\setcounter{tocdepth}{1}
\tableofcontents

\section{Introduction.}

A multimodal map $f\colon I \rightarrow I$  is a smooth map defined in an interval $I$, with a finite number
 of critical points $c_i$, all of them local maximum or local minimum, and such that $f(\partial I) \subset \partial I$.  We are going to 
 assume that  $f$ is real-analytic. 
 
For {\it unimodal}  maps with a  {\it quadratic } critical point,  the   understanding of the {\it typical} behaviour  is very satisfactory.  Lyubich \cite{lhyp} and    Graczyk and \'Swiatek \cite{gs}  proved the density of hyperbolic parameters in the quadratic family. But this was not enough   to understand the typical behaviour  at almost every parameter of  the quadratic family.  Indeed  earlier Jakobson \cite{jak} proved that in the complement of the hyperbolic parameters there is a  subset of parameters with positive measure for which the dynamics admits  an absolutely continuous invariant probability (the map is stochastic).  Finally Lyubich \cite{lyuq}  proved that for almost every parameter  in the quadratic family the map is either regular (a hyperbolic map) or stochastic.  Avila, Lyubich and de Melo  \cite{alm} generalised  this result for a non degenerate real analytic family of quadratic real analytic unimodal maps and Avila and Moreira \cite{am2} improved this, proving that in a non degenerate family the map is either regular or Collet-Eckmann at almost every parameter.  There are  similar results  for real-analytic unimodal maps with higher order by Clark \cite{clark}. See also Bruin,  Shen and  van Strien  \cite{sbr}, Avila, Lyubich and Shen \cite{als} and Shen \cite{shen} for related results. 

Similar studies for multimodal maps (or even unimodal maps with higher order) pose new difficulties. New phenomena appear, as non-renormalizable maps without decay of geometry (see  Bruin, Keller, Nowicki and  van Strien \cite{wild2}, Keller and Nowicki \cite{wild3} ). Decay of geometry was an essential  tool in the study of unimodal quadratic maps.  This was a major difficulty in the study of the so-called Fibonacci renormalization for unimodal maps with higher order  in  Smania \cite{sm4} and the proof of  the density of hyperbolicity for polynomials  in  Kozlovski,  Shen,  van Strien \cite{kss1} \cite{kss2}.  Moreover the lack of decay of geometry allows additional  metric behaviours, as the  existence of wild attractors. See Milnor \cite{wild1}, Bruin, Keller, Nowicki and  van Strien \cite{wild2} and Bruin, Keller and St. Pierre \cite{wild4}.  

Another issue is that for families of  polynomials with more than one critical point  (as in the cubic family) the parameter space has dimension larger than one. That  implies that the parapuzzle approach  as used in the unimodal case (see Lyubich \cite{lyu3}, Avila, Lyubich and de Melo \cite{alm}) does not seem to be easily adaptable here, since  the fact that holomorphic maps with one-variable are conformal was  used in a crucial way. 

So as a consequence there are a lot of unanswered questions concerning  the typical behaviour  in the {\it measure-theoretical} sense in families of polynomials and/or multimodal maps.   

One of them  is   how often maps with  {\it solenoidal attractors}  appear in these families.  We say that a set $\Lambda \subset I$ is a  {\it solenoidal attractor} of a multimodal map $f$ if there exists an increasing sequence of positive intergers $n_k$, $k \in \mathbb{N}$, and  a family of closed intervals $I^k_j \subset I$, $k \in \mathbb{N}$ and $0\leq j < n_k$, such that 
\begin{itemize}
\item[A.] For each $k$ the intervals in the family $\{ I^k_j  \}_{j < n_k}$ has pairwise disjoint interior.
\item[B.] We have $f(I^k_j) \subset  I^k_{j+1 \mod n_k}$.
\item[C.] For every $k$
$$\{c_i\}_i \cap \cup_{j < n_{k}} I^k_j \neq  \emptyset.$$
and
$$\cup_{j < n_{k+1}} I^{k+1}_j \subset  \cup_{j < n_{k}}  I^k_j.$$
\item[D.] We have
$$\Lambda= \cap_k  \cup_{j< n_k}  I^k_j.$$
\end{itemize}
See  Blokh and Lyubich \cite{bl} \cite{bl2} for more information on attractors for multimodal maps.  The solenoidal attractor $\Lambda$ has {\it bounded combinatorics } if 
$$\sup_k \frac{n_{k+1}}{n_k} < \infty.$$

One important step in previous results about the typical behaviour in  families of unimodal maps is to prove that at a  typical parameter the map {\it does not } have solenoidal attractors. This was done in the quadratic family by Lyubich \cite{lyu3} and for no degenerate families of unimodal maps by Avila, Lyubich and de Melo \cite{alm}.  An important tool in many of these results on unimodal maps is the fact that the topological  classes of unimodal maps extend to an analytic,  codimension one lamination  (except a few combinatorial types).  This implies that the holonomy  of this lamination is quite regular. Our goal  is to prove that  \\ 

\noindent {\bf Theorem A.} { \it On a generic  real-analytic finite-dimensional  family of real-analytic multimodal maps with quadratic critical points and   negative schwarzian derivative the set of parameters whose corresponding maps have a solenoidal attractor with bounded combinatorics has zero Lebesgue measure.}  \\ 

\noindent The precise statement is given in Theorem \ref{thmA}. We  also have an analogous result for families with finite smoothness and  continuous  families. The method used in the unimodal case in  Avila, Lyubich and de Melo \cite{alm}  no longer works in the multimodal case, once the lamination of topological classes  has higher codimension, so we are going to use a quite different approach.   If  a map $f$ has a solenoidal attractor with bounded combinatorics, one can find an induced map $F$ of $f$ that is   a composition of unimodal maps and it is infinitely   renormalizable as defined in \cite{sm2}. In particular the iterations of the {\it renormalization operator $\mathcal{R}$ for multimodal maps} are  well-defined for  $F$. Using the universality property proved in  \cite{sm2} one can prove that  $F$  belongs to the stable lamination of the omega-limit set $\Omega$ of $\mathcal{R}$.  The renormalization operator is a real-analytic, compact and  non-linear operator acting on a Banach space of real analytic multimodal maps.  

Our main technical result is that  \\

\noindent  {\bf Theorem B.} { \it  Consider the renormalization operator $\mathcal{R}$ acting on  real-analytic multimodal maps which  are renormalizable with combinatorics bounded by some $p > 0$. Then the omega-limit set $\Omega$ of $\mathcal{R}$  is a hyperbolic set.} \\

\noindent The precise statement is given in  Section \ref{hypsection}. Lyubich\cite{lyu} proved the hyperbolicity of the omega-limit set in the unimodal case using  the so-called Small Orbits Theorem. We use a different approach, reducing  the study of the  hyperbolicity of $\Omega$ to  the study of the existence  and regularity of solutions for a certain linear cohomological equation.   This new method allows us to deal only  with real-analytic maps and its complex analytic extensions. 

 The relationship  between renormalization and cohomological equations appears in many contexts, as for instance  in the study of rigidity of circle diffeomorphisms and generalized interval exchange transformations.  Closer to our setting we have  the  introduction by  Lyubich \cite{lyu} of the concept of horizontal direction in the study of the renormalization operator for unimodal maps and  the study of the hyperbolicity of the  fixed point of the action of a pseudo-Anosov map on certain character variety by Kapovich \cite{kapo}.

The final ingredient is a  very recent result on partially hyperbolic invariant sets on Banach spaces \cite{sm5}. The result we use is, roughly speaking,  the following  (see \cite[Theorem  1]{sm5}). Suppose that a ``regular" real-analytic operator $\mathcal{R}$ has a hyperbolic set $\Omega$, and its stable lamination $W^s(\Omega)$ satisfies the ``Transversal Empty Interior property": every regular  manifold $M$ that is  transversal to $W^s(\Omega)$ intersects  $W^s(\Omega)$ in a subset of empty interior (in the topology of $M$). Then   a generic  real-analytic finite-dimensional   family  intersects $W^s(\Omega)$ on a subset with zero Lebesgue measure. This will give us our main result. The Transversal Empty Interior property for the renormalization operator (see Corollary \ref{tei}) is closely  related with the fact that maps $F$ that are infinitely renormalizable with bounded combinatorics can be approximated by hyperbolic maps.

Some of the most classical  families of one-dimensional dynamical systems are  families of polynomials. The {\it cubic family} is the two parameter family

$$f_{a,b}(z)= z^3- 3a^2 z + b$$
The critical points of $f_{a,b}$ are $a,-a$.  We also have  \\

\noindent {\bf Theorem C.} { \it  The set of of parameters $(a,b)\in \mathbb{R}^2$ such that $f_{a,b}$ is infinitely renormalizable with bounded combinatorics has zero  $2$-dimensional Lebesgue measure. }\\

The study of the renormalization operator has a long history. It was first discovered in the unimodal case by Feigenbaum \cite{fei1}\cite{fei2} and Coullet and Tresser \cite{ct}. They conjectured that the period-doubling renormalization operator has a unique fixed point in a  space of  quadratic unimodal maps, that this  fixed point is hyperbolic and its codimension one stable manifold contains all Feigenbaum maps. Such conjectures  could explain certain intriguing universal features of the bifurcation diagram of families of unimodal maps. The existence and hyperbolicity of such fixed point was proven by Lanford \cite{lanford}.  Such conjectures were later extended for arbitrary bounded combinatorial types, when the fixed point need to be replaced  by an omega-limit set that is hyperbolic (see Derrida,  Gervois and Pomeau \cite{dgp} and Gol'berg, Sina{\u\i} and Khanin \cite{sinai}).   Sullivan \cite{ds} proved that the orbit by the renormalization operator of a map that is infinitely renormalizable with bounded combinatorics converges to the orbit of a map on the omega limit set and such orbit is determined by the combinatorics  of the map.   This Sullivan's result in particular   implies  that uniqueness of the fixed point to the period-doubling renormalization operator and that it attracts all Feigenbaum maps. McMullen\cite{mc2} proved that the rate of convergence is indeed exponential. Finally Lyubich \cite{lyu} proved that  the omega-limit set of the renormalization operator for unimodal maps is hyperbolic. In particular Lyubich found a suitable space where the renormalization operator is a complex-analytic non-linear operator. See also de Faria, de Melo and Pinto \cite{fma} for the proof of the conjectures in the $C^r$ case. 

The renormalization operator for bimodal maps was first considered in MacKay  and van Zeijts \cite{mackay} and Hu \cite{hu}. The general multimodal case, with a precise combinatorial description, was described in \cite{sm1}, as well the so-called  real and complex a priori bounds for bounded combinatorics. In  \cite{sm3} it was proved the phase space  universality in the bounded combinatorics case. 

It is natural to ask if results as Theorems A., B. and C. holds for  the full renormalization operator, that is, considering {\it unbounded } combinatorial type as well.  We believe  that recent results by  Avila and Lyubich \cite{al33}  on  the contraction of the renormalization operator in the hybrid class of infinitely renormalizable unimodal maps with unbounded combinatorics can be carried out  for multimodal maps. So the main difficulty seems to be to understand the dynamics of the renormalization operator in the directions transversal to the horizontal spaces, as in the  proof of Theorem B. New difficulties arise in the unbounded case, once the omega-limit set of the renormalization operator has not a simple structure anymore. However we are confident that a version of the Key Lemma (Theorem \ref{keylem})  can be obtained in this setting and  it   will be useful to understand the dynamics of the renormalization operator and the generic behavior in families of multimodal maps.


\section{Renormalization of extended maps.}

\begin{figure}
\psfrag{f}{$F$}
\psfrag{f2}{$F^2$}
\psfrag{f5}{$F^5$}
\psfrag{rf}{$R(F)$}

\includegraphics[scale=0.45]{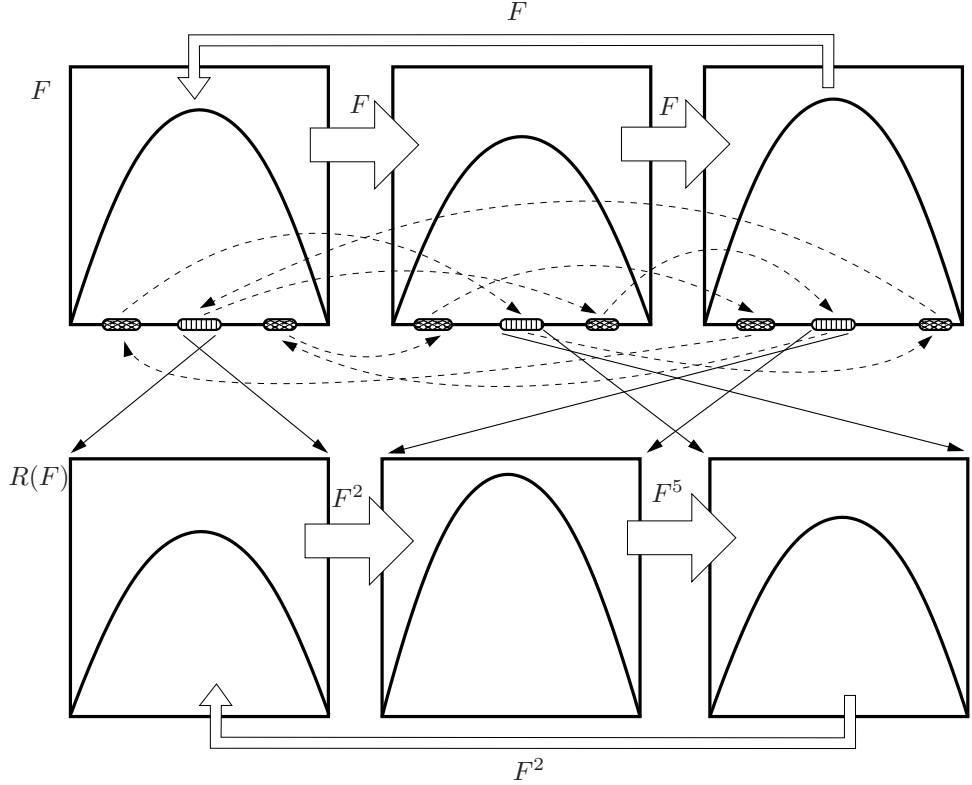}
\caption{Renormalization of an extended map of type $3$. }
\end{figure}

To study the renormalisation of multimodal maps, it  is more convenient to decompose the dynamics of $f$ in its unimodal
 parts. Let $I_{i}=[-1,a_i]$, with $a_i > 0$, be intervals and 
\begin{equation} \label{aqs} f_i\colon I_i \rightarrow I_{i + 1 \text{ mod }n}\end{equation}
 be  $C^1$ maps such that $c_i$ is its  unique critical point, that is a maximum and $f_i(\partial I_i)\subset \partial I_{i + 1 \text{ mod }n}$. An  extended map $F$ is defined by a finite sequence $(f_1,\dots,f_n)$ of  maps is  the map
defined on $I^n_F = \{(x,i) \colon x \in I_{i},  1 \leq i \leq n  \}$ as

\begin{equation}
F(x,i) = (f_i(x), i + 1 \text{ mod }n) 
\end{equation}

We say that $f$ is a multimodal map of type $n$ if it
can be written as a composition of $n$ unimodal maps: to be more precise, if there exist  maps
 $f_1, \dots, f_n$ as above satisfying 
\begin{enumerate}
\item $f = f_n \circ \dots \circ f_1$.
\item We have  $f_i(c_i) \geq c_{i+1 \text{ mod }n}$.
\end{enumerate}
The $n$-uple $(f_1,...,f_n)$ is a decomposition of $f$. In this paper, we will assume
that the unimodal maps are analytic and the critical points of $f_i$ are quadratic. Clearly $f$ has many
decompositions. 

In \cite{sm1}, we proved that deep renormalizations of infinitely renormalizable multimodal
 maps are multimodal maps of type $n$.

\subsection{Renormalization of   extended maps}\label{re_ex}  We say that $J$ is a $k$-periodic interval, $k\geq 2$, of the extended map
  $F$ if
\begin{itemize}
\item $(c_1,1) \in J$ ($(c_i,i)$ are  the critical points of $F$),
\item $\{ J, F(J), \dots, F^{k-1}(J) \}$ is a collection  of intervals with disjoint interiors,
\item The union of intervals in the above family contains $\{ (c_i,i) \}$,
\item $F^k(J) \subset J$.
\end{itemize}
We will call $k$ the period of $J$. If $F$ has a $k$-periodic interval, for some $k$, we say that $F$ is  renormalizable. 

Suppose that there exists a $k$-periodic interval for $F$. Let  $P \subset I_1\times \{1\}$ be the maximal interval
 which is a $k$-periodic interval for $F$. Then $F^{k}(\partial P) \subset \partial P$. We say
 that $P$ is a restrictive interval for $F$ of period $k$. Note that if $P$ and
$\tilde{P}$ are, respectively, restrictive intervals for $F$ of period $k$ and $\tilde{k}$,
$k < \tilde{k}$, then $\tilde{P} \subset P$.  Let $P$ be a restrictive interval and let
 $0 = \ell_1 < \dots < \ell_n$ be the iterations  such that $(c_i,i) \in F^{\ell_j}(P)$ for some $i$.
 Let $P_j$ be the symmetrization of $F^{\ell_j}(P)$ in relation to $(c_i,i)$. Observe that
$P_j$ contains a periodic point in its boundary. If $(c_i,i) \in P_j$ Let 
$$A_{P_j}\colon \mathbb{C}\times \{i\} \rightarrow \mathbb{C}\times \{j\} $$ be the affine map which maps
 $(c_i,i)$ to $(0,j)$ and this periodic point to $-1$. Let $[-1,b_j]\times \{j\}=A_{P_j}(P_j)$. Then 
 $$g_j\colon [-1,b_j]\times \{j\} \rightarrow [-1,b_{j+1}]\times \{j+1\}$$
 defined by 
$g_j =  A_{P_{j+1}} \circ F^{\ell_{j+1}-\ell_{j}}\circ A_{P_j}^{-1}$  is a unimodal map.
The extended map $G(x,j)= g_j(x,j)$ is called a   renormalization of the extended map $F$.   An extended map may have many renormalizations, but at most one with a given period. The renormalization with minimal period $k$ is called the first renormalization of $F$, and it is denoted $R(F)$. 

Following the notation in \cite{sm2}, the   primitive marked combinational data (primitive m.c.d)  associated with the first renormalization  of $F$  is  $\sigma = <A, \prec, A^c>$ where
\begin{itemize}
\item $A=\{ 1, 2, \dots, k\}$,
\item The relation $\prec$ is a partial order on $A$ defined in the following way $i \prec \ell $ if $F^iJ$ and $F^\ell J$ belongs to  the same  interval in $I^n_F$ and $F^iJ$ is on the left side of $F^\ell J$, 
\item The set $A^c$ is a subset of $A$ and $i \in A^c$ if $F^iJ$ intersects $\{ (c_i,i) \}$.
\end{itemize} 

The extended map $R(F)$ can be renormalizable again and so on. If this process can be continued
 indefinitely, we say that $F$ is infinitely renormalizable. If $F$ is infinitely renormalizable then all of its  renormalizations can be obtained iterating the operator $R$.  Denote by $P^k_0$ the restrictive interval
associated to the $k$-th renormalization $R^k(F)$. If $q \in
C(F):=\{(c_i,i)\}$, denote by the corresponding
 capital letter $Q_0^k$ the symmetrization of the interval $F^{\ell}(P_0^k)$ which contains $q$.
 We reserve the letter $p$ for $(c_1,1)$. The critical point $r$ for $F$ will be the
successor of the critical point $q$ at level $k$ if $r \in F^\ell(Q_0^k)$, for the
 minimal $\ell$ so that $F^\ell(Q_0^k)$ contains a critical point. Define $n^k_r = \ell$.
 Then, for any $r \in C(F)$, $k \in \mathbb{N}$ and $i < n^k_r$, there exists an interval
 $R^k_{-i}$ so that
\begin{itemize}
\item $F^i$ is monotone in $R^k_{-i}$,

\item $F^i(R^k_{-i})=R^k_{0}$,

\item The interval $F^{n^k_r - i}(Q^k_0)$ is contained in $R^k_{-i}$.
\end{itemize}
For details, see \cite{sm1}.

Denote by $N_k$ the period of the restrictive interval $P^k_0$. We say that $F$ has
C-bounded combinatorics if  $N_{k+1}/N_k \leq C$ for every $k$.

For $(x,i), (y,j) \in I^n_F$, we say that $(x,i) < (y,j)$ if $i = j$ and $x < y$. The
intervals of  $I^n_F$ are the sets $J \times \{i\}$, for some interval $ J \subset I_i$ and $1\leq i \leq  n$.
If $c_i$ is the critical point of $f_i$, denote $C(F)=\{(i,c_i)  \}_i$.

Let $F$ and $G$ be two infinitely renormalizable extended maps. We say that
$F$ and $G$ have same combinatorics if $F^i(c_k) < F^j(c_\ell)$ if and only if  $G^i(c_k) < G^j(c_\ell)$, for any
$i$,$j$ $\geq 0$ and $k$ and $\ell$ $< n$. 

Let $\sigma_i$ be the primitive m.c.d. of the (first) renormalization of $R^i(F)$ and $\tilde{\sigma}_i$ be the primitive m.c.d. of the (first) renomalization of $R^i(G)$. It turns out that $F$ and $G$ has the same combinatorics if and only if $\sigma_i=\tilde{\sigma}_i$ for every $i$. So we say that $F$ has combinatorics $(\sigma_1,\sigma_2,\sigma_3,\dots)$. Moreover, let $\mathcal{C}_{p,n}$  be the set of all  primitive m.c.d. that appears as the first renormalization of an extended map with $n$ intervals and it has period either smaller or equal to $p$. By Corollary 2.3 in \cite{sm2} for every given sequence $\sigma_i \in \cup_p \mathcal{C}_{p,n}$, $i\geq 1$, there exists a real analytic extended map (with, say, quadratic critical points)  whose $i$-th renormalization has primitive m.c.d. $\sigma_i$.

\subsection{Polynomial-like extended maps}  Denote $\mathbb{C}_n = \{(x,i) \colon x \in \mathbb{C}, 1 \leq i \leq n  \}$ (in other words, $\mathbb{C}_n$ is a disjoint union of $n$ copies of $\mathbb{C}$). Given an open set $O \subset \mathbb{C}_n$, denote 
$$O_i =  O\cap (\mathbb{C}\times \{i\}).$$
  A polynomial-like extended map is a map $F\colon U \rightarrow V$, where 
\begin{itemize}
\item $U$ and $V$ are open sets of $\mathbb{C}_n$, where $\overline{U} \subset V$, 
\item for each $i$, $F(U_i)=V_{i+1 \ mod \ n}$. Moreover $F\colon U_i \rightarrow V_{i+1 \mod n}$ is a proper map with a unique critical point. 
\item  for each $i$ we have that $U_i$ and $V_i$ are simply connected domains.
\end{itemize}
We define 
$$\mod(V\setminus U)= \min_i \mod(V_i\setminus U_i).$$

The filled-in Julia set $K(F)$ of a polynomial-like extended map $F$ is defined as

$$K(F)=\cap_{i\geq 0} F^{-i}(V).$$
Note that $K(F)$ is  connected if and only if all the critical points of $F$ belongs to $K(F)$. 

A  real analytic  extended map $F\colon I^n_F \rightarrow I^n_F$ has  a polynomial-like extension  if there is a polynomial-like extend map $\tilde{F}\colon U \rightarrow V$ such that $I\times \{i\} \subset U_i$ and $F=\tilde{F}$ on $I^n_F$. 
 
\subsection{Polynomial-like  renormalization of  real analytic  extended maps}\label{po_re} Let $U \subset \mathbb{C}_n$ be an open set. Given an analytic  function $F\colon U  \rightarrow \mathbb{C}_n$,  define the open set
$$\mathcal{D}^n_{U}(F):= \bigcap_{i=0}^{n-1}F^{-i}U.$$
In other words, $\mathcal{D}^n_{U}(F)$ is the domain contained in $\mathbb{C}_n$ where
$F^{n}$ is defined.

Let $F\colon U \rightarrow V$ be a complex-analytic extension  of an extended map.  Note that $F$ does not need to be a polynomial-like extension. Suppose that the extended map $F$ is $r$-times renormalizable, and  let  $P_j$ and $\ell_j$ be the intervals and integers associated with the $r$th renormalization, as  defined in Section \ref{re_ex}. Define $n_k = \ell_{k+1 \ mod \ n}-\ell_k$. 

Suppose we can find a sequence $i_k$, $k=1,\dots, n$, with $i_1=1$, simply connected domains $\hat{U}_k$ and $\hat{V}_k \subset \mathbb{C}\times \{i_k\}$,  such that \begin{itemize}
\item[1.] We have  $\{ i_k\}_k =\{1, 2, \dots, n\}$,
\item[2.] We have  $P_k \subset  \hat{U}_k$ and  $\overline{\hat{U}_k}\subset \hat{V}_k$, 
\item[3.]  The domains $\hat{U}_k$ and $\hat{V}_k$ satisfies  $\overline{\hat{U}_k} \subset \mathcal{D}^{n_k}_{U}(F)$, $F^{n_k}\hat{U}_k=\hat{V}_{k+1 \mod n}$. 
\item[4.]  The map  $$F^{n_k}\colon \hat{U}_k \rightarrow \hat{V}_{k+1\mod n}$$
is a proper map for which $(0,i_k)$ is the unique critical point. 
\end{itemize}

Let $A_j\colon \mathbb{C}\rightarrow \mathbb{C}$ be the affine maps  defined in Section \ref{re_ex}. Define  $A\colon \mathbb{C}_n \rightarrow \mathbb{C}_n$ as $A(x,i)= (A_i(x),i)$.  Then we can define  $$g_i\colon A(\hat{U}_j) \rightarrow A(\hat{V}_{j+1}) $$
as $g_j =  A \circ F^{n_k}\circ A^{-1}$. If 
$$\hat{U}= \bigcup_i A(\hat{U}_i) \times \{i\}, \ \hat{V}= \bigcup_i A(\hat{V}_i) \times \{i\},$$
then the map 
$$R^r(F)\colon \hat{U}\rightarrow \hat{V}$$
defined by $R^r(F)(x,i)=(g_i(x),i+1)$ is a polynomial-like extension of the real renormalization $R^r(F)$ and it is called a  polynomial-like $r$th renormalization of  $F$. Note that $\hat{U}_k$ and $\hat{V}_k$ are not uniquely defined, however any  polynomial-like extension of $R^r(F)$ does coincide on $I^n_{R^r(F)}$.

\begin{lem}\label{koenigs} Let $x_1,\dots,x_{n}\in \mathbb{C}$, and $U_i\subset \mathbb{C}$ be open  sets such that $x_i \in U_{i_k}$ and
$$f_{i,k}\colon U_{i,k}\rightarrow \mathbb{C}, \ k=1,2,$$
be holomorphic functions such that 
\begin{itemize}
\item[i.]   $f_{i,k}(x_i)=x_{i+1 \mod n}$.
\item[ii.] We have  $|\lambda_1\lambda_2\cdots\lambda_{n}|> 1$, where   $\lambda_i=f_{i,1}'(x_i)=f_{i,2}'(x_i)$.
\item[iii.] Let
$$g_{i,k}=f_{ (i+n-1) \, mod \,  n, \, k}\circ \cdots \circ f_{(i+1) \,  mod \,  n, \, k}\circ f_{i,k}.$$
There is $q\geq 1$ such that $g_{i,1}^q=g_{i,2}^q$ for every $i$.
\end{itemize}
Then $f_{i,1}(x)=f_{i,2}(x)$ for every $i$ and for every $x$ close to $x_i$.
\end{lem}
\begin{proof} Due iii. we have that $x_i$ is a repelling fixed point of $g_{i,k}$ and its multiplier is $\lambda_1\lambda_2\cdots\lambda_{n}$. By the K{\oe}nigs linearization theorem (see for instance Milnor \cite[Theorem 8.2]{milnor}) there is a {\it unique} germ of holomorphic function $h_i$ at $x_i$ such that $h'_i(x_i)=1$ and 
\begin{equation}  h_i\circ g_{i,1}(x) = g_{i,2} \circ h_i(x)
\end{equation} 
for $x$ close to $x_i$. Note that the uniqueness of $h_i$ and ii. implies
\begin{equation}\label{aux}  h_{i+1}\circ f_{i,1}(x) = f_{i,2} \circ h_i(x).
\end{equation} 
for $x$ close to $x_i$. Note also that 
\begin{equation}  h_i\circ g_{i,1}^q(x) = g_{i,2}^q \circ h_i(x)
\end{equation} 
But since $g_{i,1}^q=g_{i,2}^q$ this implies (due the uniqueness of the solution of the Schr\"{o}der's equation in the K{\oe}nigs linearization theorem) that $h_i(x)=x$, so by (\ref{aux}) we have  $f_{i,1} = f_{i,2}$ for every $i$.
\end{proof}

\begin{prop}[Injectivity of Renormalization] \label{inj_ren} Let $F_1, F_2$ be real-analytic extended maps of type $n$ with polynomial-like extensions of type $n$.   Suppose that $F_k$, $k=1,2$ are  renormalizable and $R(F_k)$, $k=1,2$, also have polynomial-like extensions of type $n$. Additionally assume that 
\begin{equation} r \in \overline{\{F^\ell_k(s), \ \ell \geq 0 \}}, \text{ every $r,s \in C(F_k)$, $k=1,2$,}\end{equation}
Then $R(F_1)=R(F_2)$ implies $F_1=F_2$.
\end{prop}
\begin{proof} We use an argument similar to de Melo and van Strien \cite[Chapter VI, Proposition 1.1]{ms}. To simplify the notation we assume that $F_k(x,i)=F_k(-x,i)$. Let $p_k$ be the period of the first renormalization of $F_k$ and $q_k=p_k/n$. Using the notation of Section \ref{re_ex}, we have $c_i=0$, $a_i=1$ and $b_i=1$, for every $i\leq n$. Let $P_{1,k}$ be the interval of the first renormalization of $F_k$ such that that  $(0,1)\in P_{1,k}$.  Let
 $0 = \ell_{1,k} < \dots < \ell_{n,k}$ be the iterations  such that $(0,i_{j,k}) \in F^{\ell_{j,k}}_k(P_{1,k})$ for some $i_{j,k}$.
 Let $P_{j,k}=[-\beta_{j,k},\beta_{j,k}]$ be the symmetrization of $F^{\ell_{j,k}}_k(P_{j,k})$, where $\beta_{j,k}$ is periodic.   Let 
 $$Y_{j,k}=[-\frac{1}{\beta_{j,k}}, \frac{1}{\beta_{j,k}}].$$
  The map
 $$g_{j,k}\colon Y_{j,k}\rightarrow Y_{j,k}$$
 defined by
 $$g_{j,k}(x)=\frac{-1}{\beta_{j,k}} \pi_1(F_k^n(-\beta_{j,k} x, i_{j_k})).$$
 is a multimodal map with a polynomial-like extension of degree $2^n$ and the real trace of its filled-in Julia set is $[-1/\beta_{j,k}, 1/\beta_{j,k}].$  Then $R(F_1)=R(F_2)$ implies that $g_{j,1}^{q_1}=g_{j,2}^{q_2}$ on $[-1,1]$ and for every $j$. Moreover $g_{j,k}^{q_k}$ has a polynomial-like extension of degree 
 $2^n$ whose  real trace of its filled-in Julia set is $[-1, 1].$ Note that if $|\beta_{j,1}| < |\beta_{j,2}|$ then $Y_{j,2}$  is invariant by $g_{j,1}^{q_1}$, that implies that $Y_{j,2}$ would be a restricted interval of $g_{j,1}$, which is not possible since $g_{j,1}^{q_1}$ on $[-1,1]$ is the first renormalization of $g_{j,1}$. So 
 \begin{equation} \label{cond_1} \beta_{j,1}=\beta_{j,2}\text{  and }Y_{j,1}=Y_{j,2}\text{   for  every } j\end{equation}
  Counting the number of restricted intervals associated to $[-\beta_{j,k}, \beta_{j,k}]$ in $Y_{2,k}$ we obtain $p_1=p_2$.
 
 Note that $\ell_{1,k}$ is the number of critical values of $F^{\ell_{j,k}}_k$ in $[-1,1]\times \{1\}$, which  is equal to the number of critical values of $R(F_k)$ in $Y_{1,k}\times \{1\}$. Since  $R(F_1)=R(F_2)$ and $Y_{1,2}=Y_{1,2}$ we conclude that $\ell_{1,1}=\ell_{1,2}$ and $i_{2,k}=1+\ell_{1,1}$ for $k=1,2$. Suppose by induction that $i_{j,1}=i_{j,2}$. Then $\ell_{j,k}$ is the  number of critical values of $R(F_k)$ in $Y_{j,k}\times \{j\}$, so $\ell_{j,1}=\ell_{j,2}$ and consequently $i_{j+1,k}=i_{j+1,1}+\ell_{j,1}.$ So 
\begin{equation}\label{cond_2} i_{j,1}=i_{j,2}\text{  and }\ell_{j,1}=\ell_{j,2}\text{   for  every }  j.\end{equation}
 Finally due (\ref{cond_1}), (\ref{cond_2}) and $R(F_1)=R(F_2)$ we have that 
\begin{equation}\label{cond_3}F^{\ell_{j+1,1}-\ell_{j,1}}_1=F^{\ell_{j+1,2}-\ell_{j,2}}_2\end{equation}
 in a neighborhood of the point $(-1,i_{j,1})$ and  $F^{\ell_{j+1,1}-\ell_{j,1}}_1(-1,i_{j,1})=(-1,i_{j+1,1})$ for every $j$. \\
 
 Let $\lambda_{i,k}=DF_k(-1,i) > 0$.  We claim that $\lambda_{i,1}=\lambda_{i,2}$ for every $i$. Indeed, let $p=p_1=p_2$, $q=q_1=q_2$ and $\ell_i=\ell_{i,1}=\ell_{i,2}$. So $g_{j,1}^{q}=g_{j,2}^{q}$ and (\ref{cond_1}) implies that $F_1^{qn}=F_2^{qn}$ in a neighborhood of $[-1,1]\times \{1,\dots, n\}$. In particular
 $$(\lambda_{1,1}\cdots \lambda_{n,1})^q=(\lambda_{1,2}\cdots \lambda_{n,2})^q,$$
 so 
\begin{equation}\label{cond_4} \lambda_{1,1}\cdots \lambda_{n,1}=\lambda_{1,2}\cdots \lambda_{n,2}.\end{equation}
There is exactly one $j\leq n$ such that $(0,1)\in F_1^{\ell_j}(P_{1,1})$, that is the unique $i_1$  satisfying $\ell_{i_1}= w_1 n+ 1$, for some $w_1\in \mathbb{N}$. In particular
 $$DF^{\ell_{i_1}}_k(-1,0)=(\lambda_{1,k}\cdots \lambda_{n,k})^{w_1}\lambda_{1,k}.$$
Due (\ref{cond_3})  we have $DF^{\ell_{i_1}}_1(-1,0)=DF^{\ell_{i_1}}_2(-1,0)$, so it follows from (\ref{cond_4}) that $\lambda_{1,1}=\lambda_{1,2}$. Suppose by induction that $\lambda_{j,1}=\lambda_{j,2}$ for $j<   j_0< n$. Then there is a unique $\ell_{i_{j_0}}$ such that $\ell_{i_{j_0}}= w_{j_0} n+ j_0$ for some $w_{j_0}\in \mathbb{N}$ and consequently 
$$DF^{\ell_{i_{j_0}}}_k(-1,0)=(\lambda_{1,k}\cdots \lambda_{n,k})^{w_{j_0}}\lambda_{1,k}\lambda_{2,k}\cdots \lambda_{j_0-1,k} \lambda_{j_0,k}.$$
It follows from the induction assumption, $DF^{\ell_{i_{j_0}}}_1(-1,0)=DF^{\ell_{i_{j_0}}}_2(-1,0)$ and (\ref{cond_4}) that $\lambda_{j_0,1}=\lambda_{j_0,2}$. 
So $\lambda_{j,1}=\lambda_{j,2}$ for every $j\leq  n-q$. We conclude that $\lambda_{n,1}=\lambda_{n,2}$ due (\ref{cond_4}).  This concludes the proof of the claim.

Define $f_{i,k}(x)=    \pi_1(F_k(x,i))$ and $x_i=-1$. By Lemma \ref{koenigs} we have that $f_{i,1}(x)=f_{i,2}(x)$ for every $i$ and $x$ close to $-1$, so $F_1=F_2$. 


\end{proof}

\subsection{Complex bounds and  rigidity of  real analytic, infinitely renormalizable  extended maps with bounded combinatorics} Here we summarize results in \cite{sm2}.  

\begin{thm} \label{ri_to} Let $\sigma=(\sigma_i)_{i \in \mathbb{Z}}  \in \mathcal{C}_{p,n}^{\mathbb{Z}}$. Then there exists a unique sequence of real analytic maps $F_{\sigma,i}$, $i \in \mathbb{Z}$, satisfying the following conditions
\begin{itemize}
\item[1.] The map $F_{\sigma,i}$ is renormalizable and $R(F_{\sigma,i})=F_{\sigma,i+1}$. 
\item[2.] The first renormalization of $F_{\sigma,i}$ has combinatorics $\sigma_i$. 
\item[3.] There exist polynomial-like extensions $F_{\sigma,i}\colon U^i_\sigma \rightarrow V^i_\sigma$, where
$$\inf_i \mod(V^i_\sigma \setminus U^i_\sigma) > 0.$$
\end{itemize}
\end{thm}

If $U\subset \mathbb{C}$ is a bounded open set such that $0\in U$,  denote by $\mathcal{B}(U)$ the Banach space of all  
holomorphic functions $g\colon U \rightarrow \mathbb{C}$ that has a continuous extension to $\overline{U}$ and a critical point at $0$, with the sup norm. If  $-1 \in \overline{U}$ let  $\mathcal{B}_{nor}(U)$ be the affine subspace of maps $g\in \mathcal{B}(U)$ such that $g(-1)=-1$.

In an analogous way, let $U \subset \mathbb{C}_n$ be a bounded open set such that $$U_i=~(\mathbb{C}~\times~\{i \})\cap~U~\neq \emptyset$$ and $(0,i) \in U$,  for  every $i$ and consider the set $\mathcal{B}(U)$ of all holomophic functions $G\colon U \rightarrow \mathbb{C}_n$ with the following properties:
\begin{itemize}
\item[1.] $G$ has a continuous extension to $\overline{U}$. 
\item[2.] $G$ has critical points at $(0,i)$, for every $i$.
\item[3.] $G(U_i)\subset \mathbb{C}\times\{i+1 \ mod \ n\}$.
\end{itemize}

\begin{thm}(Complex Bounds \cite{sm1}\cite{sm2}) \label{co_bo} There exists $\epsilon_0 > 0$ with the following property. If  $F$ is  a real analytic extended map that is infinitely renormalizable with   combinatorics in $\mathcal{C}_{p,n}^{\mathbb{N}}$ and with a complex analytic (but not necessarily polynomial-like) extension $F \in \mathcal{B}(U)$  then there exist a neighbourhood $V_F \subset \mathcal{B}(U)$ of $F$ and   $k_0$ with the following property. For every $k \geq k_0$ and every real analytic and infinitely renormalizable $G \in V_F$ with   combinatorics in $\mathcal{C}_{p,n}^{\mathbb{N}}$  the map $G$ has a  polynomial-like $k$th  renormalization $$R^{k}(G)\colon U^k \rightarrow V^k$$
such that $$mod(V^k\setminus U^k) > \epsilon_0.$$ 
\end{thm}

If $(-1,i)\in \overline{U}$ for every $i$, we can also consider the subset $\mathcal{B}_{nor}(U)$ of all maps $G \in \mathcal{B}(U)$  such that $G(-1,i)=(-1,i+1 \ mod \ n)$ for every $i$. 

Denote $\pi(x,i)=x$. We  identify  $\mathcal{B}(U)$ with the  Banach space
\begin{equation}\label{id} \mathcal{B}(\pi(U_1))\times \mathcal{B}(\pi(U_2))\times \dots \times \mathcal{B}(\pi(U_n))\end{equation}
 in the following way.  For each  $G \in \mathcal{B}(U)$ there is a unique decomposition  
\begin{equation}\label{de} (g_1, \dots,g_n) \in \mathcal{B}(\pi(U_1))\times \mathcal{B}(\pi(U_2))\times \dots \times \mathcal{B}(\pi(U_n)),\end{equation} where $g_i$ is defined by $g_i(x)=\pi\circ G(x,i)$ and for each $n$-uple as in (\ref{de}) we can associate $G \in \mathcal{B}(U)$ defined by $G(x,i)=(g_i(x),i+1 \ mod \ n)$.  With this identification $\mathcal{B}_{nor} (U)$  turns out to be an affine subspace of $\mathcal{B}(U)$. So given $F\in \mathcal{B}_{nor} (U)$ we can consider the tangent space of $\mathcal{B}_{nor} (U)$ at $F$, denoted by $T_F \mathcal{B}_{nor} (U)$. using the identification (\ref{id}) then $T_F \mathcal{B}_{nor} (U)$ is the subspace of $$(v_1, \dots,v_n) \in \mathcal{B}(\pi(U_1))\times \mathcal{B}(\pi(U_2))\times \dots \times \mathcal{B}(\pi(U_n))$$
such that  $v_i(-1)=0$ for $i=1,\dots,n$. In particular $T_F \mathcal{B}_{nor} (U)$ does not depend on $F$, so sometimes we will write $T \mathcal{B}_{nor} (U)$. 

Given $\delta > 0$ and $\theta > 0$, let $D_{\delta,\theta}$ be the set 
$$ \{ x \in \mathbb{C}: \  dist(x,[-1,1]) < \delta \ and \ |Im(x)| < \theta (Re(x)+1)  \}\times \{1,\dots,n\}$$

Define 

$$\Omega_{p,n} = \{ F_{\sigma,0}   \}_{\sigma \in \mathcal{C}_{p,n}^{\mathbb{Z}}}$$   

Indeed, due  Theorem \ref{ri_to} we have that 
$$\Omega_{p,n} = \{ F_{\sigma,i}   \}_{\sigma \in \mathcal{C}_{p,n}^{\mathbb{Z}}}$$   
for every $i$. 
Using  Theorem \ref{co_bo} and Theorem \ref{ri_to} one can show that there exists $\epsilon_0$ such that for  every $F_{\sigma,0} \in \Omega_{p,n}$  there exists  a polynomial-like extension $F_{\sigma_0}\colon U_\sigma^0 \rightarrow V_\sigma^0$ such that $mod(V_\sigma^0 \setminus U_\sigma^0) > \epsilon_0$.  There  exists $\delta_0$  such that for every simply connected domains $Q \supset W \supset [-1,1]$ such that $mod (Q\setminus W) \geq \epsilon_0/2$ 
we have 
\begin{equation}\label{delta_est} \{ x \in \mathbb{C}: \   dist(x,[-1,1]) \leq \delta_0 \} \subset  Q.\end{equation}
In particular
$$\overline{D_{\delta_0,\theta}} \subset U_\sigma^0$$
for every $\theta > 0$ and for every $\sigma \in \mathcal{C}_{p,n}$. In particular $\Omega_{p,n} \subset \mathcal{B}_{nor}(D_{\delta_0,\theta})$.

Consider the shift operator on $\mathcal{C}_{p,n}^{\mathbb{Z}}$, that is, if $\sigma=(\sigma_i)_{i \in \mathbb{Z}}$ then $S(\sigma)=\sigma'$, where $\sigma'_i =\sigma_{i+1}$.  
\begin{cor}\label{sh}  The set $\Omega_{p,n}\subset \mathcal{B}_{nor}(D_{\delta_0,\theta})$ is a Cantor set.  Indeed the map  $H\colon \mathcal{C}_{p,n}^{\mathbb{Z}} \rightarrow \Omega_{p,n} $ given by 
$$H(\sigma)=F_{\sigma,0},$$
is a homeomorphism.  Moreover $R(F_{\sigma,0})=F_{S(\sigma),0}$.
\end{cor}
\begin{proof} The map $H$ is continuous  and onto due  \cite[Section 7.1]{sm2}. The injectivity of $H$ follows from Proposition \ref{inj_ren}. 

\end{proof}

\section{Complexification of the renormalization operator $\mathcal{R}$.}

 \label{no_re} Given $\theta_0 > 0$,  by Theorem \ref{co_bo} for each $F \in \Omega_{p,n}$ there exist a neighbourhood $V_F \subset \mathcal{B}_{nor}(D_{\delta_0,\theta_0})$ of $F$ and $k_F$ such that for every real map $G \in V_F$ that is infinitely renormalizable with combinatorics in $\mathcal{C}_{p,n}$ and for every $k \geq k_F$ we have a polynomial-like  $k$th renormalization $R^k(G)\colon \hat{U}\rightarrow \hat{V}$ with $mod(\hat{V}\setminus\hat{U})> \epsilon_0$. In particular $R^k(G)\in \mathcal{B}_{nor}(D_{\delta_0,\theta_0})$.  Since $\Omega_{p,n}$ is a compact set, choose a finite sub cover $\{V_{F_i}\}_{i\leq \ell}$ of $\Omega_{p,n}$. Let $k_0 = \max_{i\leq \ell} k_{F_i}$ and $\mathcal{V}= \cup_{i\leq \ell} V_{F_i}$.

Let $H$ be the homeomorphism defined in Corollary \ref{sh}.  For every $\gamma=(\gamma_1,\dots,\gamma_{k_0})\in \mathcal{C}_{p,n}^{k_0}$ define the compact set 
$$\Omega_{p,n}(\gamma)=H(\{\sigma \in \mathcal{C}_{p,n}^{\mathbb{Z}}\colon \ \sigma_i=\gamma_i \ for \ 1\leq i\leq k_0\}).$$
We have 
 $$d_1=\inf \{ dist_{\mathcal{B}_{nor}(D_{\delta_0,\theta_0})}(G_1,G_2)\colon G_1 \in   \Omega_{p,n}(\hat{\gamma}), G_2 \in   \Omega_{p,n}(\tilde{\gamma}), \ \hat{\gamma}\neq \tilde{\gamma}\} >0.$$

Given $F \in \Omega_{p,n}$, consider the intervals $P_{F,j}$, $j=1, \dots, n$, integers  $n_j$, correponding the restrictive intervals of the $k_0$th renormalization of $F$, as in Section \ref{po_re}. Each interval $P_{F,j}$ contains a unique  repelling periodic  point $(\beta_{F,j},i_j)$ in its boundary. These repelling periodic points  have a complex analytic continuation $(\beta_{G,j},i_j)$ for every $G$ in a connected  neighbourhood $\tilde{W}_F$ of $F$ in $\mathcal{B}_{nor}(D_{\delta_0,\theta_0})$ that is also a repelling periodic point for $G$. Note that for a real map $G$  the point $\beta_{G,j}$ is real  and we can assume that it has the same combinatorics as $\beta_{F,j}$. We can also  assume that $\tilde{W}_F \subset \mathcal{V}$ and that the diameter of $\tilde{W}_F$ is smaller than $d_1/2$.

Let $d_2< d_1$ be a Lebesgue number of the cover $\{\tilde{W}_F\}_{F\in \Omega_{p,n}}$ 
of $\Omega_{p,n}$.
For every $F \in \Omega_{p,n}$ choose a connected neighbourhood  $W_F\subset  \tilde{W}_F$ of $F$ so that 
$$diam_{\mathcal{B}_{nor}(D_{\delta_0,\theta_0})}\ W_F < d_2/4.$$

 Let $F_1,F_2\in \Omega_{p,n}$  and  consider the complex analytic continuations $(\beta_{G,j}^1,i_j^1)$, $(\beta_{G,j}^2,i_j^2)$  of $(\beta_{F_1,j},i_j^1)$ and $(\beta_{F_2,j},i_j^2)$ defined for  every $G\in W_{F_1}$ and $G\in W_{F_2}$ respectively.  Suppose that $W_{F_1}\cap  W_{F_2} \neq \emptyset$. We claim that  $i_j^1=i_j^2$ and  $\beta_{G,j}^1=\beta_{G,j}^2$ for every $ G\in W_{F_1}\cap  W_{F_2}$ and $j$.   Since the diameter of $W_{F_1}\cup  W_{F_2}$ is smaller than $d_2$ we have that $W_{F_1}\cup  W_{F_2}\subset \tilde{W}_{F_3}$, for some $F_3\in \Omega_{p,n}$. Note that the distance between two maps in $\{ F_1, F_2, F_3\}$ is  smaller than $d_1$. In particular the combinatorics of their $k_0$th renormalizations are the same, so $i_j^1=i_j^2=i_j^3$ for every $j$.
Consider the complex analytic continuation $(\beta_{G,j}^3,i_j)$   of $(\beta_{F_3,j},i_j)$  defined for  $G\in \tilde{W}_{F_3}$. Then $(\beta_{F_1,j}^3,i_j)$ and $(\beta_{F_1,j},i_j)$ are repelling periodic points with the same combinatorics. Since $F_1$ has negative schwarzian derivative, the minimal principle implies that  $\beta_{F_1,j}^3=\beta_{F_1,j}$. In an analogous way $\beta_{F_2,j}^3=\beta_{F_2,j}$. The uniqueness of the analytic continuation of a repelling periodic point implies that $\beta_{G,j}^3=\beta_{G,j}^1$ for $G \in W_{F_1}$ and $\beta_{G,j}^3=\beta_{G,j}^2$ for $G \in W_{F_2}$. This concludes the proof of the claim. 

In particular the function 
$$G \mapsto (\beta_{G,j},i_j)$$
is well defined and complex analytic  in $\mathcal{W}= \cup_{F\in \Omega_{p,n}} W_F$. There is a small abuse of notation here since $i_j$ depends on $G$, but it is a locally constant function.

Fix   $G \in W_F$. Let $A_{G,j}\colon \mathbb{C}\times\{i_j\} \rightarrow \mathbb{C}\times\{j\}$  be the affine transformation that maps $(\beta_{G,j},i_j)$ to $(-1,j)$ and $(0,i_j)$ to $(0,j)$, and $A_G\colon \mathbb{C}_n \rightarrow \mathbb{C}_n$ as $A_G(x,i)=(A_{G,i}(x),i)$.   Let $D^{F,j}$ be the set 
$$ \overline{A^{-1}_{F,j} (\{ z \in \mathbb{C}: \  dist(z,[-1,1]) < \delta_0 \ and \ |Im(z)| < \theta_0 (Re(z)+1)   \}\times\{j\}) }.$$
Since  $mod(\hat{V}\setminus \hat{U}) > \epsilon_0$ we have that 
$$D^{F,j} \subset \hat{U}_j \subset  \mathcal{D}^{n_j}_{D_{\delta_0,\theta_0}}(F),$$
Moreover,  due the complex bounds, reducing  $\theta_0$ and $\delta_0$ we can assume that the interior of the sets in the family
$$\{ F^m(D^{F,j})   \}_{m< n_j}$$
are pairwise disjoint, and the intersection of the closure of every two of  those sets is contained in 
$$\{F^m(\beta_{F,j},i_j)  \}_{m< n_j}.$$
Let  $G \in W_F$  and define the set $D^{G,j}$ as 
$$\overline{A^{-1}_{G,j}(\{ z \in \mathbb{C} : \  dist(z,[-1,1]) < \delta_0 \ and \ |Im(z)| < \theta_0 (Re(z)+1)   \} \times \{j\})}.$$
Reducing the neighbourhood $W_F$ of $F$ and $\theta_0$  we can assume that  
$$D^{G,j} \subset  \mathcal{D}^{n_j}_{D_{\delta_0,\theta_0}}(G),$$
for every $G \in W_F$ and furthermore the interior of the sets in the family
$$\{ G^m(D^{G,j})   \}_{m< n_j}$$
are pairwise disjoint, and the intersection of the closure of every two of  those sets  is contained in 
$$\{G^m(\beta_{G,j},i_j)  \}_{m< n_j}.$$

Define the complexification of the renormalization operator  $$\mathcal{R}\colon  \mathcal{W} \rightarrow \mathcal{B}_{nor}(D_{\delta_0,\theta_0})$$ as
\begin{equation} \label{def_ren} \mathcal{R}(G)(x,j)=  A_{G,j+1} \circ G^{n_j}\circ A_{G,j}^{-1}(x,j)\end{equation}
if  $G \in W_F$.  The operator $\mathcal{R}$ is a compact complex analytic map. From now on   denote $U= D_{\delta_0,\theta_0}$.

\begin{rem} \label{sc} Let  $\tilde{U}$ be a little larger  complex open domain that contains $\overline{U}$. Consider the complex analytic transformation 
$$\tilde{\mathcal{R}} \colon  \mathcal{W} \rightarrow \mathcal{B}_{nor}(\tilde{U}).$$
defined exactly as in (\ref{def_ren}). Let $$i\colon \mathcal{B}_{nor}(\tilde{U}) \rightarrow \mathcal{B}_{nor}(U)$$
be the compact linear inclusion between these spaces. Then $\mathcal{R}= i\circ \tilde{\mathcal{R}}$, so the complexification of the renormalization operator is a strongly compact operator as defined in \cite{sm5}. 
\end{rem}

Let $v \in T_G\mathcal{B}_{nor}(U)$. If $z \in U$ and $G^j(z) \in U$ for every $j < i$ then $(G+tv)^i$ is defined in a neighbourhood of $z$ and we can define
\begin{equation}\label{de_it}  a_i(z)= \frac{\partial}{\partial t} (G+tv)^i|_{t=0}(z)= \sum_{j=0}^{i-1} DG^{i-j-1}(G^{j+1}(z))v(G^j(z)).\end{equation}

Given $F \in \Omega_{p,n}$ and $G \in W_F$. For each $v\in T_G\mathcal{B}_{nor}(U)$ and $z \in U_j$ we have

$$(D\mathcal{R}_G\cdot v)(x,j)= \frac{\partial}{\partial t} A_{G+tv,j+1}\circ  (G+tv)^{n_j}\circ A_{G+tv,j}^{-1}(x,j)|_{t=0} $$
$$= - \frac{\partial_G \beta_{G,j+1}\cdot v}{\beta_{G,j+1}}\cdot  A_{G,j+1} \circ  G^{n_j}\circ A_{G,j}^{-1}(x,j)$$
$$- \frac{1}{\beta_{G,j+1}}\big(    a_{n_j}\circ A_{G,j}^{-1}(x,j) + (\partial_x G^{n_j})\circ A_{G,j}^{-1}(x,j)\cdot (-\partial_G \beta_{G,j}\cdot v \ x,j) \big).$$
\begin{thm}\label{dense}  Let $F \in \mathcal{W}$. Then 
$D_F\mathcal{R}(T_{F}\mathcal{B}_{nor}(U))$ is dense in $T_{\mathcal{R}F}\mathcal{B}_{nor}(U)$.\end{thm} 
\begin{proof}  The proof is quite similar to the proof of the analogous statement in \cite{alm}. 
Let $w \in T_{\mathcal{R}F}\mathcal{B}_{nor}(U)$. Then $w(-1,k)=0$ and $w'(0,k)=0$ for every $k$. We are going to define a function 
$$\hat{v}\colon  \cup_j \cup_{m< n_j} G^m(D^{G,j}) \rightarrow \mathbb{C}$$
in the following way. Define the function $\hat{v}$ as 0 on 
$$\cup_j \cup_{0< m< n_j}  G^m(D^{G,j}),$$
and 
$$\hat{v}(z)= [DG^{n_j-1}(G(z))]^{-1}\cdot w\circ A_{G,j}(z)$$
for $z \in D^{G,j}$. Also define $\hat{v}(-1,k)=0$ for every $k$. Then $\hat{v}$ is well defined,  it is continuous on
$$ \Lambda= \cup_j \cup_{m< n_j} G^m(D^{G,j})  \cup \{(-1,k)\}_k,$$
and it is complex analytic in  the interior of $\Lambda$.

Moreover  $\hat{v}$ vanishes on the orbit of the periodic points $\{\beta_{G,j}\}_j$. Since $\mathbb{C}\times \{i\} \setminus \Lambda$ is a connected set, by Mergelyan's Theorem for each given $\epsilon > 0$ and $i$  we can find a polynomial $q_i$ such that $|\hat{v}(z)-q_i(z)|< \epsilon$ for $z \in \Lambda_i =\Lambda \cap \mathbb{C}\times \{i\}$. Define
$$\hat{q}_i(z) = q_i(z) -q'_i(0,i)z -q_i(-1,i)-q'_i(0,i)$$
Note that  $\hat{q}_i'(0,i)=0$ and $\hat{q}_i(-1,i)=0$. Define $q(x,i)=\hat{q}_i(x)$.  We have that $q \in T_{G}\mathcal{B}_{nor}(U)$ and 
$$|D_G\mathcal{R} \cdot q  - w|_{\mathcal{B}(U)} \rightarrow_{\epsilon \rightarrow 0} 0.$$\end{proof}

\section{Action of $D\mathcal{R}$ on  horizontal directions.} 

\subsection{Horizontal direction} \label{hdd} Let $F\colon I^n_F  \rightarrow I^n_F $ be a real analytic extended map that is either infinitely  renormalizable with bounded combinatorics in $\mathcal{C}_{p,n}$ or whose critical points belongs to the same periodic orbit.  A continuous function 
$$v\colon I^n_F \rightarrow  T \mathbb{C}_n$$
is  a  {\it horizontal direction }  of $F$ if 
\begin{itemize}
\item[1.] For each $x \in I^n_F $ we have $v(x) \in T_{F(x)}\mathbb{C}_n$.
\item[2.] The function $v$ is real analytic in the interior of $I^n_F$.
\item[3.] There is a quasiconformal vector field 
$$\alpha\colon W \rightarrow T\mathbb{C}_n,$$
defined  in a complex neighbourhood $W$ of the post critical set of $F$, such that 
\begin{equation}\label{tce3} v(x)= \alpha(F(x))- DF(x) \cdot \alpha(x)\end{equation}
for every $x$ in the post critical set. 
\item[4.] We have $\alpha(c)=0$ for every critical point $c$ of $F$. 
\end{itemize}

Denote by $E^h_F$ the set of $v \in T_F \mathcal{B}_{nor}(U)$ such that $v$ is horizontal. Of course  $E^h_F$ is a linear subspace of $T_F \mathcal{B}_{nor}(U)$.

 \begin{prop}[Infinitesimal pullback argument. Avila, Lyubich and de Melo  \cite{alm}] \label{ipa} Let  $F \in \Omega_{n,p}$.  Let $$F\colon W\rightarrow V$$ be a polynomial-like extension of $F$ and $v \in   \mathcal{B}(W)\cap T_F\mathcal{B}_{nor}(U)$ such that there exists a quasiconformal vector field $\alpha$, defined in a neighbourhood of the post critical set of $F$, such that 
 \begin{equation} \label{tcc} v(x) = \alpha  \circ F (x) - DF(x)\cdot \alpha(x)\end{equation} 
 for every $x \in P(F)$. In particular  $v \in E^h_F$.  Reducing a little bit the domain $W$, there exists a quasiconformal vector field extension $\alpha\colon W \rightarrow  \mathbb{C}$ such that (\ref{tcc}) holds for every $x\in W$. 
 \end{prop}
 
 \begin{prop}[Invariance] \label{invariance} Let $F\in \Omega_{n,p}$. Then \begin{equation} \label{r1} D_F\mathcal{R}(E^h_F) \subset E^h_{\mathcal{R}F},\end{equation} 
\begin{equation}\label{r2}    (D_F\mathcal{R})^{-1}(E^h_{\mathcal{R}F})\subset E^h_F.\end{equation}\end{prop}
\begin{proof} The proof of (\ref{r1})  is quite similar to the proof of a similar statement in \cite{sm3}.  Indeed, consider $a_i$ as in (\ref{de_it}). Note that 
$$a_i(z) = v(F^{i-1})+DF(F^{i-1}(z))a_{i-1}(z).$$
By an inductive argument one can show that 
$$a_i = \alpha\circ F^i - DF^i\cdot \alpha$$
on $P(F)$.  Denote 
$$ \alpha(\beta_{F,j+1})=\partial_F \beta_{F,j+1}\cdot v$$
then if $z=(x,j) \in P(F)$ we have
\begin{align*} &(D\mathcal{R}_F\cdot v)(x,j)\\
&= - \frac{\partial_F \beta_{F,j+1}\cdot v}{\beta_{F,j+1}}\cdot  A_{F,j+1} \circ  F^{n_j}\circ A_{F,j}^{-1}(x,j) \\
&- \frac{1}{\beta_{F,j+1}}\big(    a_{n_j}\circ A_{F,j}^{-1}(x,j) + (D F^{n_j})\circ A_{F,j}^{-1}(x,j)\cdot (-\partial_F \beta_{F,j}\cdot v \ x,j) \big)        \\
&=- \frac{\alpha(\beta_{F,j+1})}{\beta_{F,j+1}}\cdot  (\mathcal{R}F)(z)\\
&-    \frac{1}{\beta_{F,j+1}} \alpha \circ A_{F,j+1}^{-1} \circ A_{F,j+1}\circ F^{n_j} \circ A_{F,j}^{-1}(x,j)  + \frac{\beta_{F,j}}{\beta_{F,j+1}}  DF^{n_j}\circ A_{F,j}^{-1}(x,j) \cdot \frac{1}{\beta_{F,j}}\alpha \circ A_{F,j}^{-1}(x,j)\\
  &+ D_z(\mathcal{R}F) \cdot (\frac{\alpha(\beta_{F,j})}{\beta_{F,j}}\  x,j) \\
&=- \frac{\alpha(\beta_{F,j+1})}{\beta_{F,j+1}}\cdot  (\mathcal{R}F)(z)\\
&-    \frac{1}{\beta_{F,j+1}} \alpha \circ A_{F,j+1}^{-1} \circ (\mathcal{R}F)(z) + D_z(\mathcal{R}F) \cdot \frac{1}{\beta_{F,j}}\alpha \circ A_{F,j}^{-1}(x,j)\\
 &+ D_z(\mathcal{R}F) \cdot (\frac{\alpha(\beta_{F,j})}{\beta_{F,j}}\  x,j).
 \end{align*}
Define the vector field $r(\alpha)$ as
\begin{equation}\label{ralpha}  r(\alpha)(x,j)= -\frac{1}{\beta_{F,j}}\alpha\circ A_{F,j}^{-1}(x,j) -(\frac{\alpha(\beta_{F,j}) }{\beta_{F,j}}  \cdot x , j)\end{equation} 
for $z=(x,j) \in U_j$.  Then
\begin{equation} \label{ralpha2} (D\mathcal{R}_F\cdot v)(z)=  r(\alpha)\circ (\mathcal{R}F)(z) - D_z(\mathcal{R}F)\cdot r(\alpha)(z) \end{equation} 
for $z$ in the postcritical set of $\mathcal{R}F$.  Note that $r(\alpha)$ is a quasiconformal vector field in a neighbourhood of the post critical set of $\mathcal{R}F$.  So 
$D\mathcal{R}_F\cdot v \in  E^h_{\mathcal{R}F}$.

Now suppose that $v \in  (D_F\mathcal{R})^{-1}(E^h_{\mathcal{R}F})$.  Then $D\mathcal{R}_F\cdot v \in E^h_{\mathcal{R}F}$, so there exists a quasiconformal vector field $\gamma\colon \mathbb{C}_n \rightarrow \mathbb{C}$ such that 
\begin{equation} (D\mathcal{R}_F\cdot v)(z)=  \gamma \circ (\mathcal{R}F)(z) - D_z(\mathcal{R}F)\cdot \gamma(z). \end{equation} 
for every $z$ in a neighbourhood of the post critical set of $\mathcal{R}F$.  Define 
$$\delta_j = \partial_t \beta_{F+tv,j}\big|_{t=0}=\partial_F \beta_{F,j}\cdot v$$
 Define $\alpha$ in  $A_{F,j}(U)$ as 
 $$\alpha(z)= \beta_{F,j} \gamma\circ A_{F,j}^{-1}(z) +  \delta_j A_{F,j}^{-1}(z).$$
 
 Let $z$ be a point very close to the post critical set of $F$. Then
$$\{ k\geq 0  \ s.t. \ F^k(z) \in \cup_{j\leq n} A_{F,j}(U) \} \neq \emptyset$$
Let $k(z)$ be a minimal element of the above set. Not hat $z \mapsto k(z)$ is locally constant.  We define $\alpha(z)$ for $z$ close to the post critical set of $F$ by induction of $k(z)$. We already defined $\alpha(z)$ when $k(z)=0$. If $k(z) > 0$ then $k(F(z))=k(z)-1$ and we define 
$$\alpha(z)= \frac{v(z)+\alpha(F(z))}{DF(z)}.$$

One can check that $\alpha$ is a  quasiconformal vector field and $$v = \alpha(F(z))-DF(z)\alpha(z)$$ in a neighbourhood of the post critical set of $F$, so $v \in E^h_F$. 
 \end{proof}

 \begin{prop} Let  $F \in \Omega_{n,p}$. Then $D\mathcal{R}_F$ is injective.
 \end{prop} 
 \begin{proof} Let $v$ be such that $D\mathcal{R}_F\cdot v =0$. By Proposition \ref{invariance} we have that $v \in E^h_F$. By Proposition \ref{ipa} there is a quasiconformal  vector field $\alpha$, defined in a neighborhood of  the Julia set of $F$, satisfying (\ref{tcc}) for every $x$ on its Julia set. Let $r(\alpha)$ be the quasiconformal vector field defined by (\ref{ralpha}). Then by (\ref{ralpha2}) we have that $r(\alpha)$ satisfies 
 $$0=  r(\alpha)\circ (\mathcal{R}F)(z) - D_z(\mathcal{R}F)\cdot r(\alpha)(z).$$
 for every $z$ in the Julia set of $\mathcal{R}F$.  We can easily conclude that $r(\alpha)(z)=0$ at every repelling periodic point  $z$ of $\mathcal{R}F$  and consequently at every point of its Julia set.  By (\ref{ralpha}) we have that $\alpha$ is zero at every point of the small Julia sets of $F$ corresponding to this renormalization and, by (\ref{tcc}) we have that $v$ vanishes in these small Julia sets as well. So $v=0$ everywhere. 
 \end{proof}

\begin{prop}[Closedness] \label{closed} Let  $F_k \in \Omega_{n,p}$ and $v_k \in E^h_{F_k} \subset T \mathcal{B}_{nor}(U)$ be sequences such that 
$(F_k,v_k)$ converges to $(F,v) \in \mathcal{B}_{nor}(U)\times T\mathcal{B}_{nor}(U)$. Then $F\in \Omega_{p,n}$  and $v \in E^h_F$. In particular $E^h_F$ is Banach subspace of $T_F\mathcal{B}_{nor}(U)$. 
\end{prop}
\begin{proof}  Due the definition of the operator $\mathcal{R}$, the map $\mathcal{R}F$ has a polynomial-like extension 
$$\mathcal{R}F\colon W\rightarrow V,$$
with $\overline{U}\subset W$. 
Reducing $V$ a little bit, we can assume $\partial V$ is a finite union of analytic curves and  that for $k$ large enough  the map 
$$\mathcal{R}F_k\colon W_k\rightarrow V,$$
where $W_k \subset \mathbb{C}_n$, $\overline{U} \subset W_k$,  is the set whose connected components  are the connected components of $$(\mathcal{R}F_k)^{-1}V$$
that intersect $\{ (0,j)  \}_j$, is a polynomial-like extension of $\mathcal{R}F_k$.  Since $v_k \in E^h_{F_k}$ we have that $D_{F_k}\mathcal{R}\cdot v_k \in E^h_{\mathcal{R} F_k}$, so there exists a quasiconformal vector field  $\tilde{\gamma}^k$ such that 
$$(D\mathcal{R}_{F_k}\cdot v_k)(z)=  \tilde{\gamma}^k\circ (\mathcal{R}F_k)(z) - D_z(\mathcal{R}F_k)\cdot \tilde{\gamma}^k(z). $$
holds for $z$ in a neighbourhood of the post critical set of $\mathcal{R}F_k$.

Now we use the infinitesimal pullback argument in Avila, Lyubich and de Melo \cite{alm}. For each $k$, there exist $C > 0$ and a quasiconformal vector field $\gamma_0^k\colon \mathbb{C}_n \rightarrow \mathbb{C}$ with the following properties \begin{itemize}
\item[1.] The vector field $\gamma_0^k$ vanishes outside $V$. Moreover $\gamma_0^k(-1)=\gamma_0^k(0)=0$.
\item[2.] It satisfies 
$$(D\mathcal{R}_{F_k}\cdot v_k)(z)=  \gamma^k_0\circ (\mathcal{R}F_k)(z) - D_z(\mathcal{R}F_k)\cdot \gamma_0^k(z). $$
for every $ z\in \partial W_k$. 
\item[3.] The vector field $\gamma_0^k$ is $C^\infty$ in a neighbourhood of 
$$\overline{V\setminus W_k}$$
and
$$|\overline{\partial} \gamma^k_0|\leq C$$
on this set. 
\item[4.]  We have $\gamma^k_0=\tilde{\gamma}^k$ in a neighbourhood  of the post critical set of $\mathcal{R}F_k$.
\end{itemize}
Define by induction $\gamma^k_j$ as $0$ outside $V$ and 
$$\gamma^k_{j+1}(z) = \frac{\gamma^k_j\circ (\mathcal{R}F_k)(z) - (D\mathcal{R}_{F_k}\cdot v_k)(z)   }{D_z(\mathcal{R}F_k)}$$
on $V\setminus \{(0,m)\}_m$, and $\gamma^k_{j+1}(0,m)=0$. 

Using the McMullen compactness criterion for quasiconformal vectors fields \cite[Corollary A.11]{mc2}, one can prove that for each $k$ the sequence

$$\hat{\gamma}_j^k  = \frac{1}{j} \sum_{t=0}^{j-1} \gamma_t^k$$
has a convergent subsequence,  uniform on compact subsets of $\mathbb{C}_n$. Moreover such limits are quasiconformal vectors fields. Let $\gamma^k_\infty$ be one of theses limits. Since  the filled-in Julia sets of the polynomial-like extensions of $\mathcal{R}F_k$ do not support  invariant line fields \cite{sm2} we conclude that $|\overline{\partial} \gamma^k_\infty| \leq C$ on $\mathbb{C}_n$.  Note that 
\begin{equation}\label{limi} (D\mathcal{R}_{F_k}\cdot v_k)(z)=  \gamma^k_{\infty} \circ (\mathcal{R}F_k)(z) - D_z(\mathcal{R}F_k)\cdot \gamma_{\infty}^k(z),  \ z \in \overline{U}.\end{equation}
By the compactness criterion for quasiconformal vectors fields in McMullen \cite{mc2}  we can consider a convergent subsequence $\gamma^{k_t}\infty \rightarrow_t  \gamma$, where $\gamma$ is a quasiconformal vector field on $\mathbb{C}_n$ and the convergence is uniform on compact subsets of $\mathbb{C}_n$ . By (\ref{limi}) we have
\begin{equation} (D\mathcal{R}_{F}\cdot v)(z)=  \gamma \circ (\mathcal{R}F)(z) - D_z(\mathcal{R}F)\cdot \gamma(z),  \ z \in \overline{U},\end{equation}
so $D\mathcal{R}_{F}\cdot v  \in E^h_{\mathcal{R}F}$, so by (\ref{r2}) we have $ v\in E^h_F$.  \end{proof}

\begin{prop}[Contraction on the horizontal directions] \label{contraction}  There exist $K$ and $\Cl[e]{contra} > 1$ such that for every $F\in \Omega_{n,p}$ and $v \in E^h_F$ we have 
$$ |D_F\mathcal{R}^i\cdot v|_{T\mathcal{B}_{nor}(U)} \leq K {\Cr{contra}}^{-i}|v|_{T\mathcal{B}_{nor}(U)}.$$\end{prop} 

We do not provide a proof for Proposition \ref{contraction} since it can be proven   in exactly the same way it is done in the unimodal setting. One can use  the  argument by Lyubich \cite[Theorem 6.3]{lyu}  using the Schwarz's lemma and the rigidity of McMullen's towers \cite{mc2}. An infinitesimal argument using the rigidity of McMullen's towers and the compactness of the renormalization operator  is given in \cite[Proposition 3.9]{sm3} (in the case of  the  fixed point of the period doubling renormalization) can  be also applied here. We also cite the new methods  by Avila and Lyubich \cite{al33} to prove the contraction in the horizontal directions  in the case of unimodal unbounded combinatorics.

\begin{prop}[Contraction on the hybrid classes] \label{contraction2} There exists $\Cl[c]{contra2}\in (0,1)$ with the following property. Let $F$ be a real-analytic  polynomial-like map of type $n$ that is infinitely renormalizable with combinatorics bounded by $p$.  Then there exist  $G \in \Omega_{n,p}$ and $k_0=k_0(F)$ and $C=C(F)$ such that $\mathcal{R}^kF \in \mathcal{B}(U)$ for every $k\geq k_0$ and 
$$|\mathcal{R}^kF - \mathcal{R}^kG  |_{\mathcal{B}_{nor}(U)}\leq C\Cr{contra2}^k \text{ for } k\geq k_0.$$ 
\end{prop} 
\begin{proof} One can prove this in a quite similar way to the proof of the main result in \cite{sm2}.  An alternative proof is obtained using Proposition \ref{contraction} and  the same argument as in the proof of Theorem 1  in \cite{sm3}.
\end{proof}

Next we  show that every map in $\Omega_{p,n}$ can be approximated the hyperbolic  polynomial-like maps of type $n$.  

\begin{prop}\label{lprop} Let $G \in \mathcal{W}$ be such that  there exist domains $\hat{U}$ and $\hat{V}$, whose boundaries are  analytic Jordan curves, such that $\mod \hat{V}\setminus \hat{U} > \epsilon_0/2$ and $$G\colon \hat{U} \rightarrow \hat{V}$$ is a real polynomial-like map of type $n$ that is infinitely renormalizable with combinatorics bounded by $p$. Then there exist polynomial-like maps of type $n$ 
$$G_i\colon \hat{U}^i \rightarrow \hat{V}^i$$
such that  
\begin{itemize}
\item[A.] we have $\mod \hat{V}^i\setminus  \hat{U}^i \geq \epsilon_0/2$ and $G_i \in \mathcal{B}_{nor}(U)$,
\item[B.] all critical points of  $G_i$  belong to the same periodic orbit,
\item[C.] we have 
$$\lim_i |G_i-G|_{\mathcal{B}_{nor}(U)}=0.$$
\end{itemize}
\end{prop}
\begin{proof} We will use the notation introduced in  \cite{sm2}. Let $\sigma=(\sigma_1,\sigma_2,\dots)$  be the combinatorics of $G$. By Proposition 2.2 in \cite{sm2}, there exists a sequence of polynomial $P_i$ of type $n$ with combinatorics $\sigma_i\star \dots \star \sigma_1$. By Corollary 2.3 in \cite{sm2} any accumulation point of this sequence is a polynomial $P$ of type $n$  that is infinitely renormalizable with combinatorics $\sigma$.  By the proof of Theorem 2 in \cite{sm2} there is only one polynomial of type $n$ with combinatorics $\sigma$, so the sequence $P_i$ indeed converges to $P$. Indeed there are now far more general rigidity results for polynomials. See Kozlovski, Shen and van Strien \cite{kss1}\cite{kss2}. 

Since $P_i$ is a convergent sequence of polynomials of type $n$ with connected Julia sets, it is possible to choose domains $\hat{U}^i$ and $\hat{V}^i$ such that 

\begin{itemize}
\item[-] $\inf_i \mod \hat{V}^i\setminus \hat{U}^i > 0$,
\item[-] $P_i\colon \hat{U}^i  \rightarrow \hat{V}^i$ is a polynomial-like map of type $n$.
\end{itemize}
and furthermore for some $K > 0$ there are $K$-quasiconformal maps 
$$\phi_i\colon  \mathbb{C}_n \rightarrow  \mathbb{C}_n $$
such that 
\begin{itemize}
\item[-] $\phi_i(\hat{U}^i)=\hat{U}$ and $\phi_i(\hat{V}^i)=\hat{V}$,
\item[-] $\phi_i(\overline{z}) = \overline{\phi_i(z)}$,
\item[-] $P_i\colon \hat{U}^i  \rightarrow \hat{V}^i$ is a polynomial-like map of type $n$,
\item[-]    $  G\circ \phi_i=  \phi_i\circ P_i $ on $\partial \hat{U}^i$.
\item[-] The sequence $\phi_i$ converges to a $K$-quasiconformal map $\phi$.
\item[-] If $\hat{U}^\infty=\phi^{-1}(\hat{U})$  and $\hat{V}^\infty=\phi^{-1}(\hat{V})$ then $P\colon  \hat{U}^\infty \rightarrow \hat{V}^\infty$ is a polynomial-like map of type $n$.
\end{itemize}
Let $\mu_i$ be the Beltrami field that coincides with $\mu_i=\overline{\partial} \phi_i/\partial \phi_i$ on $\mathbb{C}_n\setminus \hat{U}^i$,  that is invariant under $P_i$, and $\mu_i=0$ on $K(P_i)$.
Let $\psi_i\colon \mathbb{C}_n \rightarrow \mathbb{C}_n$ be the unique quasiconformal map such that  $\psi_i(-1,j)=(-1,j)$ and  $\psi_i(0,j)=(0,j)$ for every $j$,  and $\mu_i=\overline{\partial} \psi_i/\partial \psi_i$ on $\mathbb{C}_n$. Define
$$G_i =\psi_i\circ P_i  \circ   \psi_i^{-1}.$$
Then 
$$G_i\colon  \psi_i(\hat{U}^i)    \rightarrow \psi_i(\hat{V}^i)$$
is a polynomial-like map of type $n$. Note that 
$$\inf_i \mod \psi_i(\hat{V}^i) \setminus \psi_i(\hat{U}^i) > 0.$$
Every subsequence of $G_i$ has  a convergent subsequence. Let $F$ be one these  accumulation points. We claim that $F=G$. Note that every accumulation point is of the form $F= \psi\circ P\circ \psi^{-1}$, where $\psi$ is a $K$-quasiconformal map that is an accumulation point of the sequence $\psi_i$. We can assume, without loss of generality, that $\phi_i$ converges to a $K$-quasiconformal map $\phi$.  

Notice that $$\phi_i \circ    \psi^{-1}_i \circ G_i \circ \psi_i\circ    \phi_i^{-1}(z)= \phi_i \circ    P_i  \circ    \phi_i^{-1} (z)= G(z)$$
for $z \in \phi_i(\partial \hat{U}^i)=\partial \hat{U}$.
Taking the limit on $i$ we obtain
$$\phi \circ    \psi^{-1}\circ F \circ \psi\circ    \phi^{-1}(z)= G(z)$$
for $z \in \partial \hat{U}$. Moreover since $\psi_i\circ    \phi_i^{-1}$ is conformal in $\mathbb{C}_n\setminus \hat{U}$ we conclude that $\psi\circ  \phi^{-1}$ is conformal in $\mathbb{C}_n\setminus \hat{U}$. Since $F$ and $G$ are both infinitely renormalizable with the same combinatorics, one can use the Sullivan's pullback argument to conclude that there is quasiconformal conjugacy $H$ between $F\colon \hat{U}^\infty \rightarrow \hat{V}^\infty $ and $G\colon \hat{U} \rightarrow \hat{V}$ such that $H$ is conformal in $\mathbb{C}_n\setminus K(F)$. Since there are not invariant line fields supported of $K(F)$ we conclude that $H$ in conformal on $\mathbb{C}_n$, so $H$ is affine on each connected component of $\mathbb{C}_n$. Since $H(-1,j)=(-1,j)$ and $H(0,j)=(0,j)$ for every $j$, we conclude that $H$ is the identity.  So $F=G$ and $G_i$ converges to $G$. Itens A., B. and C. of Proposition \ref{lprop} follows easily from this. 
\end{proof}

\subsection{ Vertical directions, codimension of $E^h$ and vector bundles.}\label{vdir}

Let $f\colon V_1 \rightarrow V_2$ be a polynomial-like map. Let $\mathbb{B}_f$ be the vector space of the germs of holomorphic functions defined in a neighborhood of $K(f)$. We say that  $v \in \mathbb{B}_f$ is a  {\it vertical vector}  if  there exists a  holomorphic vector field  $\alpha$ defined on $\overline{\mathbb{C}}\setminus K(f)$ such that 
\begin{equation} \label{vertical} v(x) = \alpha\circ f(x) - Df(x)\alpha(x)\end{equation} 
for every $x$ close to $K(f)$ and in the domain of $\alpha$, and additionally

\begin{equation}\label{ninf} \lim_{z\rightarrow 0} z^2 \alpha(1/z)=0.\end{equation}

We have an analogous definition for polynomial-like extended maps of type $n$.  Denote the set of vertical directions of $f$ as $\hat{E}^v_f$. Recall that $v\in\mathbb{B}_f$ is a  {\it horizontal vector} ($v \in \hat{E}^h_f$) if there is  quasiconformal vector field on $\mathbb{C}$ such that (\ref{vertical}) holds in a neighborhood of $K(f)$ and $\overline{\partial}\alpha=0$ on $K(f)$. Lyubich \cite{lyu} proved that 
\begin{equation}\label{somadi}   \mathbb{B}_f=\hat{E}^h_f+\hat{E}^v_f. \end{equation} 
The same statement holds for polynomial-like extended maps of type $n$. Note that if $F \in \mathcal{W}$ has an extension that is  a real polynomial-like extended map of type $n$  and $F$  is either infinitely  renormalizable with bounded combinatorics in $\mathcal{C}_{p,n}$ or whose critical points belongs to the same periodic orbit then $E^h_F= \hat{E}^h_F\cap T\mathcal{B}_{nor}(U)$. Here $E^h_F$ is as defined in Section \ref{hdd}.

Due the infinitesimal pullback argument, if $f$ does not have invariant line fields on its Julia set $J(f)$ then $\hat{E}^h_f \cap  \hat{E}^v_f$  is exactly the space of vectors $v $ such that  there exists a vector field $\alpha(z)=az+b$ on $\overline{\mathbb{C}}$  that satisfies   (\ref{vertical})  in a neighborhood of $K(f)$. In particular if $F\in \Omega_{n,p}$ we have
$$T\mathcal{B}_{nor}(U)=E^h_F\oplus E^v_F. $$

\begin{prop} \label{v444} If $f\colon V_1 \rightarrow V_2$ is  a polynomial-like  of degree $d$ generated by the restriction of a polynomial of degree $d$ then $\hat{E}^v_f$ is exactly the linear space of polynomials of degree $d$. If $f$ is a polynomial-like extended map of type $n$ such that on each $\mathbb{C}\times\{i\}$ the map   $f$ coincides with a quadratic polynomial then $\hat{E}^v_f$ is the space of vectors that coincides with quadratic polynomials on each $\mathbb{C}\times\{i\}$. 
\end{prop}
\begin{proof} Suppose that $f$ is  a polynomial of degree $d$ and let $v \in \hat{E}^v_f$. Then the r.h.s. of (\ref{vertical}) implies that $v$ extends to an entire holomorphic function.  Of course (\ref{ninf}) implies that
$$|\alpha(y)|\leq C|y|$$
for  some $C$, provided $y\in \mathbb{C}$ has large modulus. Since $Df$ is a polynomial of degree $d-1$ it follows from  (\ref{vertical})  that
$$|v(x)|\leq \tilde{C} |x|^d, $$
for some $ \tilde{C} $, provided $|x|$ is large. So $v$ is a polynomial whose degree is at most $d$. On the other hand, if $v$ is a polynomial of degree at most $d$ we have that $f_t= f + tv$ is a polynomial of degree $d$ for every small $t$. Every $f_t$ have the very same external class (see Lyubich \cite{lyu}). This implies that $v=\partial_t f_t|_{t=0} \in \hat{E}^v_f$. The proof in the case of a polynomial-like extended map of type $n$ is analogous.
\end{proof}

 The following is similar to Lyubich \cite[Lemma 4.10]{lyu}, but for the sake of completeness we provide details. 

\begin{prop}\label{vee} Let $F\colon W \rightarrow V$ 
be a  polynomial-like  map of type $n$ with connected Julia set satisfying 
\begin{itemize}
\item[i.] $ 0< \epsilon_0 <  mod(V\setminus W)<   \epsilon_1,$ 
\item[ii.]  $diam \ K(F)\cap (\mathbb{C}\times \{i\})\geq 1$ for every $i$,
\item[iii.]   $diam \ V \cap (\mathbb{C}\times \{i\}) \leq \Cl{udaim}$  for every $i$.
\end{itemize}

Let $v\in E^v_{F}\cap \mathcal{B}(W)$. Consider the   holomorphic vector field 
$$\alpha\colon \overline{\mathbb{C}}_n\setminus K(F)\rightarrow \mathbb{C}$$ 
such that $ \lim_{z\rightarrow 0} z^2 \alpha(1/z)=0$ and
$$v(z)= \alpha(F(z))-DF(z)\alpha(z)$$
for every $z \in W\setminus K(F)$. Then there is $\Cl{u2} > 0$, that depends only on $\epsilon_0$, $\epsilon_1$  and $\Cr{udaim}$, such that 
$$|\alpha|_{sph\  \overline{\mathbb{C}}_n\setminus F^{-1}W}\leq \Cr{u2} |v|_{\mathcal{B}(W)}.$$
Here $|\cdot|_{sph\  Q}$ denotes the sup norm on $Q$ considering  the spherical metric on each component of $\mathbb{C}_n$. 
\end{prop}
\begin{proof} Define  $K_i(F)=K(F)\cap (\mathbb{C}\times\{i\})$, $1\leq 1\leq n$. Let 
$$\phi_i\colon (\overline{\mathbb{C}}\times \{i\})\setminus K_i(F)\rightarrow \mathbb{D}\times \{i\}$$
be  conformal maps such that $\phi_i(\infty,i)=(0,i)$.  Define the conformal maps 
$$\phi\colon \overline{\mathbb{C}}_n\setminus K(F)\rightarrow  \mathbb{D}\times\{1,\dots,n\}$$
as $\phi(z,i)=\phi_i(z,i).$ 
Define 
$$\tilde{W}=\phi(W\setminus K(F)) ,   \tilde{V}=\phi(V\setminus K(F))$$
and
$$\tilde{F}\colon \tilde{W}\rightarrow \tilde{V}$$
as $$\tilde{F}(z,i)= \phi\circ     F\circ \phi^{-1}(z,i).$$
Let $\psi(x,i)=(z/|z|^2,i)$, $\hat{W}=  \overline{\tilde{W}\cup \psi( \tilde{W})}$ and $\hat{V}=  \overline{\tilde{V}\cup \psi( \tilde{V})}$. Then $\tilde{F}$ has a analytic extension to a covering
$$\hat{F}\colon \hat{W} \rightarrow \hat{V}$$
satisfying $\hat{F}(\mathbb{S}\times\{1,\dots,n\})= \mathbb{S}\times\{1,\dots,n\}.$ Note that each connected component of $$\hat{V} \setminus  \hat{W}$$ is an annulus with modulus larger than $\epsilon_0$.  This implies that there is $k$ (that depends  only on  $\epsilon_0$,  $\epsilon_1$, $\Cr{udaim}$ and $n$) such that 
$$|D\hat{F}^{kn}(z,i)|\geq 2$$
for every $(z,i)\in \hat{F}^{-(kn+1)} \hat{W}.$ Note that  there is $\Cl{bdist} >1$, that depends  only on  $\epsilon_0$,  $\epsilon_1$, $\Cr{udaim}$, such that 
\begin{itemize}
\item[-]  $|D\phi(z,i)| \in [1/\Cr{bdist}, \Cr{bdist}]$ for every $(z,i)~\in~W\setminus F^{-(kn+1)} W$,
\item[-] $|DF(z,i)|\in [1/\Cr{bdist}, \Cr{bdist}]$ for every $(z,i)~\in~F^{-1}W\setminus F^{-(kn+1)} W$, 
\item[-]  $|D\hat{F}(z,i)| \in [1/\Cr{bdist}, \Cr{bdist}]$ for every $(z,i)~\in~\hat{F}^{-1}\hat{W}\setminus \hat{F}^{-(kn+1)}\hat{W}$.
\item[-] We have $1/\Cr{bdist}\leq |z|\leq \Cr{bdist}$ for every $z \in \partial F^{-1}W.$
\end{itemize}

Define $$\hat{\alpha}\colon  \mathbb{D}\times\{1,\dots,n\} \rightarrow \mathbb{C}$$
as
$$\hat{\alpha}(z,i)=  D\phi(\phi^{-1}(z,i)\alpha(\phi^{-1}(z,i))$$
and
$$\hat{v}(z,i)=  D\phi(F(\phi^{-1}(z,i))) v(\phi^{-1}(z,i)).$$
Then $\hat{\alpha}(0)=0$,
$$\hat{v}= \hat{\alpha}\circ \hat{F}- D  \hat{F}\cdot  \hat{\alpha},$$
and consequently
$$\sum_{j=0}^{nk} D\hat{F}^{nk-j}(\hat{F}^{j+1}(z,i))\hat{v}(\hat{F}^j(z,i))=  \hat{\alpha}\circ \hat{F}^{nk}(z,i)- D  \hat{F}^{nk}(z,i)\cdot  \hat{\alpha}(z,i).$$
for  $(z,i)\in \partial \hat{F}^{-(kn+1)} \hat{W}.$ Let $(z_0,i_0)$ be such that 
$$|\hat{\alpha}(z_0,i_0)|=\max_{(z,i)\in \partial \hat{F}^{-(kn+1)}\tilde{W}}|\hat{\alpha}(z,i)|.$$
Then
$$|\sum_{j=0}^{nk} D\hat{F}^{nk-j}(\hat{F}^{j+1}(z,i))\hat{v}(\hat{F}^j(z_0,i_0))|\geq  2\max_{(z,i)\in \partial \hat{F}^{-(kn+1)} \tilde{W}}|\hat{\alpha}(z,i)| - \max_{(z,i)\in \partial \hat{F}^{-1}\tilde{W}}|\hat{\alpha}(z,i)|.$$
Since $\hat{\alpha}$ is holomorphic  the maximum principle implies
$$\max_{(z,i)\in \partial \hat{F}^{-(kn+1)}\tilde{W}}|\hat{\alpha}(z,i)| \geq \max_{(z,i)\in \partial \hat{F}^{-1}\tilde{W}}|\hat{\alpha}(z,i)|,$$
so
$$\max_{(z,i)\in \partial \hat{F}^{-(kn+1)}\tilde{W}}|\hat{\alpha}(z,i)| \leq \Cr{2bdist} \sup_{(z,i)\in   \hat{F}^{-1}\tilde{W} \setminus \hat{F}^{-(kn+1)} \tilde{W} } |\hat{v}(z,i)|.$$
Here $\Cl{2bdist}= nk\Cr{bdist}^{nk}$. Consequently 
\begin{eqnarray*}
 \max_{(z,i)\in \partial F^{-1}W}|\alpha(z,i)| 
&\leq& \Cr{bdist}  \max_{(z,i)\in \partial \hat{F}^{-1}\tilde{W}}|\hat{\alpha}(z,i)| \\
&\leq& \Cr{bdist}  \max_{(z,i)\in \partial \hat{F}^{-(kn+1)}\tilde{W}}|\hat{\alpha}(z,i)| \\
&\leq& \Cr{bdist} \Cr{2bdist}   \max_{(z,i)\in   \hat{F}^{-1}\tilde{W} \setminus \hat{F}^{-(kn+1)} \tilde{W} } |\hat{v}(z,i)|\\
&\leq& \Cr{bdist}^2 \Cr{2bdist} \max_{(z,i)\in   F^{-1}W \setminus F^{-(kn+1)}W } |v(z,i)|\\
&\leq& \Cr{bdist}^2 \Cr{2bdist} \max_{(z,i)\in   F^{-1}W} |v(z,i)|.
\end{eqnarray*} 
Note that
\begin{eqnarray*}
\sup_{(z,i)\in  \overline{\mathbb{C}}_n\setminus F^{-1}W}|\alpha(z,i)|_{sph \ \overline{\mathbb{C}}}&=& \sup_{(z,i)\in  \overline{\mathbb{C}}_n\setminus F^{-1}W}\frac{2|\alpha(z,i)|}{1+|z|^2}\\
&\leq& \sup_{(z,i)\in  \overline{\mathbb{C}}_n\setminus F^{-1}W}\big|\frac{2\alpha(z,i)}{z^2}\big|\\
&\leq& \max_{(z,i)\in  \partial F^{-1}W}\big|\frac{2\alpha(z,i)}{z^2}\big|\\
&\leq& \Cr{bdist}^2 \max_{(z,i)\in \partial F^{-1}W}|\alpha(z,i)|.
\end{eqnarray*} 

\end{proof}

\begin{prop}[Codimension of $E^h_G$] \label{codimension} For every $G \in \Omega_{n,p}$ the codimension of $E^h_G$ is $n$. 
\end{prop}
\begin{proof} 
By Proposition \ref{lprop} for every  $G \in \Omega_{n,p}$ one can  find a polynomial-like extension $G\colon \tilde{U}\rightarrow \tilde{V}$ of type $n$  and a sequence of  polynomial-like maps $G_i \colon \hat{U}^i \rightarrow \hat{V}^i$ of type $n$  whose periodic points belongs to the same critical orbit and such that $U$ is compactly contained in $\hat{U}^i$ and $\hat{U}^i$ is compactly contained in $\tilde{U}$. Denote $E^j_{G_i}(\tilde{U})=\hat{E}^j_{G_i}\cap T\mathcal{B}_{nor}(\tilde{U})$, where $j\in\{v,h\}$. We claim that $codim \ E^h_{G_i}(\tilde{U})=n$. 

\noindent Indeed  for each $q \in C(G_i)$, let $m_q^i> 0$ be such that $G_i^{m_q^i}(q)\in C(G_i)$ and $G_i^{k}(q)\not\in C(G_i)$ for every $0< k < m_q^i$. Given $v\in T\mathcal{B}_{nor}(U)$, let $v_{q,k}=v(G_i^k(q))$.  Then there is a unique solution $\{  \alpha_{q,k}\}_{q \in C(G_i), 1\leq   k \leq m_q^i }$ for the homogeneous system of linear equations
$$v_{q,k}=   \alpha_{q,k+1}-        DG_i(G_i^k(q))\cdot \alpha_{q,k}, \ \   1\leq  k < m_q^i, \ q \in C(G_i),$$
 $$v_{q,0}=   \alpha_{q,1},\ q \in C(G_i).$$
\noindent In particular 
$$(v_{q,k})_{q\in C(G_i), 0\leq  k < m_q^i}\mapsto   (\alpha_{q,k})_{q\in C(G_i), 1\leq  k \leq  m_q^i}$$
is a linear bijection. Note that  $v \in E^h_{G_i}$ if and only if $\alpha_{q,m_q^i}=0$ for every $q \in C(G_i)$, that is, if and only if $v$ belongs to the kernel of the linear map
$$v\mapsto (\alpha_{q,m_q^i})_{q\in C(G_i)}.$$
Since
$$v\mapsto (v_{q,k})_{q\in C(G_i), 0\leq  k < m_q^i}$$
and
$$(v_{q,k})_{q\in C(G_i), 0\leq  k < m_q^i}\mapsto   (\alpha_{q,m_q^i})_{q\in C(G_i)}$$
are onto continuous  linear maps  it follows  that $codim \ E^h_{G_i}(\tilde{U})=n$ (Recall that  $G_i$ has $n$ critical points). So $dim \ E^v_{G_i}(\tilde{U})=n$. Since $\tilde{U}$ is compatible with $G$,  we have that \cite[Propositon 10.4]{sm4} implies $dim \ E^v_{G}(\tilde{U})=codim \ E^h_{G}(\tilde{U})=n$. 

Unfortunatelly $U$ is not  compatible with $G$ (the repelling fixed point $-1$ of $G$ does not belong to the interior of $U$), so we need to be a little more careful to conclude that $codim \ E^h_{G}=n$. Consider  the natural affine inclusion
$$\pi\colon \mathcal{B}_{nor}(\tilde{U})\rightarrow \mathcal{B}_{nor}(U)$$
Then $D \pi$ is continuous, injective map, it has dense image and moreover
$$(D \pi)^{-1} E^h_{G}= E^h_{G}(\tilde{U}).$$
Note that  if $codim \ E^h_{G} \geq  k$ then there is a  bounded linear onto map  $\psi\colon T_G\mathcal{B}_{nor}(U)\rightarrow \mathbb{R}^{k}$ such that $E^h_{G}\subset Ker \ \psi$. Since $D \pi$ have dense image we have that $\psi\circ  D \pi$  is also a  bounded linear onto map such that $E^h_{G}(\tilde{U})\subset Ker \ \psi\circ  D \pi$, so $codim \ E^h_{G}(\tilde{U}) \geq  codim \ Ker \ \psi\circ D\pi=k$. So $codim \ E^h_{G} \leq codim \ E^h_{G}(\tilde{U})=n.$

On the other hand, since $\pi$ is injective we have that $\dim \pi(E^v_{G}(\tilde{U}))=n$. If $v \in \pi(E^v_{G}(\tilde{U}))\cap E^h_G$ then $v \in E^v_{G}(\tilde{U})\cap E^h_G(\tilde{U})=\{0\}$. So $codim \ E^h_{G}\geq n$. 
\end{proof}
The following lemma is elementary. We included it here for the sake of completeness.

\begin{lem} \label{f1}Let $(\mathcal{B}_i, |\cdot|_i)$, $i=1,2$,  be  Banach spaces. Let $\Omega$ be a compact metric space and suppose that for every $f \in \Omega$ we associate vector subspaces $E^v_f\subset \mathcal{B}_2\subset \mathcal{B}_1$, $E^h_{f,i}\subset \mathcal{B}_i$ satisfying
\begin{itemize}
\item[A.] For every $f \in \Omega$ we have $\mathcal{B}_2=E^v_f\oplus E^h_{f,2}$ and $E^v_f\cap E^h_{f,1}=\{0\}$.
\item[B.] The set
$$\{(f,v)\colon \ f \in \Omega, \ v \in E^h_{f,i}  \}$$ 
is a  closed subset of $\Omega\times \mathcal{B}_i$.
\item[C.]  We have that 
$$\{(f,v)\colon \ f \in \Omega, \ v \in E^v_f, \ |v|_i\leq 1   \}$$
is a compact subset of $\Omega\times \mathcal{B}_i$, $i=1,2$.
\item[D.]  There exists $n\in \mathbb{N}$ such that $\dim E^v_f=n$ for every $f \in \Omega$.
\item[E.] The inclusion $\imath\colon \mathcal{B}_2 \rightarrow \mathcal{B}_1$ is a compact linear operator and $$\imath^{-1}(E^h_{f,1})=E^h_{f,2}.$$
\end{itemize}
Then
\begin{itemize}
\item[I.] The set 
$$E^v=\{(f,v)\colon \ f \in \Omega, \ v \in E^v_f \},$$
with the topology induced by $\Omega\times \mathcal{B}_2$, is a topological vector bundle (with the obvious linear structure on the fibers $E^v_f$) with fibers of dimension $n$.
\item[II.] Let $\sim$ be the  equivalent relation on $\Omega\times \mathcal{B}_2$ defined by $(f,v)\sim(g,w)$ if and only if $f=g$ and $v-w\in E^h_f$. Then the  quotient topological space 
$$E = \{ (f,[v]) \ s.t. \ f \in \Omega \  and \  [v] \in \mathcal{B}/E^h_{f,2} \}.$$
is a topological vector bundle with fibers of dimension $n$.
\item[III.]  Define 
$$|(f,[v])| = dist_{\mathcal{B}_2}(v, E^h_{f,2})=\inf \{  |v-w|_2 \colon \ w \in E^h_{f,2}\}.$$
then 
$$(f,[v]) \in E\mapsto |(f,[v])|_2$$
is continuous.  
\end{itemize}
\end{lem}
\begin{proof}[Proof of I] All limits  in the proof of I.  and II are in the topology of $(\mathcal{B}_2, |\cdot|_2)$. Let $u \in \mathcal{B}_2$.   Then for every $g\in \Omega$ there are  unique vectors  $v^{g,u}\in E^v_g$ and $w^{g,u}\in E^h_{g,2}$ such that
$$u = v^{g,u}+w^{g,u}.$$
 First note that 
\begin{equation}\label{rf} \sup \ \{   |v^{g,u}|_2, g \in \Omega, \ |u|_2\leq 1  \} \cup \{   |w^{g,u}|_2, g \in \Omega,\ |u|_2\leq 1  \}  < \infty.\end{equation}
Otherwise there is a sequence $g_i \in \Omega$, $|u_i|\leq 1$ such that $$r_i=\max \{ |v^{g_i,u_i}|_2, |w^{g_i,u_i}|_2\}\rightarrow_i  \infty.$$  Since $\Omega$ is compact, without loss of generality we can assume that 
$$\lim_i g_i=g\in \Omega$$
and C. implies that we can assume 
$$\lim_i \frac{v^{g_i,u_i}}{r_i} =v \in E^v_g.$$
and consequently
\begin{equation}\label{limh} \lim_i \frac{w^{g_i,u_i}}{r_i} =-v.\end{equation}
 On the other hand by (\ref{limh}) and B. we have $v \in E^h_g$, so by A. we conclude $v=0$. This is a contradiction with the definition of $r_i$. So (\ref{rf}) holds. 
We claim that  the map
$$S\colon \Omega \times \mathcal{B}_2  \rightarrow \Omega \times \mathcal{B}_2 $$
defined by 
$$S(g,u)=(g,v^{g,u}) \in \mathcal{B}_2$$
is a continuous linear map. Indeed suppose $\lim_i g_i=g$ and $\lim_i u_i=u$. By assumption C. and (\ref{rf}), taking a subsequence  we may assume that 
$$\lim_i v^{g_i,u_i}= v\in E^v_g$$
and consequently by B.
$$\lim_i w^{g_i,u_i}= \lim_i u_i - \lim_i v^{g_i,u_i}= u-v \in E^h_{g,2}$$
By A. we conclude that $v=v^{g,u}$ and $u-v=w^{g,u}$. This proves the claim.
Let $f\in \Omega$ and choose a basis $v_1, \dots, v_n \in E^v_f$. Let $v_i^g=S(g,v_i) \in E^v_{g}$ and $w_i^g=v_i-S(g,v_i) \in E^h_{g,2}$, with $g\in \Omega$. Of course $$v_i = v_i^g+ w^g_i.$$
So for every $i$ we have that $g\mapsto v_i^g$ is continuous and moreover $v_i^f=v_i$ and $w_i^f=0$. In particular there exists an open  neighborhood $O_1$ of $f$ in $\Omega$  such that $\{v_i^g\}_i$ is a basis of $E^v_g$ for every $g \in O$.  

Reducing the neighborhood $O$, we may assume that $E^h_{g,2}\cap E^v_f=\{0\}$ for every $g \in O$. Otherwise it would exists a sequence $g_i\rightarrow_i f$ 
and $w_i \in E^h_{g_i,2}\cap E^v_f$ satisfying $|w_i|=1$. By D. we may assume $\lim_i w_i=w \in E^v_f$. By B. we have $w \in E^h_{f,2}$, which contradicts A.

Define the map
$$H\colon O\times E^v_f \rightarrow \{(g,v)\colon \ g \in O, \ v \in E^v_g\}$$
as 
$$H(g,\sum_i c_i v_i)= (g, \sum_i c_i v^g_i).$$
We have that $H$ is a continuous and bijective map. Note that  
$$H^{-1}(g,u)= (g,v),$$
where $v$ is the only vector in $E^h_{f,2}$ such that $v-u \in E^h_{g,2}$. 

We claim that $H^{-1}$ is continuous. Indeed, suppose that $\lim_i (g_i,u_i)=(g,u)$, with$g, g_i\in O$ and $u,u_i\in E^v_{g_i}$. If $H^{-1}(g_i,u_i)=(g_i,v_i)$ then $w_i=v_i-u_i\in E^h_{g_i,2}$ and $v_i \in E^v_f$. Let $r_i=\max \{ |v_i|_2,   |w_i|_2 \}$. We claim that 
\begin{equation}
\label{supri} \sup_i r_i < \infty.
\end{equation}
Otherwise without loss of generality we may assume $\lim_i r_i=\infty$.  By D. we may assume that 
$\lim_i v_i/r_i = \tilde{v}\in E^v_f$ and consequently $\lim_i w_i/r_i=\tilde{v}.$ By B. we have $\tilde{v}\in E^h_{g,2}$. So $\tilde{v}=0$, in contradiction with the definition of $r_i$. This proves (\ref{supri}). In particular without loss of generality we can  assume $\lim_i v_i=\hat{v}\in E^v_f$ and consequently by B. we have $\lim_i w_i = \hat{w}=\hat{v}-u\in E^h_{f,2}$. In particular $H^{-1}(g,u)=(g,\hat{v}).$ So $H^{-1}$ is continuous. So
$$\hat{H}\colon  O\times \mathbb{R}^n   \rightarrow    \{(g,v)\colon \ g \in O, \ v \in E^v_g\}   $$
defined by 
$$\hat{H}(g,(c_i)_i)= (g, \sum_i c_i v^g_i)$$
is a local trivialization of the vector bundle $E^v$  in the open set  $\{(g,v)\colon \ g \in O, \ v \in E^v_g\}$.
\end{proof}
\begin{proof}[Proof of II]  Let $\pi\colon \Omega\times \mathcal{B}_2\rightarrow E$ be a natural projection 
$$(f,v) \mapsto (f,[v]_f),$$
where $[v]_f$ represents the equivalent class of $v$ in $\mathcal{B}/E^h_{f,2}$. We will define a  homeomorphism 
$$T\colon E^v \rightarrow E$$
that is a vector bundle homeomorphism. Indeed let $T$ be  the restriction of $\pi$ to $E^v$. Of course $T$ is a continuous  map that preserves the linear structure in the fibers. It is also a bijection, since $T(f,v)=T(g,w)$ implies $f=g$, with $v, w\in E^v_f$ and $v-w \in E^h_{f,2}$, so by A. we have $v=w$. Note that $T^{-1}(f,[u]_f)=(f,v)$, where $v$ is the unique vector that satisfies $v \in E^v_f$ and $u-v\in E^h_{f,2}$. Note that $T^{-1}(f,[u]_f)=S(f,u)$,  where $S$ was defined in the proof of I. Consequently $T^{-1}$ is continuous, since $S$ descends to the quotient space $E$ as $T^{-1}$.
\end{proof}

\begin{proof}[Proof of III]  We claim that 
\begin{equation} \label{hjk} (f,v) \in \Omega\times \mathcal{B}_2 \rightarrow dist_{\mathcal{B}_2}(v,E^h_f)\end{equation}
is continuous. Indeed suppose that  $\lim_k (f_k,v_k)=(f,v)$.  We have
$$|dist_{\mathcal{B}_2}(v_k,E^h_{f_k,2})-dist_{\mathcal{B}_2}(v,E^h_{f_k,2})|\leq dist_{\mathcal{B}_2}(v_k-v,E^h_{f_k,2})\leq |v-v_k|_2\rightarrow_k 0,$$
in particular 
$$\lim_k dist_{\mathcal{B}_2}(v_k,E^h_{f_k,2})-dist_{\mathcal{B}_2}(v,E^h_{f_k,2}) =0,$$
so to prove the claim it is enough to show that
\begin{equation}\label{limfk} \lim_k dist_{\mathcal{B}_2}(v,E^h_{f_k,2})=dist_{\mathcal{B}_2}(v,E^h_{f,2}).\end{equation}
Fix $\epsilon > 0$ and let $w \in E^h_{f,2}$ be such that 
 $$ |v-w|_2 <  dist_{\mathcal{B}_2}(v,E^h_{f,2})+\epsilon.$$
Then $\lim_k S(f_k,w)=0$, where $S$ is as defined in the proof of I. In particular $w_k=w - S(f_k,w)\in E^h_{f_k,2}$ and for large $k$ we have $|v-w_k|< dist_{\mathcal{B}_2}(v,E^h_{f,2})+2\epsilon$.  Since $\epsilon > 0$ is arbitrary 
$$\limsup_k  dist_{\mathcal{B}_2}(v,E^h_{f_k,2})\leq dist_{\mathcal{B}_2}(v,E^h_{f,2}).$$
On the other hand, if 
$$\liminf_k  dist_{\mathcal{B}_2}(v,E^h_{f_k,2})\leq  dist_{\mathcal{B}_2}(v,E^h_{f,2})-2\epsilon$$
then we can assume (taking a subsequence)  that there is $w_k \in E^h_{f_k,2}$ such that 
$$|v-w_k|_2\leq dist_{\mathcal{B}_2}(v,E^h_{f,2})-\epsilon.$$

Let $u_k= S(f,w_k)\in E^v_{f,2}$. Note that $\sup_k |w_k|_2 < \infty$, which implies $\sup_k |u_k|_2 < \infty$ and  consequently $y_k=w_k~-u_k\in~E^h_{f,2}$ satisfies 
$$\sup_k |y_k|_2 < \infty.$$   
By B. and E. we can find $y \in E^h_{f,1}$ and a subsequence of $y_k$ such that $y_k$ converges to $y$ in $\mathcal{B}_1$. Taking a subsequence we can assume that $\lim_k u_k =u\in E^v_f$ (in the topologies of $\mathcal{B}_i$, $i=1,2$).  So $w_k$ converges to $u+y$ in $\mathcal{B}_1$. By A., B. and E. we have $u=0$.  We conclude that
$$dist_{\mathcal{B}_2}(v,E^h_{f,2})\leq \liminf_k |v-y_k|_2\leq dist_{\mathcal{B}_2}(v,E^h_{f,2})-\epsilon,$$
which is a contradiction. So (\ref{limfk}) holds. This proves the claim.  Since $(f,v)\sim(g,\tilde{v})$ implies  $dist_{\mathcal{B}_2}(v,E^h_f)=dist_{\mathcal{B}_2}(\tilde{v},E^h_g)$ the function (\ref{hjk}) descends to the quotient topological space $E$ as a continuous function.

\end{proof}

Proposition \ref{closed} implies that 
$$\{(F,v), \ F \in \Omega_{n,p} \ and \ v \in E^h_F\}$$
is a closed subset of $\Omega_{n,p} \times T\mathcal{B}_{nor}(U)$. We also have
\begin{lem}\label{f2} The set
\begin{equation} \label{cofv} \mathcal{E}=\{(F,v), \ F \in \Omega_{n,p} \ and \ v \in E^v_F, \ |v|_{T\mathcal{B}_{nor}(U)} \leq 1\}\end{equation}
is a compact  subset of $\Omega_{n,p} \times T\mathcal{B}_{nor}(U)$.
\end{lem}
\begin{proof}  Let $(F_k,v_k)$ be a sequence in the set $\mathcal{E}$.  Due the complex bounds there exist  domains $W_k, V_k \in \mathbb{C}_n$ such that $\overline{W_k}\subset V_k$
and 
$$F_k\colon W_k \rightarrow V_k$$ 
are polynomial-like  maps of type $n$ satisfying $$mod(V_k\setminus W_k)\geq \epsilon_0.$$
Using the same argument as McMullen \cite[Theorem 5.8]{mc1}  there is a polynomial-like  map of type $n$  $F\colon W\rightarrow V$ with $mod(V\setminus W)\geq \epsilon_0$ such that $\lim_k F_k=F$ in the topology defined by McMullen. In particular $\lim_k W_k =W$ and $\lim_k V_k=V$  in the Carath\'eodory topology and $F_k$ converges to $F$ uniformly on compact subsets of $W$. Consequently $\lim_k F_k=F$ in $\mathcal{B}_{nor}(U)$ and we can find $\tilde{V}$ compactly contained in $V$ such that 
$$F_k\colon F_k^{-1}\tilde{V} \rightarrow \tilde{V}$$ 
are polynomial-like  maps of type $n$ satisfying 
$$mod( \tilde{V}\setminus F_k^{-1}\tilde{V} )\geq \epsilon_0/2.$$
Let $\tilde{W}=F^{-1}\tilde{V}$.
\noindent Since $v_k\in E^v_{F_k}$ there  exist holomorphic vectors fields 
$$\alpha_k\colon \overline{\mathbb{C}}_n\setminus K(F_k)\rightarrow \mathbb{C}$$ 
such that $\alpha_k(\infty)=0$ and
$$v_k(z)= \alpha_k(F_k(z))-DF_k(z)\alpha_k(z)$$
for every $z \in W_k\setminus K(F_k)$.

Finally, for every large  $j > 0$ it is possible to find a domain $V^j$ such that 
$$K(F) \subset V^j \subset \overline{\{z \in \mathbb{C}_n\colon \ dist(z,K(F)) < 1/j \}} \subset \tilde{V}$$
and
$$F\colon F^{-1}V^j \rightarrow V^j$$
is a polynomial-like  map of type $n$. Consequently there is $k_0=k_0(j)$ such that  for every $k\geq k_0$ we have that 
$$F_k\colon F^{-1}_kV^j \rightarrow V^j$$
is a polynomial-like  map of type $n$. Note that $F^{-1}_kV^j \subset \tilde{W}$ for large $k$. Due Proposition \ref{vee} we have that 
$$|\alpha_k|_{sph\  \overline{\mathbb{C}}_n\setminus F_k^{-2}V^j}\leq C_j |v_k|_{\mathcal{B}(V^j)}\leq C_j |v_k|_{\mathcal{B}(\tilde{W})}.$$
for large $k$. We claim that
\begin{equation}\label{limvk} \sup_k |v_k|_{\mathcal{B}(\tilde{W})} < \infty.\end{equation}
Indeed, otherwise  we may assume that  $r_k=|v_k|_{\mathcal{B}(\tilde{W})}$ diverges to infinity.  Then $\hat{\alpha}_k=\alpha_k /r_k$ satisfies
$$|\hat{\alpha}_k|_{sph\  \overline{\mathbb{C}}_n\setminus F_k^{-2}V^j}\leq C_j $$
This implies that a subsequence of  $\hat{\alpha}_k$ converges uniformly on compact subsets  of $\mathbb{C}_n\setminus K(F)$  to a holomorphic vector field $\hat{\alpha}\colon \overline{\mathbb{C}}_n\setminus K(F)\rightarrow \mathbb{C}$ and the  corresponding subsequence of  $\hat{v}_k = v_k/r_k$ converges uniformly on compact subsets of $W\setminus K(F)$ to
$$\hat{v}=\hat{\alpha}\circ F- DF\cdot \hat{\alpha}.$$
By the maximum principle we have that $\hat{v}_k$ is a uniform Cauchy sequence on  compact subsets of  $W$, so $\hat{v}$ extends to a holomorphic function on $W$  and $\lim_k \hat{v}_k=\hat{v}$ uniformly on compact subsets of $W$. Since $|\hat{v}_k|_{\mathcal{B}(\tilde{W})}=1$ for every $k$ we have $|\hat{v}|_{\mathcal{B}(\tilde{W})}=1$. On the other hand $$|\hat{v}|_{\mathcal{B}(U)}=\lim_k |\hat{v}_k|_{\mathcal{B}(U)}=\lim_k 1/r_k=0,$$
so $\hat{v}=0$ everywhere, in contradiction with $|\hat{v}|_{\mathcal{B}(\tilde{W})}=1$. This proves the claim.

\noindent Now we can use the same argument as in  the previous paragraph to conclude that there is a subsequence of $\alpha_k$ that converges uniformly on compact subsets  of $\mathbb{C}_n\setminus K(F)$  to a holomorphic vector field $\alpha\colon \overline{\mathbb{C}}_n\setminus K(F)\rightarrow \mathbb{C}$ and the  corresponding subsequence of  $v_k$ converges in $\mathcal{B}(\tilde{W})$ (and in particular in $\mathcal{B}(U)$) so a vector $v$ that satisfies 
$$v= \alpha \circ F - Df\cdot \alpha$$
on $W\setminus K(F)$. This concludes the proof.
\end{proof}

As an immediate consequence of Lemmas \ref{f1} and \ref{f2} we have

\begin{prop} \label{qts} The quotient topological space 
$$E = \{ (F,[v]) \ s.t. \ F \in \Omega_{n,p} \  and \  [v] \in T\mathcal{B}_{nor}(U)/E^h_F \}.$$
is a topological vector bundle with fibers of dimension $n$. Moreover
\begin{equation} \label{nomvb}(F,[v])\mapsto |(F,[v])|=dist_{T\mathcal{B}_{nor}(U)}(v, E^h_F)\end{equation}
is a continuous function. 
\end{prop}
\begin{proof} Let $\hat{U}$ be a symmetric  domain with respect to the real trace of $\mathbb{C}_n$, that is compactly contained  in the interior of $U$, such that the $\hat{U}\cap \mathbb{R}$ is an interval that contain in its interior the convex closure of the postcritical set of every $F \in \Omega_{n,p}$. Let $\mathcal{B}_2= T\mathcal{B}_{nor}(U)$, $\mathcal{B}_1 = \mathcal{B}(\hat{U})$, $E^h_{F,2}= E^h_F$ and $E^h_{F,1}$ be the set of horizontal vectors of $\mathcal{B}(\hat{U})$. Apply Lemma \ref{f1}. 
\end{proof}

\section{Hyperbolicity of the $\omega$-limit set $\Omega_{n,p}$ of $\mathcal{R}$.}\label{hypsection}

Given  $F \in \Omega_{n,p}$, denote 
$$\mathfrak{B}_+(F)=\{ v \in   T_F\mathcal{B}_{nor}(U) \ s.t.  \ \sup_{i \in \mathbb{N}} |D\mathcal{R}^i_f\cdot v_i| < \infty\}.$$
Recall  we choose $U=D_{\delta_0,\theta_0}.$ The goal of this section is to prove 

\begin{prop}\label{key} Suppose that for every $F \in \Omega_{n,p}$ we have 
\begin{equation}\label{bounded}
\mathfrak{B}_+(F) \subset  E^h_F.
\end{equation}
Then $\Omega_{n,p}$ is a hyperbolic set. Moreover its stable direction is exactly $E^h$. 
\end{prop}

Proposition \ref{key} reduces the study of the hyperbolicity of $\Omega_{n,p}$ to the study of the existence and regularity of the solutions $\alpha$ of the  cohomological equation (\ref{tce3}).   So to show that $\Omega_{n,p}$ is a hyperbolic set  it remains to prove 

\begin{thm}[Key Lemma] \label{keylem} If $F \in \Omega_{n,p}$ then 
\begin{equation}\label{boundedl}
\mathfrak{B}_+(F) \subset  E^h_F.
\end{equation}
\end{thm}

We will prove Theorem \ref{keylem} in Section \ref{keysec}. As an immediate consequence of Proposition \ref{key} and Theorem \ref{keylem} we have 

\begin{thm}[Theorem B]\label{hip1}  $\Omega_{n,p}$ is a hyperbolic set. Moreover its stable direction is exactly $E^h$. 
\end{thm}

\subsection{ A criterium for hyperbolicity of cocycles} Let $E$ be a topological vector bundle with base $\Omega$ and fiber $\mathbb{R}^n$ and projection $p\colon E \rightarrow \Omega$.   We denote elements of $E$ by $(x,v)$, where $x \in \Omega$  and  $v \in p^{-1}(x)$. We also assume that $\Omega$ is compact.  Additionally, assume that $|\cdot|$ is a continuous function 
$$(x,v)\in E\mapsto |(x,v)|\in \mathbb{R}$$
so  that $|\cdot|$ is a norm on each fiber   $p^{-1}(x)$. We will abuse the notation writing $|v|$ instead of $(x,v)$.

Let $L\colon E \rightarrow  E$ be a fiber-preserving homeomorphism that is linear on the fibers.  The map $L$ is called a linear cocycle on $E$.  Define

$$\mathfrak{B}_+=\{ (x,v) \in E \ s.t.  \sup_{i \in \mathbb{N}} |v_i| < \infty, \ where \ L^i(x,v)=(x_i,v_i)\}.$$
$$\mathfrak{B}=\{ (x,v) \in E \ s.t.  \sup_{i \in \mathbb{Z}} |v_i| < \infty, \ where \ L^i(x,v)=(x_i,v_i)\}.$$
$$\mathcal{S}=\{ (x,v) \in E \ s.t. \lim_{i\rightarrow +\infty}  |v_i|=0, \ where \ L^i(x,v)=(x_i,v_i)\}.$$
$$\mathcal{U}=\{ (x,v) \in E \ s.t. \lim_{i\rightarrow -\infty}  |v_i|=0, \ where \ L^i(x,v)=(x_i,v_i)\}.$$ 
and the zero section

$$E_0  = \{ (x,0) \in E\}.$$

We  say that the cocycle $L$ is uniformly expanding  if   there exist $K > 0$ and $\theta > 1$ such that for every $(x,v) \in E$ we have
\begin{equation} \label{ue} |v_i|\geq K \theta^i |v|\end{equation}
for every $i\geq 0$, where $L^i(x,v)=(x_n,v_n)$. 

\begin{prop} \label{crit}The cocycle $L\colon E \rightarrow E$ is  uniformly  expanding  if and only if 
\begin{equation} \label{bb} \mathfrak{B}_+=E_0.\end{equation} 
\end{prop} 
\begin{proof} Of course if $L\colon E \rightarrow E$  is uniformly expanding then (\ref{bb}) holds. To prove the reverse implication, note that (\ref{bb}) implies $\mathcal{S}=\mathfrak{B}=E_0$.  By Theorem 2 in Sacker and Sell \cite{ss} (see also Section 7 there) we have that $L\colon E \rightarrow E$ is uniformly expanding. 
\end{proof}

One can also prove Proposition \ref{crit} applying  Sacker and Sell's results in \cite[Lemma 9 and Theorem 2]{ss2}. We just refer to that because the proof of these results in \cite{ss2} seems to be more elementary than the proof of Theorem 2 in \cite{ss}.

\subsection{Unstable invariant cones} Let $F \in \Omega_{n,p}$, Denote 
$$\mathfrak{B}_+(F)=\{ v \in   T_F\mathcal{B}_{nor}(U) \ s.t.  \ \sup_{i \in \mathbb{N}} |D\mathcal{R}^i_f\cdot v_i| < \infty\}.$$

Consider the topological vector bundle $E$  (see Proposition \ref{qts})
$$E = \{ (F,[v]) \ s.t. \ F \in \Omega_{n,p} \  and \  [v] \in T_F\mathcal{B}_{nor}(U)/E^h_F \}.$$
with  the continuous function (\ref{nomvb}) that restriced to each  $T_F\mathcal{B}_{nor}(U)/E^h_F$ is the usual quotient norm. 
Recall that due Proposition \ref{qts} we have that 
$$\dim T_F\mathcal{B}_{nor}(U)/E^h_F =n.$$
By Proposition \ref{invariance} the linear transformation
 \begin{equation}\label{derr} D\mathcal{R}_F \colon T_F\mathcal{B}_{nor}(U) \rightarrow T_{\mathcal{R}F} \mathcal{B}_{nor}(U)\end{equation}
 induces a bounded linear transformation 
 $$L_F \colon T_F\mathcal{B}_{nor}/E^h_F \rightarrow T_{\mathcal{R}F} \mathcal{B}_{nor}(U)/E^h_{\mathcal{R}F}.$$
 \begin{lem} The map 
$$L(F,v) = (\mathcal{R}F, L_F\cdot v)$$
is a vector bundle isomorphism in the vector bundle $E$, that is, it is a homeomorphism of $E$ onto itself that preserves the linear structure on the fibers.
 \end{lem}
 \begin{proof} Let $\pi\colon \Omega_{n,p} \times T\mathcal{B}_{nor}(U)\rightarrow E$ be a natural projection 
$$(F,v) \mapsto (F,[v]_F),$$
where $[v]_F$ represents the equivalent class of $v$ in $T\mathcal{B}_{nor}(U)/E^h_F$. Of course
$$\tilde{L}\colon \Omega_{n,p} \times T\mathcal{B}_{nor}(U) \rightarrow  E$$
defined by $\tilde{L}(F,v)= \pi(\mathcal{R}F,D\mathcal{R}_F\cdot v)$ is continuous. Then  by Proposition \ref{invariance} the map $\tilde{L}$ descends to the topological quotient space $E$ as the continuous map $L$.

By Theorem  \ref{dense} we have that the linear transformation (\ref{derr}) has dense image in $T_{\mathcal{R}F} \mathcal{B}_{nor}(U)$. This implies that  for every $F$ the linear map $L_F$ is invertible. By Corollary \ref{sh} we have that $\mathcal{R}\colon \Omega_{n,p}  \rightarrow \Omega_{n,p} $ is a homeomorphism. We conclude that $L$ is invertible. It remains to show that its inverse is continuous.  Since 
$$E^1=\{ (F,[v]_F) \colon  F \in \Omega_{n,p} \ and \ |[v]_F|=1\}$$
is compact, $L$ is invertible, and the function 
$$\psi\colon E^1 \rightarrow  \mathbb{R}^+$$
defined  by $\psi(F,[v]_F)=|L(F,[v]_F)|$  is continuous, we have that 
\begin{equation}\label{infc}C=\min_{(F,[v]_F)\in E^1} \psi(F,[v]_F) > 0.\end{equation}
So suppose $\lim_k (F_k,[v_k]_{F_k})= (F,[v]_F)$ and 
$$L^{-1}(F_k,[v_k]_{F_k})=(\mathcal{R}^{-1}F_k,[w_k]_{\mathcal{R}^{-1}F_k})$$
 Then $\lim_k \mathcal{R}^{-1}F_k= \mathcal{R}^{-1}F$  and by (\ref{infc}) we have $\sup_k |[w_k]_{\mathcal{R}^{-1}F_k}|\leq \hat{C}$ for some constant $\hat{C}$. Taking a subsequence we can assume that $\lim_k [w_k]_{\mathcal{R}^{-1}F_k}= [\tilde{w}]_{\mathcal{R}^{-1}F}$. Since $L$ is continuous $$(F,[v]_{F})=\lim_k L(\mathcal{R}^{-1}F_k,[w_k]_{\mathcal{R}^{-1}F_k})=(F,L_F[\tilde{w}]_{\mathcal{R}^{-1}F}).$$
 From the injectivity of $L$ we conclude that $[\tilde{w}]_{\mathcal{R}^{-1}F}=L^{-1}_F[v]_{F}$. So $L^{-1}$ is continuous.
\end{proof} 

\begin{prop} Suppose that for every $F \in \Omega_{n,p}$ we have 
\begin{equation}
\mathfrak{B}_+(F) \subset  E^h_F.
\end{equation}
Then the cocycle $L$ is uniformly expanding, that is, there is $C > 0$ and $\Cl[e]{eq} > 1$ such that for every $v \in \mathcal{B}_{nor}(U)$ and $F \in \Omega_{n,p} $ we have 
\begin{equation}\label{expanding_f}  d(D\mathcal{R}^i_{F}\cdot v, E^h_{\mathcal{R}^i F})\geq C{\Cr{eq}}^i d( v, E^h_{F}).\end{equation}\end{prop}
\begin{proof}   Indeed, suppose that $[v] \in T_F\mathcal{B}_{nor}/E^h_F$ satisfies
$$ \sup_i |L^i_F \cdot [v]| < \infty.$$
By Propostion \ref{crit}, it is enough  to show that $[v]=0$, that is, $v \in E^h_F$. Firstly note that 
$D\mathcal{R}^i_F\cdot v =     u_i + w_i$, where  $\sup_i |u_i|=C < \infty$ and $w_i \in E^h_{\mathcal{R}^iF}$. Note that 
$$D\mathcal{R}_{\mathcal{R}^iF}\cdot (u_i + w_i) = u_{i+1} + w_{i+1},$$
so 
$$w_{i+1}= D\mathcal{R}_{\mathcal{R}^iF}\cdot u_i - u_{i+1} + D\mathcal{R}_{\mathcal{R}^iF}\cdot w_i.$$
So
$$w_{i+j}= D\mathcal{R}^j_{\mathcal{R}^iF}\cdot u_i - u_{i+j} + D\mathcal{R}^j_{\mathcal{R}^iF}\cdot w_i,$$
in particular
$$|w_{i+j}| \leq C(1+K\theta^{-j}) + K \theta^{-j} |w_i|,$$
where $K$ and $\theta >  1 $ are as in Proposition \ref{contraction}. This implies that $\sup_i |w_i| < \infty$ and consequently 
$$\sup_i | D\mathcal{R}^i_F\cdot v | < \infty.$$ 
By (\ref{boundedl}) we have that $v \in E^h_F$.  So $L$ is uniformly expanding.\end{proof}

Let $C > 0$ and $\theta > 1$ be as in (\ref{expanding_f}). Choose $j_0 > 0$ such that 
$$ C  \theta^{j_0} > 1.$$
If  $\epsilon > 0 $ is small enough we have  that 
$$\tilde{\theta}= C e^{-\epsilon}  \theta^{j_0} > 1.$$
Denote
$$\tilde{C} = \sup_{F \in \Omega_{n,p}} |D\mathcal{R}_F^{j_0}|.$$
Define the cone $C^u_F(K)$ as the set of all $v \in T\mathcal{B}_{nor}(U)$ that can be written as $v= u + w$,  where
\begin{itemize}
\item[A.]  $|u|\leq e^{\epsilon} d( v, E^h_{F})$, 
\item[B.]  $w \in E^h_F$  and 
\item[C.] $|w| \leq K |u|$. 
\end{itemize}
Note that
\begin{equation} \label{cobre}  \bigcup_{K > 0} C^u_F(K) = \big(T\mathcal{B}_{nor}(U)\setminus E^h_F\big) \cup \{ 0\}.\end{equation} 

Our goal is to show that if (\ref{expanding_f}) holds then there is $K > 0$ such that 
$$F\in \Omega_{n,p}\mapsto C^u_F(K)$$
is a field of unstable $\mathcal{R}$-invariant  cones on $\Omega_{n,p}$.

 \begin{prop} \label{est_cone2}Assume that (\ref{expanding_f}) holds. Then for $\epsilon >0$ small enough the following holds. Let $F \in \Omega_{n,p}$. If 
$v_0= u_0 + w_0$, with 
$$0< |u_0|\leq  e^{\epsilon} d( v_0, E^h_{F})\   and \  w_0\in E^h_F.$$
then for every $w_1 \in E^h_{\mathcal{R}^{j_0}F}$ and $u_1$ satisfying $D\mathcal{R}^{j_0}_{F}\cdot v_0 = u_1 + w_1$  we  have 
\begin{equation}\frac{|w_1|}{|u_1|}\leq  \label{cone_est}   \tilde{C} \tilde{\theta}^{-1} + 1 + \tilde{\theta}^{-2}  \frac{|w_0|}{|u_{0}|}. \end{equation} 
 \end{prop} 
 \begin{proof}  Let $v_1= D\mathcal{R}^{j_0}_{F}\cdot v_0$. 
 In particular 
  
  $$\frac{|u_{1}|}{|u_0|}\geq e^{-\epsilon} \frac{d(D\mathcal{R}^{j_0}_{F}\cdot v_0 , E^h_{\mathcal{R}^{j_0}F})}{d(v_0, E^h_{F})}\geq C e^{-\epsilon}  \theta^{j_0}=\tilde{\theta} > 1.$$
 Then
 $$D\mathcal{R}^{j_0}_F v_0 = D\mathcal{R}^{j_0}_F u_0 + D\mathcal{R}^{j_0}_F w_0 = u_{1}+ w_{1}.$$
  So
  $$ D\mathcal{R}^{j_0}_F \frac{u_0}{|u_{1}|}  + D\mathcal{R}^{j_0}_F \frac{w_0}{|u_{1}|} = \frac{u_{1}}{|u_{1}|}+ \frac{w_{1}}{|u_1|}$$
  and
 \begin{eqnarray}  \frac{|w_{1}|}{|u_1|}&\leq&   |D\mathcal{R}^{j_0}_F| \frac{|u_0|}{|u_{1}|}+ 1   + |D\mathcal{R}^{j_0}_F \frac{w_0}{|u_{1}|}| \nonumber \\
 &\leq&    \tilde{C} \tilde{\theta}^{-1} + 1 + C \theta^{-{j_0}}  \frac{|w_0|}{|u_{1}|}\nonumber \\
 &\leq&  \tilde{C} \tilde{\theta}^{-1} + 1 + \tilde{\theta}^{-1}C \theta^{-{j_0}}  \frac{|w_0|}{|u_{0}|} \nonumber \\
 &\leq& \tilde{C} \tilde{\theta}^{-1} + 1 + \tilde{\theta}^{-2}  \frac{|w_0|}{|u_{0}|}.
 \end{eqnarray} 
 \end{proof}
 
 \begin{cor}[Invariant Cones]\label{invariant_cones}Assume that (\ref{expanding_f}) holds.   If $v_0 \in C^u_F(K_0)$ then $v_1 = D\mathcal{R}_F^{j_0}\cdot v_0 \in C^u_{\mathcal{R}F}(K_1)$, where
 $$K_1= \tilde{C} \tilde{\theta}^{-1} + 1 + \tilde{\theta}^{-2}K_0.$$ 
 \end{cor}
 \begin{proof} We can assume that $v_0\neq 0$. Since $v_0 \in C^u_F(K_0)$   there exist   $w_0 \in E^h_F$ and $u_0$ such that $v_0=u_0+w_0$ and 
 $$ 0< |u_0|\leq e^\epsilon d(v_0,E^h_F) \  and \ |w_0| \leq K_0 |u_0|.$$
Moreover there exist  $w_1 \in E^h_{\mathcal{R}^{j_0}F}$ and $u_1$ satisfying $v_1 = u_1 + w_1$, with $$|u_1|\leq e^{\epsilon} d( v_1, E^h_{\mathcal{R}^{j_0}F}).$$  
By Proposition \ref{est_cone2} we have that $|w_1|\leq K_1 |u_1|$, so $v_1 \in C^u_{\mathcal{R}^{j_0}F}(K_1)$. 
 \end{proof}

To simplify the notation, we will replace the operator $\mathcal{R}$ by its iteration $\mathcal{R}^{j_0}$. The following two corollaries are an immediate consequence of Corollary \ref{invariant_cones}.

 \begin{cor}[Forward Invariant Cones] \label{finvariant_cones}Assume that (\ref{expanding_f}) holds.  If 
 $$K \geq \frac{\tilde{C}\tilde{\theta}^{-1}+1}{1-\tilde{\theta}^{-2}}.$$
 then for every $F \in \Omega_{n,p}$ 
 \begin{equation} D\mathcal{R}_F  C^u_F(K) \subset  C^u_{\mathcal{R}F}(K).\end{equation}
 \end{cor}

 \begin{cor}[Absorbing  Cones] \label{absorbing_cones}Assume that (\ref{expanding_f}) holds.  For each 
 $$K_0 >  \frac{\tilde{C}\tilde{\theta}^{-1}+1}{1-\tilde{\theta}^{-2}}$$
the following holds:  for every $K > 0$ there exists $i$ such that for all 
 $F \in \Omega_{n,p}$ 
 \begin{equation} D\mathcal{R}_F^{i} C^u_F(K) \subset  C^u_{\mathcal{R}^{i}F}(K_0).\end{equation}
 \end{cor}

\begin{cor}[Unstable  Cones] \label{uunstable_cones} Assume that (\ref{expanding_f}) holds.  For each 
 $K_0 >  0$ there exists $C > 0$  such that for all 
 $F \in \Omega_{n,p}$, $v \in C^u_F(K_0)$  and $i\geq 0$
 \begin{equation} |D\mathcal{R}_F^{i} v| \geq   C\theta^i |v|.\end{equation}
 \end{cor}
 \begin{proof} If $v \in C^u_F(K_0)$ then $v = u +w$, with $ w \in E^h_F$,  
 $$|u|\leq e^\epsilon d(v,E^h_F)$$
 and
 $$|w|\leq K_0 |u|.$$
 So 
 $$|v|\leq |u| + |w| \leq (1+ K_0 )e^\epsilon d(v,E^h_F) .$$
 By (\ref{expanding_f}) we have
 $$|D\mathcal{R}^i_F\cdot v| \geq d(D\mathcal{R}^i_F\cdot v,E^h_{\mathcal{R}^iF})\geq C\theta^i d(v,E^h_F) \geq \frac{C}{(1+ K_0 )e^\epsilon}\theta^i |v|.$$
 \end{proof}
 
Now fix 
$$K_0> \frac{\tilde{C}\tilde{\theta}^{-1}+1}{1-\tilde{\theta}^{-2}}.$$
Choose $i > 0$ such that 
$$\theta_1 = \frac{C}{(1+ K_0 )e^\epsilon}\theta^i > 1.$$
Replace (once again)  the operator $\mathcal{R}$ by its iteration $\mathcal{R}^i$.

\begin{cor}[Unstable  Invariant Cones near  $\Omega_{n,p}$] \label{unstable_cones} Assume that (\ref{expanding_f}) holds.  For each 
 \begin{equation}\label{uns} K_0> \frac{\tilde{C}\tilde{\theta}^{-1}+1}{1-\tilde{\theta}^{-2}}.\end{equation} 
and $\hat{\theta} \in (1, \theta_1)$  there exists $\delta > 0$ such that if 
 $$dist(F,G_0) < \delta $$ 
 for some  $G_0 \in \Omega_{n,p}$ then 
$$ D\mathcal{R}_FC^u_{G_0}(K_0) \subset C^u_{\mathcal{R}G_0}(K_0)$$
and 
 \begin{equation} |D\mathcal{R}_F \cdot v| \geq  \hat{\theta} |v|.\end{equation}
  for every $v \in C^u_{G_0}(K_0)$.
 \end{cor}
\begin{proof} 
Define $K_1= \tilde{C} \tilde{\theta}^{-1} + 1 + \tilde{\theta}^{-2}K_0.$ Then 
$$\frac{\tilde{C}\tilde{\theta}^{-1}+1}{1-\tilde{\theta}^{-2}} <   K_1 < K_0.$$
Choose $\gamma \in (0, \epsilon)$ small enough such that 
$$K_1e^\gamma < K_0.$$
Let $v \in C^u_{G_0}(K_0)$. Then there exist     $w_0 \in E^h_{G_0}$ and $u_0$ such that  $v = u_0+w_0$ and 
 $$|u_0|\leq e^\epsilon d(v_0,E^h_{G_0}) \  and \ |w_0| \leq K_0 |u_0|.$$ 
 Moreover  there exist  $w_1 \in E^h_{\mathcal{R}G_0}$ and $u_1$ satisfying
 $$|u_1|\leq e^{\gamma/3} d(D\mathcal{R}_{G_0}\cdot v,E^h_{\mathcal{R}G_0})$$ 
 and $D\mathcal{R}_{G_0}\cdot v = u_1 + w_1$.   By  Proposition \ref{est_cone2}
$$|w_1| \leq K_1 |u_1|.$$
Then 
\begin{equation} \label{a1} D\mathcal{R}_{F}\cdot v = u_1  + (D\mathcal{R}_{G_0}-D\mathcal{R}_{F})\cdot v + w_1.\end{equation} 
Note that
\begin{eqnarray} && d((D\mathcal{R}_{G_0}-D\mathcal{R}_{F})\cdot v,E^h_{\mathcal{R}G_0}) \nonumber  \\
 &\leq&  | (D\mathcal{R}_{G_0}-D\mathcal{R}_{F})\cdot v| \nonumber  \\
 &\leq&  |D\mathcal{R}_{G_0}-D\mathcal{R}_{F}| e^\epsilon (1+K_0) d(v,E^h_{G_0})\nonumber  \\
  &\leq& |D\mathcal{R}_{G_0}-D\mathcal{R}_{F}| \frac{e^\epsilon (1+K_0)}{\theta} d(D\mathcal{R}_{G_0}\cdot v,E^h_{\mathcal{R}G_0}) 
\end{eqnarray} 
Let $\delta _1 > 0$ be such that $|F-G_0| < \delta_1$ implies 
$$ 1 - |D\mathcal{R}_{G_0}-D\mathcal{R}_{F}| \frac{e^\epsilon (1+K_0)}{\theta} \geq e^{-\gamma/3},$$
$$e^{\gamma/3} + |D\mathcal{R}_{G_0}-D\mathcal{R}_{F}| \frac{e^\epsilon (1+K_0)}{\theta} \leq e^{2\gamma/3}, $$
and
$$\tilde{\theta}= \theta_1-  |D\mathcal{R}_{G_0}-D\mathcal{R}_{F}|> \hat{\theta}  > 1. $$
Then 
\begin{eqnarray} &&d(D\mathcal{R}_{F}\cdot v,E^h_{\mathcal{R}G_0})\nonumber  \\ 
&\geq& d(D\mathcal{R}_{G_0}\cdot v,E^h_{\mathcal{R}G_0}) - d((D\mathcal{R}_{G_0}-D\mathcal{R}_{F})\cdot v,E^h_{\mathcal{R}G_0}) \nonumber  \\
 &\geq&  e^{-\gamma/3}  d(D\mathcal{R}_{G_0}\cdot v,E^h_{\mathcal{R}G_0}) 
\end{eqnarray} 
so 
\begin{eqnarray} |u_1  + (D\mathcal{R}_{G_0}-D\mathcal{R}_{F})\cdot v|&\leq& e^{\gamma/3} d(D\mathcal{R}_{G_0}\cdot v,E^h_{\mathcal{R}G_0}) +  | (D\mathcal{R}_{G_0}-D\mathcal{R}_{F})\cdot v| \nonumber \\
&\leq& e^{2\gamma/3} d(D\mathcal{R}_{G_0}\cdot v,E^h_{\mathcal{R}G_0}) \nonumber \\ 
 &\leq& e^\gamma d(D\mathcal{R}_{F}\cdot v,E^h_{\mathcal{R}G_0}) \nonumber \\
\label{a2} &\leq& e^\epsilon d(D\mathcal{R}_{F}\cdot v,E^h_{\mathcal{R}G_0}),
\end{eqnarray} 
and moreover
\begin{eqnarray} |u_1  + (D\mathcal{R}_{G_0}-D\mathcal{R}_{F})\cdot v|&\geq& |u_1| -  | (D\mathcal{R}_{G_0}-D\mathcal{R}_{F})\cdot v| \nonumber \\
&\geq& d(D\mathcal{R}_{G_0}\cdot v,E^h_{\mathcal{R}G_0}) -   | (D\mathcal{R}_{G_0}-D\mathcal{R}_{F})\cdot v|   \nonumber \\ 
 &\geq& e^{-\gamma/3} d(D\mathcal{R}_{G_0}\cdot v,E^h_{\mathcal{R}G_0}).
\end{eqnarray}  
Finally 
\begin{eqnarray} |w_1|&\leq& K_1 |u_1|  \nonumber \\
&\leq& K_1 e^{\gamma/3}  d(D\mathcal{R}_{G_0}\cdot v,E^h_{\mathcal{R}G_0}) \nonumber \\
 &\leq& K_1 e^{2\gamma/3} |u_1  + (D\mathcal{R}_{G_0}-D\mathcal{R}_{F})\cdot v|.\nonumber \\
\label{a3} &\leq& K_0 |u_1  + (D\mathcal{R}_{G_0}-D\mathcal{R}_{F})\cdot v|.
\end{eqnarray} 
So (\ref{a1}), (\ref{a2}) and (\ref{a3}) implies that $D\mathcal{R}_{F}\cdot v \in C^u_{\mathcal{R} G_0}(K_0)$.  Furthermore
\begin{eqnarray}  |D\mathcal{R}_{F}\cdot v| &\geq&   |D\mathcal{R}_{G}\cdot v|  -  |(D\mathcal{R}_{G_0}-D\mathcal{R}_{F})\cdot v| \nonumber \\
&\geq&   \theta_1 |v|  -  |D\mathcal{R}_{G_0}-D\mathcal{R}_{F}||v| \nonumber \\
&\geq&  \hat{\theta} |v|.
\end{eqnarray} 
\end{proof}

\begin{proof}[Proof of Proposition  \ref{key}]  Let $K_0$ be as in Corollary \ref{unstable_cones}.  We claim that every cone $C^u_{F}(K_0)$, with $F \in \Omega_{n,p}$, contains a subspace $S_F$ of dimension $n$. Indeed, since $E^h_{G}$, $G \in \Omega_{n,p}$,  has finite codimension $n$ there is a subspace $E_G \subset T\mathcal{B}_{nor}(U)$ of dimension $n$  such that 
$$E_G\oplus E^h_{G}= T\mathcal{B}_{nor}(U).$$
Since $\Omega_{n,p}$ is a Cantor set and $G\mapsto E^h_G$ is a continuous distribution, it is easy to see that we can find a finite  covering $\{O_j\}_j$ by compact  subsets of $\Omega_{n,p}$ and subspaces $E_j$ with dimension  $n$ such that 
$$E_j\oplus E^h_{G}= T\mathcal{B}_{nor}(U)$$
for every $G \in O_j$. 
By (\ref{cobre}), Proposition \ref{est_cone2} and Corollary \ref{absorbing_cones}  there exists $i_0$ such that 
$$D\mathcal{R}^{i_0}_GE_j\subset C^u_{\mathcal{R}^{i_0}G}(K_0).$$
for every $G \in O_j$.  
Moreover  $\mathcal{R}$ is invertible on $\Omega_{n,p}$, so we can choose $G$ such that $\mathcal{R}^{i_0}G=F$. Since  $D\mathcal{R}_G$ is injective for every $G \in \Omega_{n.p}$ we conclude that $S_F=D\mathcal{R}^{i_0}_GE_j$ is a subspace of dimension $n$ in $C^u_{F}(K_0)$.  This concludes the proof of the claim. Note that
$$S_F\oplus E^h_{F}= T\mathcal{B}_{nor}(U).$$

Let $G \in \Omega_{n,p}$. Since $\mathcal{R}$ is invertible on $\Omega_{n,p}$ there is a unique sequence $G_{i} \in \Omega_{n,p}$ such that $\mathcal{R}G_{i+1}=G_i$ and $G_0=G$. Let $S'_i$ be an arbritrary subspace of dimension $n$ contained in $C^u_{G_i}(K_0)$,  Then $D\mathcal{R}_{G_i}^i (S'_i)$ is a subspace of dimension $n$ contained in $C^u_{G}(K_0)$. Since 
\begin{equation}\label{intu}  C^u_{G}(K_0)\cap E^h_G=\{0\}\end{equation} 
we have that there is a linear function
$$\mathcal{H}_i\colon S_{G} \rightarrow E^h_G$$
such that $\{v+\mathcal{H}_i(v), \ v \in S_G\}= D\mathcal{R}_{G_i}^i (S'_i).$  Since $v+\mathcal{H}_i(v) \in C^u_G(K_0)$ we have $v+\mathcal{H}_i(v)= u_1+w_1$, with $w_1 \in E^h_G$ and
$$|u_1|\leq e^\epsilon d(v+\mathcal{H}_i(v),E^h_G)= e^\epsilon d(v,E^h_G)< e^\epsilon |v|.$$
and $ |w_1|\leq K_0 |u_1|.$
In particular
$$|\mathcal{H}_i(v)|=|u_1+w_1-v|\leq (1+ e^\epsilon(1+K_0))|v|,$$
So the family of functions $\mathcal{H}_i$ is an equicontinuous family of functions.  By the strong compactness of the operator $\mathcal{R}$ and Arzela-Ascoli  Theorem there exists a subsequence that converges to a bounded linear function $\mathcal{H}^G\colon S_G \rightarrow E^h_G$ satisfying
$$\{v+\mathcal{H}^G(v), \ v \in S_G\} \subset \bigcap_{i=0}^\infty  \overline{D\mathcal{R}_{G_i}^i C^u_{G_i}(K_0)} \subset  C^u_G(K_0).$$
Due (\ref{intu}) and the contraction in the horizontal directions we have that there is only one possible accumulation point $\mathcal{H}^G$ for sequences as $(\mathcal{H}_i)_i$. Let 
$$E^u_G=\{v+\mathcal{H}^G(v), \ v \in S_G\}.$$
 Then we can easily conclude that $D\mathcal{R}_G(E^u_G)=E^u_{\mathcal{R}G}$.  Then $G\mapsto E^u_G$ is the unstable direction of $\mathcal{R}$. 
\end{proof}

\section{Induced expanding  maps.} 

Let $F \in \Omega_{n,p}$.  We are going to define a real induced map $G_\mathbb{R}\colon D\rightarrow \mathbb{R}$ for $F$ whose domain $D$ is the union of intervals $R_{-i}^k$, $k\geq 0$ and $0< i < n_r^k$, $r \in C(F)$, satisfying 
$$R_{-i}^{k}\subset  \bigcup_{q \in C(F)} Q_0^{k-1}.$$
If $R_{-i}^{k} \subset  Q_0^{k-1}$ and  $s \in C(F)$ is the sucessor of $q$ at the $k-1$ level  we  define  $G_\mathbb{R}(x)=F^{n_s^{k-1}}(x)$ for every $x \in R_{-i}^{k}$. Note that  $$G_\mathbb{R}\colon R_{-i}^{k}\rightarrow R_{-i+n_s^{k-1}}^k$$
is a diffeomorphism. Due the real bounds there exists $\epsilon_0$ such that 

\begin{equation} \label{esc_ep} 0< \epsilon_0< \inf_{r \in C(F), k}  \frac{2 dist(P(F),\partial R_0^k) }{|R_0^k|}.\end{equation} 

We will now define a complex-analytic extension $G$ of $G_\mathbb{R}$ that is a conformal iterated dynamical system. Suppose that  $R_{-i}^{k}\subset Q^{k-1}_0$. Let $\mathscr{R}_0^k$ be the ball $B(r,d^k_r)$, where 
\begin{equation}\label{dkr}  d^k_r=(1-\frac{\epsilon_0}{16})\frac{|R_0^k|}{2}.\end{equation}
 Since $F$ belongs to the Epstein class, there exists a simply connected domain $\mathscr{R}_{-i}^k$ such that  $\mathscr{R}_{-i}^k\cap \mathbb{R} \subset  R_{-i}^k$ and $\mathscr{R}_{-i}^k $ is contained in the ball whose diameter is $R_{-i}^k\cap \mathbb{R}$ and moreover 
$$F^i\colon \mathscr{R}_{-i}^k \rightarrow \mathscr{R}_{0}^k $$
is univalent. Due the real bounds, we can reduce $\epsilon_0$ if necessary in such way that 
$$\mathscr{R}_{-i}^k \subset \mathcal{Q}^{k-1}_0$$
for every $R_{-i}^k \subset Q^{k-1}_0$ and 
$$\inf_k \inf_{\substack{\{q,r \} \subset C(F) \\  R_{-i}^k \subset Q^{k-1}_0}} dist(\mathscr{R}_{-i}^k,  \partial \mathcal{Q}^{k-1}_0) > 0$$

\begin{figure}
\psfrag{g}{$G$}
\psfrag{q}[][][0.8]{$q$}
\psfrag{r}[][][0.8]{$r$}
\psfrag{s}[][][0.8]{$s$}
\psfrag{sk1}{$\mathcal{S}^{k-1}_0$}
\psfrag{rk1}{$\mathscr{R}^{k-1}_0$}
\psfrag{qk1}{$\mathcal{Q}^{k-1}_0$}
\psfrag{sk}[][][0.8]{$\mathcal{S}^{k}_0$}
\psfrag{sk-1}[][][0.8]{$\mathcal{S}^{k}_{-i_1}$}
\psfrag{sk-2}[][][0.8]{$\mathcal{S}^{k}_{-i_2}$}
\psfrag{sk-3}[][][0.8]{$\mathcal{S}^{k}_{-i_3}$}
\psfrag{sk-4}[][][0.8]{$\mathcal{S}^{k}_{-i_4}$}
\psfrag{rk}[][][0.8]{$\mathscr{R}^{k}_0$}
\psfrag{rk-1}[][][0.8]{$\mathscr{R}^{k}_{-j_1}$}
\psfrag{qk}[][][0.8]{$\mathcal{Q}^{k}_0$}
\psfrag{qk-1}[][][0.8]{$\mathcal{Q}^{k}_{-\ell_1}$}


\includegraphics[scale=0.35]{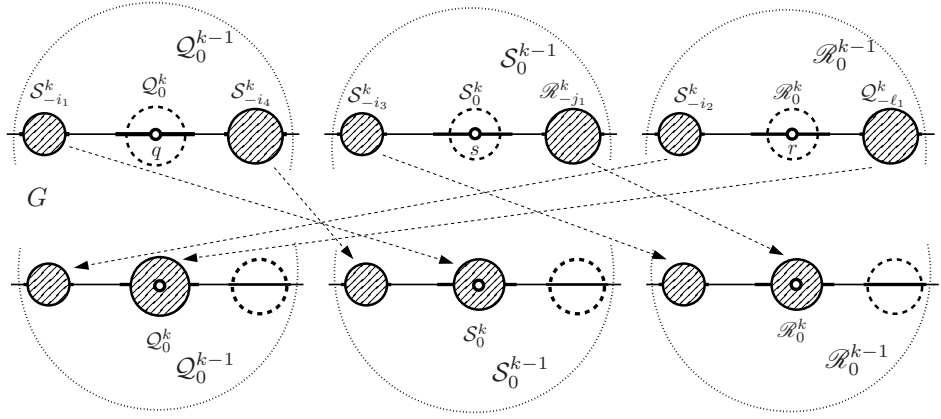}
\caption{How $G$ acts on the domains $\mathscr{R}_{-i}^k \subset \cup_{q \in C(F)} \mathcal{Q}^{k-1}_0$ if the combinatorics of the $k$-th renormalization is the same as the combinatorics of the renormalization $R(F)$   in Fig. 1. At level $k-1$ we have that  $s$ is the sucessor of $q$,  $r$ is the sucessor of $s$ and $q$ is the sucessor of $r$.  At level $k$ we have that  $r$ is the sucessor of $q$, $s$ is the sucessor of $r$  and  $q$ is the sucessor of $s$. Moreover  $\ell_1, j_1 > 0$ and $0< i_1< i_2< i_3 < i_4$.  }
\end{figure}

We define $G$ on
$$\mathscr{D}=\bigcup_k \bigcup_{r,q \in C(F)} \bigcup_{R_{-i}^{k} \subset Q^{k-1}_0} \mathscr{R}_{-i}^k$$
as $G(z)=F^{n_s^{k-1}}(z)$ for every $z \in \mathscr{R}_{-i}^{k}$, where $R_{-i}^k  \subset Q_0^{k-1}$ and $s \in C(F)$ is the sucessor of $q$ at level $k-1$.  Note that $P(F)\setminus C(F)  \subset \mathscr{D}$ and $$G(P(F)\setminus C(F) )= P(F).$$

\begin{lem} [Markovian property of the induced map]\label{lemd} Let  $r_j \in C(F)$, $m_i \in \mathbb{N}$ and $0\leq  i_j < n_{r_j}^{m_i}$, $j\leq \ell$ be such that 

\begin{itemize}
\item[A.] either we have that $m_{j+1}=m_{j}$, $i_{j+1} <  i_j$, $r_{j+1}=r_j=r$ for some $r \in C(F)$   and moreover
  $$R_{-i_{j}}^{m_{j}}, R_{-i_{j+1}}^{m_{j}} \subset \bigcup_{q \in C(F)} Q_0^{m_j-1}.$$
and
$$R_{-i}^{m_{j}} \not\in \bigcup_{q \in C(F)} Q_0^{m_j-1}$$
for every $i$ satisfying $i_{j+1} < i < i_{j}$. In particular $$G_{\mathbb{R}}\colon R_{-i_{j}}^{m_{j}} \mapsto R_{-i_{j+1}}^{m_{j+1}}.$$
is a diffeomorphism.

\item[B.] or 
$$m_{j+1} > m_j,$$ 
with
$$ R_{-i_{j}}^{m_{j}} \subset \bigcup_{q \in C(F)} Q_0^{m_j-1},$$
and
$$R_{-i}^{m_{j}} \not\in \bigcup_{q \in C(F)} Q_0^{m_j-1}$$
for every $i$ satisfying $0 < i < i_j$. Here $r=r_j$. In particular $$G\colon R_{-i_{j}}^{m_{j}} \mapsto R_{0}^{m_{j}}.$$
is a diffeomorphism.
Moreover  $i_{j+1} > 0$ and 
$$S_{-i_{j+1}}^{m_{j+1}} \subset R_0^{m_{j+1}-1},$$
where $s=r_{j+1}$. In particular $$G_{\mathbb{R}}\colon R_{-i_{j}}^{m_{j}} \mapsto R_{0}^{m_{j}}.$$
is a diffeomorphism and $$S_{-i_{j+1}}^{m_{j+1}}  \subset R_{0}^{m_{j}}=G_{\mathbb{R}}(R_{-i_{j}}^{m_{j}} ).$$
\end{itemize}

Then there exists a unique  interval $W$ such that $$G^\ell_\mathbb{R}\colon W \rightarrow R_{-i_\ell}^{m_\ell}$$
is a diffeomorphism and $W$ is the set of points $z$ such that for every $j\leq \ell$ we have  $G^j_\mathbb{R} (z) \in Q_{-i_{j}}^{m_{j}}$, where $q=r_j$.  Moreover  $W= R_{-i}^{m_\ell}$, for some $i$, where $r=r_\ell$.  
\end{lem} 
\begin{proof} If $\ell=0$, there is nothing to prove, since $W=R^{m_0}_{-i_0}$, with $r =r_0$. Suppose by induction on $\ell$  that   Lemma \ref{lemd} holds for $\ell$. 
 Let  $r_j \in C(F)$, $m_i \in \mathbb{N}$ and $0< i_j < n_{r_j}^{m_i}$, $j\leq \ell+1$ be as in the statement of  the lemma.  By the induction assumption there exists $b$ such that 
$$G^\ell_\mathbb{R}\colon R_{-b}^{m_{\ell+1}} \rightarrow R_{-i_{\ell+1}}^{m_{\ell+1}}$$
is a diffeomorfism and  for every $j\leq \ell$ we have  $G^j_\mathbb{R}R_{-b}^{m_{\ell+1}} \subset Q_{-i_{j+1}}^{m_{j+1}} $, where $q=r_{j+1}$.
In particular $R_{-b}^{m_{\ell+1}} \subset S_{-i_{1}}^{m_{1}}$, with $s=r_1$.  There are two cases. If $m_0=m_1$ then $S_{-i_{0}}^{m_{0}}=S_{-i_{0}}^{m_{1}}$, $i_0 > i_1$, and 
$$G_{\mathbb{R}}\colon S_{-i_{0}}^{m_{1}}\rightarrow S_{-i_{1}}^{m_{1}}$$
is  a diffeomorphism, and $W=R^{m_\ell}_{-(b+ i_0-i_1)}$  is the unique interval $W \subset S_{-i_{0}}^{m_{1}}$ such that  $G_\mathbb{R}(W)=R^{m_\ell}_{-b}$. If $m_1 > m_0$ then $S_{-i_{1}}^{m_{1}}\subset Q_{0}^{m_0}$, with $q=m_0$, and 
$$G\colon Q_{-i_0}^{m_0} \rightarrow Q_0^{m_0}$$
is a diffeomorphism.  Then $W=R^{m_\ell}_{-(b+i_0)}$ is the unique interval $W \subset Q_{-i_0}^{m_0}$ such that $G(W)=R^{m_\ell}_{-b}$. 
\end{proof}

\begin{figure}
\psfrag{sk1}{$\mathcal{S}^{k+1}_{-i}$}
\psfrag{rk1}{$\mathscr{R}^{k+1}_{-j}$}
\psfrag{qk2}[][][0.3]{$\mathcal{Q}^{k+2}_0$}
\psfrag{qk1}[][][0.8]{$\mathcal{Q}^{k+1}_0$}
\psfrag{qk0}{$\mathcal{Q}^{k}_0$}
\psfrag{sk2}[][][0.5]{$\mathcal{S}^{k+2}_{-\ell}$}
\psfrag{qk2}[][][0.3]{$\mathcal{Q}^{k+2}_0$}
\psfrag{qk2-1}[][][0.3]{$\mathcal{Q}^{k+2}_{-a}$}
\psfrag{sk2-1}[][][0.3]{$\mathscr{R}^{k+2}_{-b}$}
\psfrag{rk3}[][][0.2]{$\mathcal{Q}^{k+3}_{-d}$}


\includegraphics[scale=0.6]{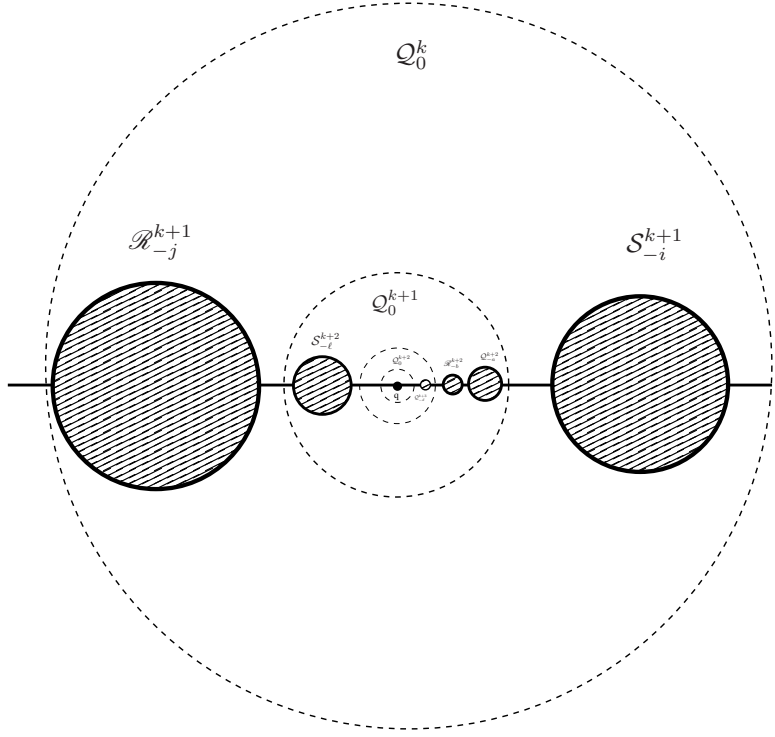}
\caption{The domain $\mathscr{D}$ of $G$ close to a critical point $q$ (the point in the center). It contains infinitely many topological disks accumulating on the critical point $q$.  }
\end{figure}

The next proposition says that the postcritical set $P(F)$ of $F$ is the maximal invariant set of the induced map $G$.
\begin{prop} \label{pulafora} Given $(z,j) \in \mathscr{D}$ we have that  $(z,j)$ belongs to $\mathscr{D}\setminus P(F)$ if and only if   there exists $k_0\geq 0$ such that $G^k(z,j) \in \mathscr{D}$ for every $k< k_0$ and $G^{k_0}(z,j) \not\in \mathscr{D}$.
\end{prop}
\begin{proof} Let  $r_j \in C(F)$, $m_i \in \mathbb{N}$ and $0< i_j < n_{r_j}^{m_i}$ such that either  $$m_{j+1} > m_j$$ and
$$R_{-i_{j+1}}^{m_{j+1}} \subset Q_0^{m_j}.$$
where  $r=r_{j+1}$ and $q=r_j$, or  $m_{j+1}=m_{j}$ and $i_{j+1}=i_j-1$. We claim that there exists a unique $z \in \mathbb{C}_n$ such that 
\begin{equation}\label{it34}  G^j(z) \in \mathscr{R}_{-i_{j}}^{m_{j}}\end{equation} 
for every $j\geq 0$. Indeed, for each $\ell$, let $D_\ell$ be the set of points such that (\ref{it34}) holds for every $j\leq \ell$.  Of course $D_{\ell+1}\subset D_\ell$. If $m_{\ell+1}=m_\ell$ then $D_{\ell+1}=D_\ell$. If $m_{\ell+1}> m_\ell$ then 
$$G^{\ell+1}\colon D_\ell \rightarrow \mathcal{Q}_0^{m_{\ell}},$$
with  $q=r_{\ell}$ and
$$G^{\ell+1}\colon D_{\ell+1} \rightarrow \mathcal{R}^{m_{\ell+1}}_{-i_{\ell+1}},$$
with $r=r_{\ell+1}$, are univalent (and onto). By the definition of $G$ 
$$M=\inf_{m} \inf_{\mathcal{R}^{m+1}_{-i} \subset \mathcal{Q}_0^m} \mod( \mathcal{Q}_0^m\setminus \mathcal{R}^{m+1}_{-i} ) > 0,$$
so
$$\mod(D_\ell\setminus D_{\ell+1})= \mod(\mathcal{Q}_0^{m_{\ell}}\setminus \mathcal{R}^{m_{\ell+1}}_{-i_{\ell+1}})\geq M,$$
in particular, since there exists infinitely many $\ell$ such that $m_{\ell+1} > m_\ell$ we have that 
$$\lim_\ell \ diam(D_\ell)=0$$
and
$$\cap_\ell D_\ell=\{z_0\}$$
for some $z_0 \in \mathscr{D}$. This completes the proof of the claim. Note that since $D_\ell$ are symmetric with respect to the $\mathbb{R}$  we have that $z_0 \in \mathbb{R}$ and by (\ref{it34})
\begin{equation}\label{iit34}  G^j(z_0) \in \mathscr{R}_{-i_{j}}^{m_{j}}\cap \mathbb{R} \subset R_{-i_{j}}^{m_{j}},\end{equation} 
where $r=r_j$, for every $j$. By Lemma \ref{lemd} there exists $b_j$ such that 
$$z_0 \in R^{m_j}_{-b_j},$$
where $r=r_j$, for every $j\geq 0$. In particular $z_0 \in P(F)$. 
\end{proof}

\section{Induced problem.} 

Let $a_i$ be as in (\ref{de_it}). Define the function
$$V\colon \mathscr{D} \rightarrow \mathbb{C}$$ as $$V(z)=a_i(z) =\frac{\partial}{\partial t} (F_t)^{n_s^{k-1}}(z)\big|_{t=0} $$ for every $z \in \mathscr{R}_{-i}^k\subset \mathscr{D}$, provided $R_{-i}^k \subset Q_0^{k-1}$ and $s$ is the sucessor of $q$ at the $k-1$ level. Here $F_t =F+tv$.

\begin{lem} \label{ind_lem}Suppose that $\alpha$ is a continuous vector field  in a neighbourhood of $P(F)$, such that
\begin{itemize}
\item[A1.] $\alpha(c)=0$ for every $c \in C(F)$, and\\
\item[A2.]  We have
\begin{equation} \label{ind_eq1} V(z)=\alpha\circ G(z) - DG(z)\cdot \alpha(z)\end{equation}
for every $z \in P(F)\setminus C(F)$.
\end{itemize}
 Let  $x \in P(F)$ and $\ell > 0$ such that $G^j(x)\not\in  C(F)$ for every $0\leq j < \ell$.  Then for each $j < \ell$ there is $r, q  \in C(F)$, with $b_j > 0$   such that $G^j(x) \in R_{-b_j}^{k_j}\subset Q^{k_j-1}_0$. Define $i_j = n^{k_j-1}_s$, where $s \in C(F)$ is the sucessor of $q$ at level $k-1$. Note that the critical points $q,s,r$ may depend on $x$ and $j$. Let    $$m= \sum_{j=0}^{\ell-1} i_j.$$
Then 
\begin{itemize}
\item[B1.] For every $z \in \mathbb{C}$ close enough to $x$ we have that $G^\ell(z)$ is well defined and $G^\ell(z)=F^m(z)$.
\item[B2.] For every  $z \in P(F)$ close enough to $x$  we have that $\alpha(G^\ell(z))$ is well defined and 
$$\alpha(z) = \frac{\alpha(G^\ell(z))}{DG^\ell(z)} -\sum_{a=1}^{m} \frac{v(F^{a-1}(z))}{DF^a(z)}.$$
\item[B3.] In particular for every  $z \in P(F)$ close enough to $x$ such that  $G^\ell(x) \in C(F)$ we have 
$$\alpha(x) = -\sum_{a=1}^{m} \frac{v(F^{a-1}(x))}{DF^a(x)}.$$
\end{itemize}
\end{lem}
\begin{proof} We let to the reader to show that $G^\ell(z)=F^m(z)$ for $z$ close enough to $x$.  By (\ref{ind_eq1}) we get
$$\alpha(x) = \frac{\alpha(G^\ell(z))}{DG^\ell(z)} -\sum_{j=1}^{\ell} \frac{V(G^{j-1}(x))}{DG^j(z)}.$$
So
$$\sum_{j=1}^{\ell} \frac{V(G^{j-1}(x))}{DG^j(z)} = \sum_{j=1}^{\ell} \frac{(\partial_t (F_t)^{i_j})(F^{i_0+i_1+\cdots+i_{j-1}}(z))}{DF^{i_0+i_1+\cdots+i_{j}}(z)} $$ 
$$ =  \sum_{j=1}^{\ell} \sum_{b=0}^{i_j-1}  \frac{DF^{i_j-b-1}(F^{b+1}(F^{i_0+i_1+\cdots+i_{j-1}}(z)))v(F^b(F^{i_0+i_1+\cdots+i_{j-1}}(z)))}{DF^{i_0+i_1+\cdots+i_{j}}(z)}   $$ 
$$=  \sum_{j=1}^{\ell} \sum_{b=0}^{i_j-1}  \frac{DF^{i_j-b-1}(F^{b+1}(F^{i_0+i_1+\cdots+i_{j-1}}(z)))v(F^{i_0+i_1+\cdots+i_{j-1}+b}(z))}{ DF^{i_j-b-1}(F^{b+1}(F^{i_0+i_1+\cdots+i_{j-1}}(z)))   DF^{i_0+i_1+\cdots+i_{j-1}+b+1}(z)}    $$ 
$$=  \sum_{j=1}^{\ell} \sum_{b=0}^{i_j-1}  \frac{v(F^{i_0+i_1+\cdots+i_{j-1}+b}(z))}{   DF^{i_0+i_1+\cdots+i_{j-1}+b+1}(z)}    = \sum_{a=1}^{m} \frac{v(F^{a-1}(z))}{DF^a(z)}.$$ 

\end{proof}

\begin{prop} \label{ind_ipa} Suppose that $\alpha$ is a continuous vector field  in a neighbourhood of $P(F)$, such that $\alpha(r)=0$ for every $r \in C(F)$ and  
\begin{equation} \label{ind_eq} V(z)=\alpha\circ G(z) - DG(z)\cdot \alpha(z).\end{equation}
for every $z \in P(F)\setminus C(F)$. Then

\begin{equation} \label{or_eq} v(z)=\alpha\circ F(z) - DF(z)\cdot \alpha(z)\end{equation}
for every $z \in P(F)$. 
\end{prop}
\begin{proof} For each $R_{-m}^k$, with $r \in C(F)$ and  $0< m < n_r^k$,  there exists a unique $r^k_{-m} \in R_{-m}^k$ such that 
$$F^m(r^k_{-m})=r.$$
The set 
$$\Gamma = \{  r^k_{-m} \}_{r \in C(F), 0< m < n_r^k} \subset P(F)$$
is dense on $P(F)$ and
$$F(\Gamma)=\Gamma \cup C(F).$$

 We claim that   (\ref{or_eq}) holds for every $z \in \Gamma$. Indeed, given $r^k_{-m} \in \Gamma$, there exists $\ell >0$ such that $G^\ell(r^k_{-m})=0$. Moreover by Lemma \ref{ind_lem} we have that $G^\ell = F^m$ in a neighbourhood of $r^k_{-m}$ and 
\begin{equation}\label{formula} \alpha(r^k_{-m} ) = -\sum_{a=1}^{m} \frac{v(F^{a-1}(r^k_{-m} ))}{DF^a(r^k_{-m} )}.\end{equation}

Suppose that $m=1$. Then $\ell=1$, $G=F$ in a neighbourhood of $r_{-1}^k$, $V(r_{-1}^k)=v(r_{-1}^k)$ and  $G(r_{-1}^k) \in C(F)$. By (\ref{ind_eq}) it follows that 
$$v(r_{-1}^k)= \alpha\circ G(r_{-1}^k) - DG(r_{-1}^k)\cdot \alpha(r_{-1}^k)= \alpha\circ F (r_{-1}^k) - DF(r_{-1}^k)\cdot \alpha(r_{-1}^k).$$
If $m > 1$ then $F(r_{-m}^k)= r_{-(m-1)}^k \in \Gamma$, so by (\ref{formula}) we have

$$\alpha(F(r_{-m}^k))-DF(r_{-m}^k)\alpha(r_{-m}^k)$$
$$ = -\sum_{a=1}^{m-1} \frac{v(F^{a-1}(r^k_{-(m-1)} ))}{DF^a(r^k_{-(m-1)} )}+ DF(r_{-m}^k) \sum_{b=1}^{m} \frac{v(F^{b-1}(r^k_{-m} ))}{DF^b(r^k_{-m} )}$$
$$ = -\sum_{a=1}^{m-1} \frac{v(F^{a-1}(r^k_{-(m-1)} ))}{DF^a(r^k_{-(m-1)} )}+\sum_{b=1}^{m} \frac{v(F^{b-1}(r^k_{-m} ))}{DF^{b-1}(r^k_{-(m-1)} )}$$ 
$$ = -\sum_{a=1}^{m-1} \frac{v(F^{a-1}(r^k_{-(m-1)} ))}{DF^a(r^k_{-(m-1)} )}+v(r^k_{-m}) + \sum_{b=2}^{m} \frac{v(F^{b-2}(r^k_{-(m-1)} ))}{DF^{b-1}(r^k_{-(m-1)} )}$$ 
$$ = -\sum_{a=1}^{m-1} \frac{v(F^{a-1}(r^k_{-(m-1)} ))}{DF^a(r^k_{-(m-1)} )}+v(r^k_{-m}) + \sum_{a=1}^{m-1} \frac{v(F^{a-1}(r^k_{-(m-1)} ))}{DF^{a}(r^k_{-(m-1)} )}$$ $$ = v(r^k_{-m}).$$ 

So (\ref{or_eq}) holds for every $z \in \Gamma$. Since $v$, $\alpha$ and $F$  are  continuous in a neighbourhood of $P(F)$ and $\Gamma$ is dense in $P(F)$, it follows that  (\ref{or_eq}) holds for every $z \in P(F)$.  

\end{proof}

\begin{cor}[Induced problem] \label{main_cor}  Let  $F \in \Omega_{n,p}$ and $v \in \mathfrak{B}_+(F)$. If  there exists a quasiconformal vector field $\alpha$, defined in a neighbourhood  of $P(F)$, 
such that $\alpha(r)=0$ for every $r \in C(F)$ and  
\begin{equation}V(z)=\alpha\circ G(z) - DG(z)\cdot \alpha(z)\end{equation}
for every $z \in P(F)\setminus C(F)$, then $v \in E^h(F)$.
\end{cor}
\begin{proof} This is an immediate consequence of Proposition \ref{ipa} and Proposition \ref{ind_ipa}. 

\end{proof}

\section{Solving the induced problem.}\label{keysec}

We are going to change a little bit the notation. Let $P^k_i$, $i=1, \dots, n$   be the set of restrictive intervals for $F$ associated to the renomalization $\mathcal{R}^kF$, that is
\begin{itemize}
\item[A.] These intervals are pairwise disjoint, $C(F)\cap P^k_i \neq \emptyset $ and 
$$C(F)\subset \cup_i P^k_i.$$
\item[B.] There are integers $n^k_i$ such that 
$$F^{n^k_i}\colon P^k_i \rightarrow P^k_{i+1 \mod n}$$
is an unimodal map. 
\item[C.] The have that $P^k_i= [\delta^k_i,b^k_i]$, where $\delta^k_i$ is a $\mu^{(k)}$-periodic repelling fixed point, with
$$\mu^{(k)} =\sum_i n_i^k,$$
and  $F^{n^k_i}(\delta^k_i)= F^{n^k_i}(b^k_i).$
\item[D.] The renormalization associated to the restrictive interval $P^k_1$ is $\mathcal{R}^kF$. 
\end{itemize}


If $t$ is small enough then  $F_t=F+tv$ is close to $F$ and  $\delta^k_i$ has an analytic continuation $\delta^k_i(t)$. Denote by $\partial_t \delta^k_i$ the derivative of this continuation with respect to $t$ at $t=0$.     Consider $ \{q_i^k\}= C(F)\cap P^k_i$. Then  $q_{i+1}^k$ is the successor of $q_i^k$ at level $k$.  
Let $j_i^k \in \{1, \dots, n\}$ be such that  $q_i^k \in I_{j_i^k}$, and 
$$A^k_i\colon \mathbb{C}\times\{j_i^k\} \rightarrow \mathbb{C}\times\{ i\}$$
be the only affine transformation such that $A^k_i(\delta^k_i)=(-1,i)$ and $A^k_i(q^k_i)=(0,i)$.  

\begin{lem}\label{div} We have $V =V_1 + V_2$, where
\begin{equation} V_1(z,j_i^k)= -\delta_{i+1}^k v^k\circ A_i^k (z,j_i^k),\end{equation} 
and
\begin{equation} V_2(z,j_i^k)=  \frac{ \partial_t \delta_{i+1}^k}{\delta^k_{i+1}}  F^{n_i^k}(z,j_i^k) -  D F^{n_i^k}(z,j_i^k) \cdot  \frac{ \partial_t \delta_{i}^k}{\delta^k_{i}}  (z,j_i^k)\end{equation}
for every $(z,j_i^k) \in \mathscr{R}_{-\ell}^{k+1}$, with $R_{-\ell}^{k+1} \subset Q_0^{k}$, $q=q_i^k$.
\end{lem} 
\begin{proof} Note that  
$$v^k(x,i)= \partial_t (\mathcal{R}^k(F_t)) \big|_{t=0}(x,i)=(D\mathcal{R}_F^k\cdot v)(x,i)$$
$$= - \frac{\partial_t \delta_{i+1}^k}{\delta_{i+1}^k}\cdot  A_{i+1}^k \circ  F^{n_i^k}\circ (A_{i}^k)^{-1}(x,i)$$
$$- \frac{1}{\delta_{i+1}^k}\big(    a_{n_i^k}\circ (A_{i}^k)^{-1}(x,i) + (D F^{n_i^k})\circ (A_{i}^k)^{-1}(x,i)\cdot (-\partial_t \delta_{i}^k \ x,i) \big),$$

So if $(x,i)= A^k_i(z,j_i^k)$, with $(z,j_i^k) \in \mathscr{R}_{-\ell}^{k+1}$, with $R_{-\ell}^{k+1} \subset Q_0^{k}$  we have that 

$$V(z,j_i)= a_{n_i^k}(z,j_i^k)$$
$$ = -\delta_{i+1}^k  v^k\circ A_i^k (z,j_i^k) + \frac{ \partial_t \delta_{i+1}^k}{\delta^k_{i+1}}  F^{n_i^k}(z,j_i^k) -  D F^{n_i^k}(z,j_i^k) \cdot  \frac{ \partial_t \delta_{i}^k}{\delta^k_{i}}  (z,j_i^k) .$$
\end{proof}

\begin{lem} \label{div2} Let $v \in \mathfrak{B}_+(F)$. There exists $C_3 > 0$ such that for every $k$ and every $i_{k},i_{k+1}$ such that  $q_{i_k}^k=q_{i_{k+1}}^{k+1}$ we have  
\begin{equation}
\big| \frac{\partial_t \delta^{k+1}_{i_{k+1}}}{\delta^{k+1}_{i_{k+1}}} -\frac{\partial_t \delta^{k}_{i_k}}{\delta^{k}_{i_k}} \big| \leq C_3
\end{equation}
\end{lem}
\begin{proof}
Let $q = q_{i_k}^k=q_{i_{k+1}}^{k+1}$ and  $i_j$ be such that  $q^j_{i_j}=q$. Note that $\beta^{k+1}_{q} = A^k_i(\delta^{k+1}_{i_{k+1}})$ is a periodic point for $\mathcal{R}^kF$ with period $y^k= \mu^{(k+1)}/\mu^{(k)}$. Indeed
$$\beta^{k+1}_q = (\frac{\delta^{k+1}_{i_{k+1}}}{\delta^{k}_{i_k}},i_{k}).$$  If $t$ is small enough then there is an analytic continuation $\beta^{k+1}_q(t)$ for $\beta^{k+1}_q$, that is a periodic point for $\mathcal{R}^kF_t$. Since

$$ \partial_t (\mathcal{R}^k(F_t)) \big|_{t=0}(x,i)=(D\mathcal{R}_F^k\cdot v)(x,i)= v^k(x,i),$$
by the Implicit Function Theorem we have that

$$ \partial_t (\mathcal{R}^k(F_t))^{y^k} \big|_{t=0}(\beta^{k+1}_q)  + D(\mathcal{R}^kF)^{y^k}(\beta^{k+1}_q) \partial_t \beta^{k+1}_q = \partial_t \beta^{k+1}_q.$$
So
$$\partial_t  \beta^{k+1}_q$$
$$ = \frac{1}{1-D(\mathcal{R}^kF)^{y^k}(\beta^{k+1}_q) } \sum_{j=0}^{y^k-1} D(\mathcal{R}^kF)^{y^k-j-1}((\mathcal{R}^kF)^{j+1}(\beta^{k+1}_q) )v^k((\mathcal{R}^kF)^j(\beta^{k+1}_q) )$$
$$ = \sum_{\ell=1}^{\infty } \frac{-1}{D(\mathcal{R}^kF)^{y^k \ell}(\beta^{k+1}_q)} \sum_{j=0}^{y^k-1} D(\mathcal{R}^kF)^{y^k-j-1}((\mathcal{R}^kF)^{j+1}(\beta^{k+1}_q) )v^k((\mathcal{R}^kF)^j(\beta^{k+1}_q) )$$
$$ = \sum_{\ell=1}^{\infty } \sum_{j=0}^{y^k-1} -\frac{v^k((\mathcal{R}^kF)^{y^k(\ell-1) + j}(\beta^{k+1}_q) )}{D(\mathcal{R}^kF)^{y^k(\ell-1)+j+1}(\beta^{k+1}_q )}$$
$$= \sum_{a=0}^{\infty} -\frac{v^k((\mathcal{R}^kF)^a(\beta^{k+1}_q) )}{D(\mathcal{R}^kF)^{a+1}(\beta^{k+1}_q)}$$
Due the real bounds and $v \in \mathfrak{B}_+(F)$   there exist $C_1, C_2$ such that
$$\sup_k \{|D(\mathcal{R}^kF)|_{\mathcal{B}(V)} , \ |\beta^k_q|, \ \frac{1}{|\beta^k_q|}, \ |v^k|_{\mathcal{B}(V)} \} \leq C_1$$
and
$$1 < C_2 < \inf_{k,q} | D(\mathcal{R}^kF)^{y^k}(\beta^{k+1}_q)|.$$
So
\begin{equation}\label{der_beta} \sup_{k,q} |\frac{\partial_t \beta^{k+1}_q}{\beta^{k+1}_{q} }|= C_3 < \infty.\end{equation}
If $i_j$ satisfies $q^j_{i_j}=q$ then 
$$\delta^{k+1}_{i_{k+1}}(t) = \prod_{j=0}^{k} \frac{\delta^{j+1}_{i_{j+1}}(t)}{\delta^j_{i_j}(t)}= \prod_{j=1}^{k+1} \beta^j_q(t),$$
if we derive with respect to $t$ at $t=0$ we obtain
$$\partial_t \delta^{k+1}_{i_{k+1}} = \sum_{j=1}^{k+1} \partial_t \beta^j_q \prod_{\ell\neq j} \beta^\ell_q$$
and we conclude that
\begin{equation} \frac{\partial_t \delta^{k+1}_{i_{k+1}}}{\delta^{k+1}_{i_{k+1}}} = \sum_{j=1}^{k+1} \frac{\partial_t \beta^j_q }{ \beta^j_q }\end{equation}
and
\begin{equation}
\big| \frac{\partial_t \delta^{k+1}_{i_{k+1}}}{\delta^{k+1}_{i_{k+1}}} -\frac{\partial_t \delta^{k}_{i_k}}{\delta^{k}_{i_k}} \big| = \big|  \frac{\partial_t \beta^{k+1}_{q} }{ \beta^{k+1}_{q} } \big|\leq C_3
\end{equation}
\end{proof}

\begin{lem}\label{divv} For every $(z,j_i^k) \in \mathscr{R}_{-\ell}^{k+1}$, with $R_{-\ell}^{k+1} \subset Q_0^{k}$, $q=q_i^k$ we have 
\begin{equation} |V_1(z,j_i^k)| \leq |\delta_{i+1}^k| \sup_k |v^k|,\end{equation} 
and
\begin{equation} |V_2(z,j_i^k)|\leq  C_3 k    |\delta_{i+1}^k|  +   C_3 k   |\delta_{i}^k|  \sup |D G| \end{equation}
In particular
$$\lim_{w\rightarrow C(F)} V(w)=0.$$
\end{lem}
\begin{proof} It follows from Lemma \ref{div} and Lemma \ref{div2}.
\end{proof}

\begin{lem} \label{qvec} Let $\psi\colon \mathbb{R}_+^\star \rightarrow \mathbb{R}$ be a $C^\infty$ function. Define $\gamma\colon \mathbb{C}\setminus \{0\} \rightarrow \mathbb{C}$ as
$$\gamma(z)=\psi(|z|) z.$$
Then 
$$|\overline{\partial} \gamma(z)| =  \frac{|z D\psi(|z|)|}{2}.$$
If $|D \psi(t)|\leq  C t^{-1}$  then $\gamma$ can be extended as a quasiconformal  vector field on $\mathbb{C}$ such that $\gamma(0)=0$.
\end{lem}
\begin{proof} If $z = x+ iy$, $x, y \in \mathbb{R}$ then
$$\psi(|z|) = \psi(\sqrt{x^2 + y^2}),$$
In particular 
$$\overline{\partial} (\psi (|z|)) = \frac{2x D\psi(\sqrt{x^2+y^2})}{4 \sqrt{x^2+y^2}} + i \frac{2y D\psi(\sqrt{x^2+y^2})}{4 \sqrt{x^2+y^2}},$$
and
$$|\overline{\partial} (\psi (|z|)) |= \frac{|D\psi(|z|)|}{2}.$$
In particular
$$|\overline{\partial} (\psi(|z|) z)| =  \frac{|z D\psi(|z|)|}{2}.$$
Let $\epsilon \in (0,1)$. Note that 
$$|\psi(1)-\psi(\epsilon)|= |\int_{\epsilon}^1 D\psi(t)\ dt| \leq \int_{\epsilon}^1 \frac{1}{t} \ dt = -\ln \epsilon.$$ 
$$ |\psi(\epsilon)| \leq - \ln \epsilon + |\psi(1)|.$$
In particular
\begin{equation} \label{cont2} \lim_{z \rightarrow 0} z \psi(|z|) =0,\end{equation}
so defining $\gamma(0)=0$ we obtain a  continuous extension to $\mathbb{C}$ of $\gamma$. To show that $\gamma$ is a quasiconformal vector field, note that 
$$|\overline{\partial} \gamma(z)| =\frac{|zD\psi(|z|)}{2}\leq \frac{C}{2}$$
for every $z \in \mathbb{C}\setminus\{0\}$. By \cite[Lemma 3, page 53]{a_book} there exists a quasiconformal vector field $\tilde{\gamma}$ on $\mathbb{C}$ such that  its distributional derivative belongs to $L^2(\mathbb{C})$, it satisfies  $\overline{\partial} \tilde{\gamma}(z)= 0$ if $|z| \geq 1 $ and 
 $$\overline{\partial} \tilde{\gamma}(z)= \overline{\partial} \gamma(z)$$
 for almost every $z$ satisfying $|z|< 1$. So $\gamma - \tilde{\gamma}$ is continuous on $\{ z \in \mathbb{C}\colon \ |z| < 1\}$ and (by Weyl's  Lemma)  holomorphic on $\{ z \in \mathbb{C}\colon \  0< |z| < 1\}$. By Riemann's Theorem on removable singularities we have that  $0$ is a  removable singularity, so $\gamma - \tilde{\gamma}$ is holomorphic on $|z| < 1$, and therefore  $\gamma$ is quasiconformal on $\mathbb{C}$. 
\end{proof} 

An illustrative example of application of Lemma \ref{qvec} is obtained considering $\psi(x)= \log (x).$

\begin{prop} \label{alpha2}Let $v \in \mathfrak{B}_+(F)$. There exists a quasiconformal vector field  $\alpha_2\colon \mathbb{C}_n \rightarrow \mathbb{C}$ such that
\begin{equation}\label{v2} V_2(z,j) = \alpha_2\circ G (z,j) - DG(z,j)\cdot \alpha_2(z,j)\end{equation}
for every $(z,j) \in \partial \mathscr{D}$ and moreover $\alpha_2(0,j)=0$ for every $j$. 
\end{prop} 
\begin{proof} Due the real bounds we have that 

\begin{equation}\label{ep1} \inf_{q \in C(F)} \inf_k \min_{R^{k+1}_{-\ell}\subset Q^k_0} \frac{dist(R^{k+1}_{-\ell},\partial Q^k_0) }{|Q^{k}_0|}=\epsilon_1 > 0 .\end{equation} 
Without loss of generality we can choose $\epsilon_0$ in (\ref{esc_ep}) satisfiying  $\epsilon_0 < \epsilon_1/4$. Let $\phi\colon \mathbb{R} \rightarrow \mathbb{R}$ be a $C^\infty$ function such that 
\begin{itemize}
\item[i.] $\phi(x) \in [0,1]$ for every $x \in \mathbb{R}$.
\item[ii.] If  $|x| < 1-\epsilon_1/4$ then  $\phi(x)=1$.
\item[iii.] If $|x| > 1 - \epsilon_1/8$ then $\phi(x)=0$. 
\end{itemize}
Given $q \in C(F)$, let $i_j$ be such that $q^j_{i_j}=q$ and $\delta^j_{i_j}$ and $\beta_q^j$, $j \in \mathbb{N}$, be as in the  proof of Lemma \ref{div2}. For every $x \in \mathbb{R}^\star$ define
$$\psi_q(x) = \sum_{j=1}^\infty \frac{\partial_t \beta^j_q}{\beta^j_{q}} \cdot \phi(\frac{x}{\delta^{j}_{i_j}})$$
The function $\psi_q$ is well defined in $\mathbb{R}\setminus\{0\}$, it is $C^\infty$ on $\mathbb{R}\setminus \{0\}$ and  if 
\begin{equation} \label{cond_p}  (1-\epsilon_1/8) |\delta^{k+1}_{i_{k+1}}| \leq  |x|\leq (1-\epsilon_1/2) |\delta^k_{i_k}|\end{equation}
then 
\begin{equation} \label{piecewise} \psi_q(x) = \sum_{j=1}^k \frac{\partial_t \beta^j_q}{\beta^j_q} \cdot \phi(\frac{x}{\delta^{j}_{i_j}})=\sum_{j=1}^k \frac{\partial_t \beta^j_q}{\beta^j_q} = \frac{\partial_t \delta^k_{i_k}}{\delta^k_{i_k}}.\end{equation}
Moreover, notice that if
\begin{equation}\label{sel_x}   |\delta^{k+1}_{i_{k+1}}| \leq  |x|\leq |\delta^k_{i_k}| \end{equation} 
then 
$$\psi_q(x) = \frac{\partial_t \beta^{k}_q}{\beta^{k}_q} \cdot \phi(\frac{x}{\delta^{k}_{i_k}}) +\sum_{j=1}^{k-1} \frac{\partial_t \beta^j_q}{\beta^j_q},$$
so
$$D\psi_q(x)= \frac{\partial_t \beta^{k}_q}{\delta^k_{i_k} \beta^{k}_q} \cdot D\phi(\frac{x}{\delta^{k}_{i_k}}),$$
and by (\ref{der_beta}) and (\ref{sel_x})
\begin{equation} \label{eqvec}
|D\psi_q(x)| \leq \frac{1}{|\delta^k_{i_k}|}  C_3\max |D\phi| 
\leq  \frac{  C_3\max |D\phi| }{|x|}
\end{equation}

If $q=(0,j)$, define 
\begin{equation} \alpha_2(z,j)= \psi_{q}(|z|) z.\end{equation}

By (\ref{eqvec}) and Lemma \ref{qvec} we conclude that $\alpha_2$ is a quasiconformal vector field with $\alpha_2(0,j)=0$ for every $j$. 

It remains to show that $\alpha_2$ satisfies (\ref{v2}). Indeed, let  $(z,j)\in \partial \mathscr{R}_{-\ell}^{k+1}$, with $R_{-\ell}^{k+1} \subset Q_0^{k}$ and $\ell > 0$. Since  $\mathscr{R}_{-\ell}^{k+1}$ belongs to a disc with diameter given by the interval $R_{-\ell}^{k+1}$ and it does not intercept $\mathcal{Q}_0^{k+1}$, we have
\begin{equation}\label{dome}  (1- \epsilon_1/64)|\delta^{k+1}_{i_{k+1}}|   \leq d^{k+1}_q \leq |z|\leq (1-\epsilon_1) |\delta^k_{i_k}|.\end{equation}
for every $(z,j)\in \partial \mathscr{R}_{-\ell}^{k+1}$. Here $d^{k+1}_q$ is as defined in (\ref{dkr}). 
By (\ref{piecewise}) we have that 
$$\alpha_2(z,j)= \frac{\partial_t \delta^{k}_{i_k}}{\delta^{k}_{i_k}} z $$
for every
$$(z,j) \in \partial \mathscr{R}_{-\ell}^{k+1}.$$
Indeed, let  $(z,\tilde{\j})\in \mathscr{R}_{-\ell}^{k+1}$, with $R_{-\ell}^{k+1} \subset S_0^{k}$, with  $\ell > 0$ and such that $q$ is the sucessor of $s$ at level $k$. Then
$$G(\mathscr{R}_{-\ell}^{k+1}) \subset \mathcal{Q}_0^{k}.$$
If  $G(\mathscr{R}_{-\ell}^{k+1})=\mathscr{R}_{-a}^{k+1}$, for some $a > 0$,  then the points in this image also satisfies (\ref{dome}). Otherwise $r=q$ and  $G(\mathscr{R}_{-\ell}^{k+1})=\mathcal{Q}_0^{k+1}$, so if $(w,j) \in G(\partial \mathscr{R}_{-\ell}^{k+1})$ then 
\begin{equation}\label{dome1}  (1- \epsilon_1/64)|\delta^{k+1}_{i_{k+1}}|\leq  |d_q^{k+1}|=|w| \leq |\delta^{k+1}_{i_{k+1}}| \leq (1-\epsilon_1) |\delta^k_{i_k}|.\end{equation}
By (\ref{piecewise}) we have that 
$$\alpha_2(w,j)= \frac{\partial_t \delta^{k}_{i_k}}{\delta^{k}_{i_k}} w $$
for every
$$(w,j) \in  G(\partial \mathscr{R}_{-\ell}^{k+1}),$$
Since these estimates holds for every $q \in C(F)$ we have that for  every $(z,j)\in \partial \mathscr{R}_{-\ell}^{k+1}$, with $R_{-\ell}^{k+1} \subset Q_0^{k}$ and $\ell > 0$ we have 
\begin{eqnarray}
 \alpha_2\circ G (z,j) - DG(z,j)\cdot \alpha_2(z,j) 
 &=& \frac{\partial_t \delta^{k}_{i_k+1}}{\delta^{k}_{i_k+1}} G(z,j)-   DG(z,j) \frac{\partial_t \delta^{k}_{i_k}}{\delta^{k}_{i_k}}  (z,j) \nonumber \\
&=& \frac{\partial_t \delta^{k}_{i_k+1}}{\delta^{k}_{i_k+1}} F^{n^{k}_{i_k}}(z,j)-   DF^{n^{k}_{i_k}}(z,j) \frac{\partial_t \delta^{k}_{i_k}}{\delta^{k}_{i_k}}  (z,j) \nonumber \\
&=& V_2(z,j).
\end{eqnarray}
\end{proof}

\begin{prop}\label{alpha1} Let $v \in \mathfrak{B}_+(F)$. There exists a quasiconformal vector field  $\alpha_1\colon \mathbb{C}_n \rightarrow \mathbb{C}$ such that
\begin{equation}\label{v1} V_1(z,j) = \alpha_1\circ G (z,j) - DG(z,j)\cdot \alpha_1(z,j)\end{equation}
for every $(z,j) \in \partial \mathscr{D}$ and moreover $\alpha_1(q)=0$ for every $q \in C(F)$. 
\end{prop} 
\begin{proof} Let $\epsilon_1 > 0$ be as in (\ref{ep1}). 
Let $\phi\colon \mathbb{C} \rightarrow \mathbb{R}$ be a $C^\infty$ function such that 
\begin{itemize}
\item[i.] $\phi(x) \in [0,1]$ for every $x \in \mathbb{C}$.
\item[ii.] If  $|x|=1 -\epsilon_0/16$ and then  $\phi(x)=1$.
\item[iii.] If either $|x| > 1-\epsilon_0/32 $, $|x| < 1 - \epsilon_1/2$ then $\phi(x)=0$. 
\end{itemize}
Given $r \in C(F)$, let  
$$i_0=0 < i_1 < \dots < i_{\ell^k_r} < n^k_r$$ 
be the sequence of integers $i < n^k_r$ such that 
$$R^k_{i} \subset \bigcup_{q \in C(F)} Q^{k-1}_0.$$
Denote by $\delta^k_r$ the periodic point in the boundary of $R_0^k$. There is an onto univalent  extension 
$$G^j \colon D_{r,-j}^k\rightarrow B(r,|\delta^k_r|),$$
where $D_{r,-j}^k\cap \mathbb{R} =R_{-i_j}^k$.  Note that 
$\mathscr{R}_{-i_j}^k \subset D_{r,-j}^k$ and
\begin{equation} \label{zero} G^j(\mathscr{R}_{-i_j}^k)=B(r,(1-\epsilon_0/16)|\delta^k_r|)=\mathscr{R}_{0}^k.\end{equation} 
Let $q_j \in C(F)$ such that $q_j$ and $R^k_{-i_j}$ belongs to the same connected component  of $I^n_F$. Define a function
$$\psi^k_{r,-j}\colon \mathbb{C}_n \rightarrow \mathbb{C}$$
in the following way.  Let  $$\psi^k_{r,0}(z,a)=0$$
for every $(z,a) \in \mathbb{C}_n$, and define by induction on $j$  
\begin{equation}\label{tce_zero} \psi^k_{r,-(j+1)}(z,a)=\phi\Big(\frac{G^{j+1}(z,a)}{\delta_r^k}\Big) \frac{\psi^k_{r,-j}(G(z,a)) - \delta_{q_{j}}^{k-1}v^{k-1}\circ A_{q_{j+1}}^{k-1}   (z,a)}{DG(z,a)}  \end{equation} 
for every $(z,a)\in D^k_{r,-(j+1)}$. Note that $\psi^k_{r,-(j+1)}(z,a)=0$ for $(z,a)$ in a neighbourhood of $\partial D^k_{r,-(j+1)}$, so we can extend $\psi^k_{r,-(j+1)}$ to a $C^\infty$ function on $\mathbb{C}_n$ defining $\psi^k_{r,-(j+1)}(z,a)=0$ for $(z,a) \not\in  D^k_{r,-(j+1)} $. 
\noindent Finally note that by  (\ref{zero}) 
$$\phi\Big(\frac{G^{j+1}(z,a)}{\delta_r^k}\Big)  =1$$ 
for every $(z,a) \in \partial \mathscr{R}_{-i_{j+1}}^k$, so by (\ref{tce_zero})
 \begin{equation}\label{tce_um} -V_1(z,a)= \delta_{q_{j}}^{k-1}v^{k-1}\circ A_{q_{j+1}}^{k-1}   (z,a) =\psi^k_{r,-j}(G(z,a)) - DG(z,a) \psi^k_{r,-(j+1)}(z,a)   \end{equation} 
 for every $(z,a) \in \partial \mathscr{R}_{-i_j}^k$. Note that  given $(z,a) \in \mathbb{C}_n\setminus C(F)$, there exists an open neighbourhood of $(z,a)$ that intersects only one of the supports of  a function in the family
 $$\mathcal{F}= \{   \psi^k_{r,-j}    \}_{k, \ r \in C(F), j \leq \ell^k_r}.$$
 In particular the function 
 $$\alpha_1(z,a)=-\sum_k \sum_{r  \in C(F)} \sum_{j=0}^{\ell^k_r} \psi^k_{r,-j}(z,a) $$
is well defined and it is $C^\infty$ on $\mathbb{C}_n\setminus C(F)$. Moreover  given $\mathscr{R}_{-i_j}^k$, the function $\psi^k_{r,-j}$ is the unique function in this family whose support intersects  $\partial \mathscr{R}_{-i_j}^k$.   By (\ref{tce_um}), this implies that
 $$V_1(z,a)= \alpha_1 \circ G(z,a)) - DG(z,a) \alpha_1(z,a).$$  
 for every $(z,a) \in \partial \mathscr{D}$.  It is easy to prove by induction that 
\begin{equation} \label{cf2} |\psi^k_{r,-j}(z,a)| \leq   \sup_k |v^{k-1}| \sum_{\ell=0}^{j-1} \frac{|\delta^{k-1}_{q_{j-1-\ell}}|}{|DG|^{\ell+1}(z,a)},\end{equation}
in particular 
\begin{equation} \label{cf} |\psi^k_{r,-j}| \leq  \max_{q \in C(F)} |\delta^{k-1}_q| \sup_k |v^{k-1}| \sum_{j=0}^{\sup_k \frac{p_k}{p_{k-1}} } \frac{1}{(\inf |DG|)^j}.\end{equation}
Due the real bounds
\begin{equation}\label{zerok}  |\psi^k_{r,-j}| \leq  C \theta^{-k}\end{equation}
If $(z,a) \in \mathcal{Q}^{k-1}_0\setminus \{ q\}$, for some $q \in C(F)$ then either $\alpha_1(z,a)=0$ or $(z,a)$ belongs to the support of a unique function $\psi^b_{r,-j}$, with $b \geq k$. We conclude that 
\begin{equation}\label{cont1}  \lim_{(z,a)\rightarrow q} \alpha_1(z,a)=0,\end{equation}
so we can extend $\alpha_1$ as a continuous function on $\mathbb{C}_n$.
Note that
\begin{eqnarray}\label{dbar_zero} &\ & \overline{\partial} \psi^k_{r,-(j+1)}(z,a) \nonumber \\
&=& \overline{\partial}\phi\Big(\frac{G^{j+1}(z,a)}{\delta_r^k}\Big) \frac{\overline{DG^{j+1}}(z,a)}{\overline{\delta_r^k}} \frac{\psi^k_{r,-j}(G(z,a)) - \delta_{q_{j}}^{k-1}v^{k-1}\circ A_{q_{j+1}}^{k-1}   (z,a)}{DG(z,a)} \nonumber \\  
&+& \phi\Big(\frac{G^{j+1}(z,a)}{\delta_r^k}\Big) \overline{\partial}  \psi^k_{r,-j}(G(z,a)) \frac{\overline{DG}(z,a)}{DG(z,a)}
 \end{eqnarray} 
 By the real bounds there exists $C > 0$ such that 
 $$ \frac{|DG^{j-\ell}|(G^{\ell+1}(z,a))}{|\delta_r^k|} \leq  \frac{C}{|R^k_{-i_{j-\ell}}|}$$
 for every $r \in C(F)$, $K$ and $j$ and $(z,a) \in D^k_{r, -(j+1)}$. In particular
\begin{eqnarray} &&\frac{|DG^{j+1}|(z,a)}{|\delta_r^k|} \frac{|\psi^k_{r,-j}(G(z,a)) - \delta_{q_{j}}^{k-1}v^{k-1}\circ A_{q_{j+1}}^{k-1}   (z,a)|}{|DG(z,a)|}   \nonumber \\
&\leq&  \frac{|DG^{j+1}|(z,a)}{|\delta_r^k|}\Big( \sum_{\ell=0}^{j-1} \frac{|\delta^{k-1}_{q_{j-1-\ell}}|}{|DG|^{\ell+2}(z,a)}+  \frac{|\delta_{q_{j}}^{k-1}|}{|DG|(z,a)}\Big) \nonumber \\
&\leq&  \frac{|DG^{j+1}|(z,a)}{|\delta_r^k|} \sum_{\ell=0}^{j} \frac{|\delta^{k-1}_{q_{j-\ell}}|}{|DG|^{\ell+1}(z,a)}\nonumber \\
&\leq& \sum_{\ell=0}^{j}   \frac{|DG^{j-\ell}|(G^{\ell+1}(z,a))|\delta^{k-1}_{q_{j-\ell}}|}{|\delta_r^k|} \leq C\sum_{\ell=0}^{j}    \frac{|\delta^{k-1}_{q_{j-\ell}}|}{|R^k_{-i_{j-\ell}}|} \nonumber \\
\label{est55} &\leq& C \sup_k \frac{\mu^{(k)}}{\mu^{(k-1)}}. \end{eqnarray}
The last inequality follows from the fact that $R^k_{-i_{j-\ell}}\subset Q^{k-1}_0$, with $q = q_{j-\ell}$, and $2\delta^{k-1}_{q_{j-\ell}} = |Q^{k-1}_0|$ and by the real bounds there exists  $C > 0$ such that 
$$\frac{|Q_0^{k-1}|}{|R_{-i}^k|}\leq C$$
for every $R_{-i}^k \subset Q^{k-1}_0$.
 So by (\ref{est55})
 \begin{equation}\label{dbar_zero_est} \sup |\overline{\partial} \psi^k_{r,-(j+1)}|
\leq  C  \sup  |\overline{\partial}\phi|  \sup_k \frac{\mu^{(k)}}{\mu^{(k-1)}} +\sup |\overline{\partial}  \psi^k_{r,-j}|.
 \end{equation} 
Since $\psi^k_0=0$ we obtain 
$$\sup  |\overline{\partial} \psi^k_{r,-j}|\leq  C \sup_k \frac{\mu^{(k)}}{\mu^{(k-1)}}$$
Since for every  $(z,a) \in \mathbb{C}_n\setminus C(F)$, there exists an open neighbourhood of $(z,a)$ that intersects only one  of the supports of  the functions in the family $\mathcal{F}$, we  conclude that $\alpha_1$ is a quasiconformal vector field on $\mathbb{C}_n\setminus C(F)$. Since $\alpha_1$ is continuous at $C(F)$, we can use the argument in the end of the proof of Lemma \ref{qvec} to conclude that $\alpha_1$ is a quasiconformal vector field on $\mathbb{C}_n$. 
\end{proof}  

\begin{proof}[Proof of Theorem \ref{keylem}] Let $v \in \mathfrak{B}_+(F)$. Let $\alpha_1$ and $\alpha_2$ as in Proposition \ref{alpha2} and \ref{alpha1}. Define $\alpha^0 = \alpha_1 + \alpha_2$. Then by Lemma \ref{div}
\begin{equation}\label{v} V(z,j) = \alpha^0\circ G (z,j) - DG(z,j)\cdot \alpha^0(z,j)\end{equation}
for every $(z,j) \in \partial \mathscr{D}\setminus C(F)$ and moreover $\alpha^0(q)=0$ for every $q \in C(F)$. Now we will use  an argument similar to the infinitesimal pullback argument for polynomial-like maps \cite{alm}. We define  by induction on $m$ a sequence of quasiconformal vector fields $$\alpha^m\colon \mathbb{C}_n \rightarrow \mathbb{C}$$
such that $\alpha^m(q)=0$ for every $q \in C(F)$,
\begin{equation}\label{v5} V(z,j) = \alpha^m\circ G (z,j) - DG(z,j)\cdot \alpha^m(z,j)\end{equation}
for every $(z,j) \in \partial \mathscr{D}$ and moreover
\begin{equation}\label{v6} V(z,j) = \alpha^{m-1}\circ G (z,j) - DG(z,j)\cdot \alpha^{m}(z,j)\end{equation}
for every $(z,j) \in \mathscr{D}$, $m\geq 1$ and
$$\sup_{\mathbb{C}_n} |\overline{\partial} \alpha^{m+1}|\leq  \sup_{\mathbb{C}_n} |\overline{\partial} \alpha^{m}|.$$
 Indeed, suppose by induction we have defined $\alpha^{k}$. Define $\alpha^{k+1}(z,j)=\alpha^{k}(z,j)$ for every $(z,j)\not\in \mathscr{D}$ and 
\begin{equation}\label{m23} \alpha^{k+1}(z,j)=\frac{\alpha^{k}(G(z,j))-V(z,j)}{DG(z,j)}\end{equation}
for every $(z,j) \in \mathscr{D}$.  Note that 
$$|\overline{\partial} \alpha^{k+1}(z,j)|= |\overline{\partial} \alpha^{k}(z,j)|,$$
for every $(z,j) \not\in \overline{\mathscr{D}}$ and 
$$|\overline{\partial} \alpha^{k+1}(z,j)|= |\overline{\partial} \alpha^{k}(G(z,j))|$$ 
for $(z,j)\in \mathscr{D}$. So $\alpha^{k+1}$ is a quasiconformal vector field in $\mathbb{C}_n\setminus \partial \mathscr{D}$. 
Moreover  (\ref{v6}) holds for $m=k+1$ and due (\ref{v5}) with $m=k$ we have that $\alpha^{k+1}=\alpha^k$ on $\partial \mathscr{D}\setminus C(F)$, so $\alpha^{k+1}$ is continuous at points in $$\partial \mathscr{D}\setminus C(F).$$ 
If $(z,j) \in   \partial \mathscr{D}\setminus C(F)$ then $(z,j) \in \partial \mathscr{R}^n_{-i}$, for some $r \in C(F)$. In particular
there exists a  neighbourhood $W$ of $(z,j)$ such that $\alpha^{k+1}$ is continuous on $W$ and a quasiconformal vector field 
on $W\setminus \partial \mathscr{R}^n_{-i}$.  Since  $\partial \mathscr{R}^n_{-i}$ is an analytic curve we conclude that $\alpha^{k+1}$ is a quasiconformal vector field on $W$. So $\alpha^{k+1}$ is a quasiconformal vector field on $\mathbb{C}_n\setminus C(F)$. Finally notice that $\alpha^{k+1}$ is continuous at points in $C(F)$. Indeed, suppose that 
$$(z_\ell,i)\rightarrow_\ell (0,i).$$
If $(z_\ell,i) \not\in \mathscr{D}$ for every $\ell$ then
$$\lim_{\ell\rightarrow \infty} \alpha^{k+1}(z_\ell,i)= \lim_{\ell\rightarrow \infty} \alpha^{k}(z_\ell,i) =0.$$
If $(z_\ell,i) \in \mathscr{D}$ for every $\ell$ then
\begin{equation} \label{lim3} \lim_{\ell\rightarrow \infty} \alpha^{k+1}(z_\ell,i)= \lim_{\ell\rightarrow \infty} \frac{\alpha^{k}(G(z_\ell,i))-V(z_\ell,i)}{DG(z_\ell,i)}.\end{equation}
Since the accumulation points of the sequence 
$G(z_\ell,i)$ belongs to $C(F)$ we have
$$\lim_{\ell\rightarrow \infty} \alpha^{k}(G(z_\ell,i)) =0.$$
By Lemma \ref{divv}  it follows that 
$$\lim_{\ell\rightarrow \infty} V(z_\ell,i) =0.$$
Since by the complex  bounds we have that 
$$\inf_{(z,j) \in \mathscr{D}} |DG(z,j)| > 0$$
we conclude by (\ref{lim3}) that 
$$ \lim_{\ell\rightarrow \infty} \alpha^{k+1}(z_\ell,i)=0.$$
so $\alpha^{k+1}$ is continuous at points in $C(F)$. By the same argument in the end of the proof of Lemma \ref{qvec} we conclude that $\alpha^{k+1}$ is a quasiconformal vector field on $\mathbb{C}_n$ and
\begin{equation} \label{lim_qc} \sup_{\mathbb{C}_n} |\overline{\partial} \alpha^{k+1}|\leq  \sup_{\mathbb{C}_n} |\overline{\partial} \alpha^{k}|.\end{equation} 
Given a  point $(z,j) \in \mathbb{C}_n$ such that  there exists $k_0\geq 0$ such that $G^k(z,j) \in \mathscr{D}$ for every $k< k_0$ and $G^{k_0}(z,j) \not\in \mathscr{D}$, we claim that $\alpha^k(z,j)=\alpha^{k_0}(z,j)$ for every $k\geq k_0$.  Note that $\alpha^k(z,j)=\alpha^0(z,j)$ for every $(z,j) \not\in \mathscr{D}$, so the claim holds for $k_0=0$. Suppose by induction on $k_0$ that the claim hold for $k_0$. If $G^i(z,j) \in \mathscr{D}$ for $i\leq k_0$ and $G^{k_0+1}(z,j) \not\in \mathscr{D}$ then $G^i(G(z,j)) \in \mathscr{D}$ for $i < k_0$ and $G^{k_0}(G(z,j))\not\in \mathscr{D}$, so by the induction assumption $\alpha^{k}(G(z,j))=\alpha^{k_0}(G(z,j))$ for every $k\geq k_0$. Since $G(z,j) \in \mathscr{D}$ we have by (\ref{m23}) that 
$$\alpha^{k+1}(z,j)=\frac{\alpha^{k}(G(z,j))-V(z,j)}{DG(z,j)}= \frac{\alpha^{k_0}(G(z,j))-V(z,j)}{DG(z,j)}=\alpha^{k_0+1}(z,j),$$
for every $k\geq k_0$, which proves the claim. \\
By Proposition  \ref{pulafora}  the  sequence $\alpha^k$ converges almost everywhere.  By (\ref{lim_qc})  and McMullen \cite{mc2} we have that every subsequence $\alpha^k$ has a subsequence that converges uniformly to some quasiconformal vector field.  So $\alpha^k$ converges uniformly to a quasiconformal vector field $\alpha$. Taking $m\rightarrow \infty$ in (\ref{v6}) we conclude that 
\begin{equation} V(z,j) = \alpha\circ G (z,j) - DG(z,j)\cdot \alpha(z,j),\end{equation}
for every $(z,j) \in \mathscr{D}$ and in particular for every $(z,j) \in P(F)\setminus C(F)$. Since $\alpha(q)=0$ for every $q \in C(F)$, by Corollary \ref{main_cor} we conclude that $v \in E^h(F)$.
 \end{proof}

\section{Transversal families have  hyperbolic parameters.}\label{teisection}
From now on we will consider only {\it real maps}. In particular $\mathcal{B}_{nor}^{\mathbb{R}}(U)$ will now denote the real Banach space of all $F\in \mathcal{B}_{nor}(U)$ which are real on the real line.  Since  $\Omega_{n,p}$ is a hyperbolic set we can define the stable 
$$G\mapsto E^h_G$$
and unstable 
$$G\mapsto E^u_G$$ 
subspace distributions defined for $G \in \Omega_{n,p}$, and the corresponding projections on the spaces
$$\pi_G^h \colon T \mathcal{B}_{nor}^{\mathbb{R}}(U) \rightarrow E^h_G$$
and
$$\pi_G^u \colon T \mathcal{B}_{nor}^{\mathbb{R}}(U) \rightarrow E^u_G.$$
Recall that $U=D_{\delta_0,\theta_0}$. As defined in \cite{sm5} we also have the adapted norms $|\cdot|_{G,0}$ that satisfy
 $$|v|_{G,0}= |\pi_G^h(v)|_{G,0}+  |\pi_G^u(v)|_{G,0},$$
and  the family of cones $C^u_{\epsilon}(G)$, with $\epsilon > 0$, for which  $v \in  C^u_\epsilon(G)$ if and only if 
$$|\pi_G^h(v)|_{G,0}\leq \epsilon  |\pi_G^u(v)|_{G,0}.$$
Those cones are  unstable and forward-invariant for the action of $\mathcal{R}$ on  $\Omega_{n,p}$ provided  $\epsilon$ is  small enough. In particular if $\epsilon$ is small there is $\theta > 1$ such that for every $ F \in \mathcal{B}_{nor}^{\mathbb{R}}(U)$ close enough to  some $G \in \Omega_{n,p}$, and $v \in C^u_{2\epsilon}(G)$ 
$$|D\mathcal{R}_{F}\cdot v|_{\mathcal{R}G,0}\geq  \theta |v|_{G,0}.$$
and moreover there is $\epsilon' \in (0,\epsilon)$ such that 
$$|\pi_{\mathcal{R}G}^h(D\mathcal{R}_F\cdot v)|_{\mathcal{R}G,0}\leq 2\epsilon'  |\pi_{\mathcal{R}G}^u(D\mathcal{R}_F\cdot v)|_{\mathcal{R}G,0}.$$
Furthermore  there is $\lambda \in (0,1)$   such that for every $G \in \Omega_{n,p}$, $v \in E^h_G$ and $k \in \mathbb{N}$ 
$$|D\mathcal{R}_{G}^k\cdot v|_{\mathcal{R}^kG,0}\leq \lambda^k |v|_{G,0}.$$

 Define the $\delta$-shadow of $G$ as
$$W^s_\delta(G)= \{F \in \mathcal{W} \colon \ dist_{\mathcal{B}_{nor}^{\mathbb{R}}(U)}(\mathcal{R}^k F,  \mathcal{R}^k G)\leq \delta, \ for \ every \ k\geq 0 \},$$ 
and the $\delta$-shadow of $\Omega_{p,n}$ as
$$W^s_\delta(\Omega_{p,n})=\cup_{G \in \Omega_{p,n}} W^s_\delta(G).$$
We also define 
$$\mathbb{B}^u_G(v_0,\delta) = \{v \in E^u_G\cap \mathcal{B}^{\mathbb{R}}(U) \colon \ |v-v_0|_{G,0} \leq \delta \},$$
where $v_0\in E^u_G\cap \mathcal{B}^{\mathbb{R}}(U)$,
and 
$$E^h_G+G = \{v+G\colon \ v \in E^h_G\}.$$ 
 Let  $\delta_3  > 0$ (we will use this notation to follow \cite{sm5}). Define  $\mathcal{T}^1_0(G,\delta,\epsilon)$,  with $\delta \in (0, \delta_3)$, as  the set of $C^1$ functions 
$$\mathcal{H}\colon \mathbb{B}^u_G(v_0,\delta) \rightarrow E^h_G+G,$$
with $v_0\in E^u_G\cap \mathcal{B}^{\mathbb{R}}(U)$, such that 
$$|D\mathcal{H}|_{(G,0),(G,0)}\leq \epsilon.$$
and
$$F_0 = v_0 +\mathcal{H}(v_0) \in W^s_{\delta_3}(G).$$
We will call $F_0$ the {\it base  point} of $\mathcal{H}$. In particular $w+ D_x\mathcal{H}\cdot w \in C^u(G)$, for every $x \in \mathbb{B}^u_G(v_0,\delta)$, $w \in E^u(G)$, and  $G \in \Omega_{p,n}$.

Denote
$$\hat{\mathcal{H}}=\{ v+\mathcal{H}(v)\colon \ v \in \mathbb{B}^u_G(v_0,\delta)\}.$$

 The Transversal Empty Interior assumption  for the renormalization operator is the main result of this section.

\begin{cor}[Transversal Empty Interior Assumption]\label{tei}  For every small $\epsilon > 0$  we can choose $\delta_3$ small enough such that the following holds.   For every $G\in \Omega_{p,n}$ and for every $C^{1+Lip}$  function $\mathcal{H}\in  \mathcal{T}^1_0(G,\delta',\epsilon)$, with $\delta' < \delta_3$, we have that $\hat{\mathcal{H}}\cap W^s_{\delta_3}(\Omega)$ has empty interior in $\hat{\mathcal{H}}$.
\end{cor}

This property  is closely  related with the fact that maps $F$ that are infinitely renormalizable with bounded combinatorics can be approximated by hyperbolic maps.  

We are going to introduce notation from \cite{sm5}. Denote by  $C^{\omega_{\mathbb{R}}}([-1,1]^j,\mathcal{B}_{nor}^{\mathbb{R}}(U))$ the  space of    functions 
$$\gamma\colon (-1,1)^j \rightarrow \mathcal{B}_{nor}^{\mathbb{R}}(U)$$ 
which  can be extended to a complex analytic function  
$$\gamma\colon \mathbb{D}^j \rightarrow \mathcal{B}_{nor}(U),$$ 
and moreover there is  a continuous extension of $\gamma$  to $\overline{\mathbb{D}}^j$. Endowed with the sup norm on $\overline{\mathbb{D}}^j$ the space  $C^{\omega_{\mathbb{R}}}([-1,1]^j,\mathcal{B}_{nor}^{\mathbb{R}}(U))$ is a real Banach space.  

Endow $T_{\mathbb{C}}=\overline{\mathbb{D}}^{\mathbb{N}}$ with the product topology. Let $\Gamma^\omega(\mathcal{B}_{nor}(U))$   be the set of continuous  functions $\gamma\colon T_{\mathbb{C}} \mapsto \mathcal{B}_{nor}(U)$ which  are holomorphic when we fix all but a finite number of entries of $\lambda \in T_{\mathbb{C}}$ and $|\lambda_i| < 1$ for every $i$.  Endowing  $\Gamma^\omega(\mathcal{B}_{nor}(U))$ with the sup norm we obtain  a complex Banach space.

Note that since $U$ is symmetric with respect to the real line, that is, $(z,i)\in U$ iff $(\overline{z},i) \in U$, there is a complex conjugation on the complex Banach space $\mathcal{B}(U)$ defined by $\overline{f}(z)=\overline{f(\overline{z})}$ for $f \in \mathcal{B}(U)$. Define $\Gamma^{\omega_{\mathbb{R}}}(\mathcal{B}_{nor}^{\mathbb{R}}(U))$ as the {\it real } Banach space  that consists of the restrictions to $T=[-1,1]^{\mathbb{N}}$ of  functions $\gamma \in \Gamma^{\omega}(\mathcal{B}_{nor}(U))$ satisfying $\gamma(\overline{\lambda})= \overline{\gamma(\lambda)}$.

 We say that a set   $\Theta \subset \mathcal{B}_{nor}(U)$ is a {\it $\Gamma^{\omega_{\mathbb{R}}}(\mathcal{B}_{nor}^{\mathbb{R}}(U))$-null set}  is there exists a residual subset $\mathcal{F} \subset \Gamma^{\omega_{\mathbb{R}}}(\mathcal{B}_{nor}^{\mathbb{R}}(U))$ such that 
 $$m(\lambda \in [-1,1]^\mathbb{N}\colon \gamma(\lambda)\in \Theta)=0$$ for every $\gamma \in \mathcal{F}$. Here  $m$ is the product measure obtained considering the normalised  Lebesgue measure on each copy of  $[-1,1]$.

The Transversal Empty Interior property will allows us to   apply \cite[Theorem  1]{sm5} to the renormalization operator. Indeed  we already verified that 

\begin{itemize}
\item[-] $\mathcal{R}\colon  \mathcal{W}^{\mathbb{R}} \rightarrow  \mathcal{B}_{nor}^{\mathbb{R}}(U)$ is a real analytic map. Here $\mathcal{W}^{\mathbb{R}}=\mathcal{W}\cap \mathcal{B}_{nor}^{\mathbb{R}}(U)$.
\item[-] The map $\mathcal{R}$ is a strongly compact operator (Remark \ref{sc}),
\item[-] $\Omega_{n,p}$ is a hyperbolic set (Theorem \ref{hip1}),
\item[-]  For every $F \in \mathcal{R}^{-i}W^s_\delta(\Omega_{n,p})$, with $i\in \mathbb{N}$ and $\delta > 0$,  we have that 
$D_F\mathcal{R}^i(T_{F}\mathcal{B}_{nor}^{\mathbb{R}}(U))$ is dense in $T_{\mathcal{R}^iF}\mathcal{B}_{nor}^{\mathbb{R}}(U)$. This is  an easy  consequence  of Theorem  \ref{dense}. 
\end{itemize}
so  \cite[Theorem  1]{sm5}  in our setting boils down to   
 
 \begin{thm}{( \cite[Theorem  1]{sm5})} Suppose that the renormalization operator $\mathcal{R}$ satisfies  additionally  \begin{itemize}
\item[A.] There exists $\delta_3 > 0$   such that $W^s_{\delta_3}(\Omega_{n,p})$ satisfies  the  Transversal Empty Interior assumption.
\end{itemize}
Then 
$W^s(\Omega_{n,p})$  is a  $\Gamma^{\omega_{\mathbb{R}}}(\mathcal{B}_{nor}^{\mathbb{R}}(U))$-null set. Indeed  for every $j$  there exists a residual set of  real-analytic  maps $\gamma\in C^{\omega_{\mathbb{R}}}([-1,1]^j,\mathcal{B}_{nor}^{\mathbb{R}}(U))$ such that 
$$m(t \in [-1,1]^j \colon \ \gamma(t) \in W^s(\Omega_{n,p}))=0.$$
Here $m$ is  the  Lebesgue measure on $[-1,1]^j$.  \end{thm} 

In particular this implies  that  a generic  real-analytic finite-dimensional   family  in $\mathcal{W}^{\mathbb{R}}$ intersects $W^s(\Omega_{n,p})$ on a subset with zero Lebesgue measure. So we have a version of Theorem A. for real-analytic families of extended maps  that {\it belong} to $\mathcal{W}^{\mathbb{R}}$. Indeed  the full-blown  version of Theorem A. is proven in Section \ref{multimodalfamilies}.

\begin{prop} \label{nuca} For every $\epsilon > 0$ small enough there is $\gamma > 0$ with the following property. Suppose that  $F \in \mathcal{B}_{nor}^{\mathbb{R}}(U)$, that  $F$ has a polynomial-like extension  $F\colon \hat{V}^0 \rightarrow \hat{V}^1$, with $\overline{U}\subset \hat{V}^0$,  and 
$$dist_{\mathcal{B}_{nor}^{\mathbb{R}}(U)}(F,G) < \gamma,$$
for some $G \in \Omega_{p,n}$. If  $$v\in \hat{E}^h_F\cap C^u_{2\epsilon}(G)$$ and $v \in \mathcal{B}^{\mathbb{R}}(\hat{V}^0)\cap \mathcal{B}^{\mathbb{R}}(U)$ then $v=0$. 
\end{prop}
\begin{proof} Suppose by contradiction that there exist  sequences $G_i \in \Omega_{p,n}$, $F_i \in \mathcal{B}_{nor}^{\mathbb{R}}(U)$ and $v_i \in \mathcal{B}^{\mathbb{R}}(\hat{V}^0_i )$ such that  
\begin{itemize}
\item[-] We have
$$dist_{\mathcal{B}_{nor}^{\mathbb{R}}(U)}(F_i,G_i)\rightarrow_i 0,$$
\item[-] The maps $F_i$ have a polynomial-like of type $n$ extension $F_i\colon \hat{V}_i^0 \rightarrow \hat{V}_i^1$ and $\overline{U}\subset \hat{V}^0_i$.
\item[-] The vectors satisfy $v_i \in \hat{E}^h_{F_i}\cap C^u_{2\epsilon}(G_i)$,  $|v_i|_{\mathcal{B}^{\mathbb{R}}(U)}\neq 0$ and $v_i \in \mathcal{B}^{\mathbb{R}}(\hat{V}_i^0)$.
\end{itemize}
In particular for large $i$  the critical points of $F_i$ belongs to $K(F_i)$ and $F_i$ is renormalizable. Without loss of generality  we can assume that $|v_i|_{\mathcal{B}(U)}=1$ for every $i$.
Since $v_i \in \hat{E}^h(F_i)$ and $F_i$ are very close to $\Omega_{n,p}$ we have that  $\mathcal{R}F_i$ has a polynomial-like extension of type $n$
$$\mathcal{R}F_i\colon V_i^0 \rightarrow V_i^1$$
with $\mod V_i^1\setminus V_i^0 > \epsilon_0$ . Moreover $D_{F_i}\mathcal{R}\cdot v_i \in \hat{E}^h_{\mathcal{R}F_i}\cap \mathcal{B}^{\mathbb{R}}(V_i^0)$ and there is $C > 0$ such that 
$$|D_{F_i}\mathcal{R}\cdot v_i|_{\mathcal{B}^{\mathbb{R}}(V_i^0)}\leq C$$
for every large $i$. 
 Note that  $\mathcal{R}F_i\colon V_i^0 \rightarrow V_i^1$ is real on the real line and consequently  it  is hybrid conjugate with a real polynomial of type $n$ (see the Straightening lemma  in \cite[Proposition 4.1]{sm2}). It follows from Shen \cite{shenet} that $\mathcal{R}F_i$ does not have invariant line fields on its Julia set. So one can use the infinitesimal pullback argument  to conclude that there exist quasiconformal vector fields $\alpha_i\colon \mathbb{C}_n \rightarrow \mathbb{C}$ with $\sup_i |\overline{\partial} \alpha_i|_\infty < \infty$ such that 
\begin{equation} \label{lim_i} D_{F_i}\mathcal{R}\cdot v_i = \alpha_i \circ \mathcal{R} F_i - D(\mathcal{R}F_i)\cdot \alpha_i\end{equation}
on a domain a little bit  smaller than $V_i^0$, and  in particular on $U=D_{\delta_0,\theta_0}$.   By a  compactness argument \cite{mc2}  the sequence $\alpha_i$ has a convergent subsequence that converges to  a quasiconformal vector field $\alpha$. Let 
$$D_{\delta}= \{ x \in \mathbb{C}: \  dist(x,[-1,1]) < \delta\}\times \{1,\dots,n\}.$$ 
Note that by (\ref{delta_est})  there exists $\delta' > \delta_0$ such that  $\overline{U} \subset D_{\delta'} \subset V_i^0$ for every $i$. Since $$|D_{F_i}\mathcal{R}\cdot v_i|_{\mathcal{B}(D_{\delta'})}\leq C,$$
there exists a subsequence of $D_{F_i}\mathcal{R}\cdot v_i$  that converges to some $v$ on $\mathcal{B}^{\mathbb{R}}(U)$ satisfying 
$|v|_{\mathcal{B}(U)}\leq C$.  Since $\Omega_{p,n}$ is compact, without loss of generality we can assume that $\mathcal{R} G_i$  (and so $\mathcal{R}  F_i$) converges to some $G \in \Omega_{p,n}$. By  (\ref{lim_i}) we obtain
$$v =\alpha\circ G - DG\cdot \alpha$$
on the pos-critical set of $G$. Since  there are not invariant line fields supported in the Julia set of $G$, by the infinitesimal pullback argument we conclude that $v \in E^h_G$.  
In particular
\begin{equation}\label{pop1} |D\mathcal{R}_{G}^k\cdot v|_{\mathcal{R}^k(G),0}\leq \lambda^k |v|_{G,0} \leq C_1\lambda^k.\end{equation}
On the other hand since $v_i \in C^u_{2\epsilon}(G_i)$ and $\epsilon$ is small we have 
$$|D_{G_i}\mathcal{R}\cdot v_i |_{\mathcal{R}(G_i),0}\geq \theta| v_i |_{G_i,0}\geq C_2  \theta |v_i|_{\mathcal{B}^{\mathbb{R}}(U)}= C_2 \theta > 0,$$
The compactness of $\Omega_{p,n}$  gives   $\lim_i D_{G_i}\mathcal{R}\cdot v_i =v$ and  consequently for every $k\geq 1$
$$\lim_i D\mathcal{R}_{G_i}^k\cdot v_i=D\mathcal{R}_{G}^{k-1}\cdot v$$
and we  have that $v_i \in C^u_{2\epsilon}(G_i)$ so  for $k\geq 1$ 
$$|D\mathcal{R}_{G_i}^k\cdot v_i|_{\mathcal{R}^k(G_i),0}\geq  \theta^{k-1}  |D\mathcal{R}_{G_i}\cdot v_i|_{\mathcal{R}G_i,0}\geq C_2 \theta^k.$$   
Taking the limit on $i$ we obtain 
\begin{equation}\label{pop2} |D\mathcal{R}_{G}^k\cdot v|_{\mathcal{R}^k(G),0}\geq C_2 \theta^{k}.\end{equation}
Since $\lambda < 1 < \theta$ we conclude that  (\ref{pop1}) and  (\ref{pop2}) give us a contradiction.
\end{proof}

\begin{prop} \label{hip2}  For $\epsilon > 0$ small we can choose $\delta_3$ small enough  such that for every $\delta' \in (0,\delta_3)$ the following holds.  Let  $\mathcal{H}$ be a $C^{1+Lip}$ function 
$$\mathcal{H}\colon \mathbb{B}^u_G(u_0,\delta') \rightarrow E^h_G+G$$
such that $\mathcal{H} \in \mathcal{T}^1_0(G,\delta',2\epsilon)$, 
where $G \in \Omega_{p,n}$. Then there exists $w \in  \mathbb{B}^u_G(u_0,\delta')$ such that $w+\mathcal{H}(w)$ is a map whose critical points belong to the same periodic orbit. 
\end{prop}
\begin{proof}  Define
$$\tilde{\mathcal{H}}\colon \mathbb{B}^u_G(u_0,\delta')\times \{v \in E^h(G)\colon \  |v|\leq \delta'  \}\rightarrow \mathcal{B}_{nor}^{\mathbb{R}}(U)$$
as  $\tilde{\mathcal{H}}(u,v)= u+\mathcal{H}(u)+v$. Let $F=u_0+ \mathcal{H}(u_0)$. Note that $\tilde{\mathcal{H}}$ is a homeomorphism on its image, which  is an open neighbourhood of $F$. Define $\tilde{\mathcal{H}}_1(u,v)= \mathcal{R}\tilde{\mathcal{H}}(u,v)$. If $\delta_3$ is small enough there is  a smooth family of domains $\hat{U}_{(u,v)}$ such that for every $(u,v)$ in the domain of $\tilde{\mathcal{H}}_1$  we have that 
$$\tilde{\mathcal{H}}_1(u,v)\colon \hat{U}_{(u,v)} \rightarrow \hat{V}$$
is a polynomial-like map of type $n$ such that 
$$\mod \hat{V}\setminus \hat{U}_{(u,v)}  > \epsilon_0.$$
Reducing  $\hat{V}$ a little bit, for every $$(w,z) \in E^u(G)\times E^h(G)$$ we have
$$D_u \tilde{\mathcal{H}}_1(u,v)\cdot w + D_v \tilde{\mathcal{H}}_1(u,v)\cdot z \in \mathcal{B}( \hat{U}_{(u,v)}).$$
Moreover if $w \in E^u(G)\setminus\{0\}$  we have
$$D_u \tilde{\mathcal{H}}_1(u,v)\cdot w  \in C^u_{2\epsilon}(\mathcal{R}G),$$
so by  Proposition \ref{nuca}, 
\begin{equation}\label{nothor} D_u \tilde{\mathcal{H}}_1(u,v)\cdot w  \not\in E^h_{\tilde{\mathcal{H}}_1(u,v)  },\end{equation}
and moreover  for every $(w,z) \in E^u(G)\times E^h(G)$ such that 
$$D_u \tilde{\mathcal{H}}_1(u,v)\cdot w + D_v \tilde{\mathcal{H}}_1(u,v)\cdot z \in E^h_{\tilde{\mathcal{H}}_1(u,v)},$$ 
we have 
\begin{equation}\label{notincone} D_u \tilde{\mathcal{H}}_1(u,v)\cdot w + D_v \tilde{\mathcal{H}}_1(u,v)\cdot z \not\in C^u_{2\epsilon}(\mathcal{R}G).\end{equation}
The image of $\tilde{\mathcal{H}}$ is an open neighbourhood of $u_0+\mathcal{H}(u_0) \in W^s_{\delta_3}(G)$, with $G\in \Omega_{n,p}$, in particular by Proposition \ref{lprop}  there exists $(u_1,v_1)$ such that $\tilde{\mathcal{H}}(u_1,v_1)$ is a map whose critical points belong to the same periodic orbit, and consequently $\tilde{\mathcal{H}}_1(u_1,v_1)=\mathcal{R} \tilde{\mathcal{H}}(u_1,v_1)$ is also a map whose critical points belong to the same periodic orbit. Furthermore one can choose $(u_1,v_1)$  arbitrarily close to $(u_0,0)$.   If $v_1=0$ choose $w=u_1$ and we finished the proof in this case. Otherwise $v_1 \neq 0$ and  we consider the $C^{1+Lip}$ smooth map 
$$(u,t,x) \in E^u(G)\times \mathbb{R} \times U  \mapsto  f_{(u,t)}(x):=\tilde{\mathcal{H}}_1(u,tv_1)(x).$$
The critical points of $f_{(u_1,1)}$ belong to the same periodic orbit, so there are natural numbers $i_k$, $k=1,\dots,n$ and  we can index the critical points $$Crit=\{ (0,j)\}_{0\leq j\leq n-1} =\{ (0,j_k)\}_{0\leq k\leq n-1}$$ in such way that for every $k\leq n-1$
$$f^{i_k}_{(u_0,1)}(0,j_k)=(0,j_{k+1 \ mod \ n}) \ and \ f^{i}_{(u_0,1)}(0,j_k)\not\in Crit \ for \ i < i_k.$$
We claim that there is  a function $t \mapsto u(t)$, defined for every $t \in [0,1]$ such that 
\begin{equation}\label{impl} f^{i_k}_{(u(t),t)}(0,j_k)=(0,j_{k+1 \ mod \ n}) \ and \ f^{i}_{(u(t),t)}(0,j_k)\not\in Crit \ for \ i < i_k,   \end{equation} 
for every  $k \leq n-1$. Indeed, let $Y$ be the set of $q \in [0,1]$ such that there exists a continuous function $u$ defined on $[q,1]$ such that (\ref{impl}) holds for every $t \in [q,1]$ and $u(1)=u_1$. Note that $1 \in Y$.  We need to show that $0 \in Y$. It is enough to show that $Y$ is a open and closed subset of  $[0,1]$.  Indeed, suppose that  $(u_2,t_2)$ satisfies \begin{equation}\label{impl2} f^{i_k}_{(u_2,t_2)}(0,j_k)=(0,j_{k+1 \ mod \ n}) \ and \ f^{i}_{(u_2,t_2)}(0,j_k)\not\in Crit \ for \ i < i_k,\end{equation} 
for every $k\leq n-1$.  Note that  the linear map 
$$w\mapsto (D_u f^{i_k}_{(u_2,t_2)}(0,j_k)\cdot w)_{0\leq k\leq n-1}$$
is invertible, otherwise it would exists $w \in E^u(G)\setminus\{0\}$ such that $$D_u f^{i_k}_{(u_2,t_2)}(0,j_k)\cdot w =0$$ for every $k$, so using the infinitesimal pullback argument one can conclude that  $$D_u \tilde{\mathcal{H}}_1(u_2,t_2v_1)\cdot w \in E^h_{ \tilde{\mathcal{H}}_1(u_2,t_2v_1)},$$ which contradicts (\ref{nothor}). So by the Implicit Function Theorem  there exists an open interval $O$ with $t_2 \in O$ such that there is a unique continuous function $u$ defined on $O$ such that (\ref{impl}) holds for every $t \in O$ and $u(t_2)=u_2$. We conclude that $Y$ is an open set and that for each $ q\in Y$ there exists an {\it unique} continuous function $U$ defined on $[q,1]$ and satisfying  and (\ref{impl}) and $u(1)=u_1$. To show that $Y$ is closed, suppose that $q_n \in Y$ is a decreasing  sequence converging to some $q \in [0,1]$. Then there exists a unique continuous function $u$ defined in $(q,1]$ such that $u(1)=u_1$ and (\ref{impl}) holds.   We claim that $u$ is a Lipchitz function on $(q,1]$, so we can extend it to a continuous function $u$ defined in  $[q,1]$. Indeed note that
$$\partial_t f_{(u(t),t)}  \in E^h_{f_{(u(t),t)}},$$
so by (\ref{notincone}) 
\begin{equation}\label{conenotb} \partial_t f_{(u(t),t)}   \not\in C^u_{2\epsilon}(\mathcal{R}G).\end{equation}
Moreover
\begin{eqnarray}\partial_t f_{(u(t),t)} &=   D\mathcal{R}_{u(t)+\mathcal{H}(u(t)) + tv_1}\cdot ( u'(t) + D_u\mathcal{H}_{u(t)}\cdot u'(t) + v_1) \end{eqnarray} 
Let
$$y = D\mathcal{R}_{u(t)+\mathcal{H}(u(t)) + tv_1}\cdot ( u'(t) + D_u\mathcal{H}_{u(t)}\cdot u'(t)).$$ Note that 
$$|y|_{\mathcal{R}G,0}\geq \lambda \frac{1-2\epsilon}{1+2\epsilon}  |u'(t)|_{G,0}$$
Suppose that $|u'(t)|_{G,0}\geq L |v_1|_{G,0}$. If $\delta'$ is small enough then  there is $C > 0$ such that 
\begin{eqnarray}
|\pi_{\mathcal{R}G,0}^u(\partial_t f_{(u(t),t)})|_{\mathcal{R}G,0} &\geq& 
\big(1- \frac{C}{L \lambda}\frac{1+2\epsilon}{1-2\epsilon} \big)  |y|_{\mathcal{R}G,0},
\end{eqnarray} 
  and
\begin{eqnarray}
|\pi_{\mathcal{R}G,0}^h(\partial_t f_{(u(t),t)})|_{\mathcal{R}G,0}&\leq&   \big(2\epsilon'+ \frac{C}{L \lambda}\frac{1+2\epsilon}{1-2\epsilon} \big)  |y|_{\mathcal{R}G,0}.
\end{eqnarray} 
If   $L$ is  large  enough then 
$$2\epsilon'+ \frac{C}{L \lambda}\frac{1+2\epsilon}{1-2\epsilon}  \leq   2\epsilon (1- \frac{C}{L \lambda}\frac{1+2\epsilon}{1-2\epsilon} ),$$
which implies that $\partial_t f_{(u(t),t)} \in C^u_{2\epsilon}(\mathcal{R}G)$. This contradicts (\ref{conenotb}). In particular there is $L$ satisfying $|u'(t)|_{G,0}\leq  L |v_1|_{G,0}$ for every $t \in (q,1]$  and  consequently $u$ is a Lipchitz function.  So  we can extend $u$ to a continuous map to $[q,1]$.  It is easy to see that (\ref{impl}) also holds for $t=q$. We conclude that $Y$ is closed. Since $Y$ is an open, closed, non empty subset of $[0,1]$ we conclude that $Y=[0,1]$, so in particular $0 \in Y$ and therefore there exists $w$ such that $f_{(w,0)}=\tilde{\mathcal{H}}_1(w,0)= \mathcal{R}(w+\mathcal{H}(w))$ is a map whose critical points belong to the same periodic orbit, and consequently $w+\mathcal{H}(w)$ has the same property. 
\end{proof}

\begin{proof}[Proof of Corollary \ref{tei}]  Let $\epsilon$ be small. It is easy to see that if $\delta_3 > 0$ is  small enough then  for every $G\in \Omega_{p,n}$ and for every $C^{1+Lip}$  function $\mathcal{H}\in  \mathcal{T}^1_0(G,\delta',\epsilon)$, with $\delta' < \delta_3$ and  for every 
$$F \in \hat{\mathcal{H}}\cap W^s_{\delta_3}(\Omega_{n,p}),$$
there is $G_F \in \Omega_{n,p}$ such that $F\in W^s_{\delta_3}(G_F)$ and $\delta''> 0$ such that 
$$\{ w+\mathcal{H}(w)\colon \  w \in   \mathbb{B}^u(v_0,\delta') \}\cap \{ u+v+G_F\colon \ u \in \mathbb{B}^u_{G_F}(\pi^u_{G_F}(F-G_F),\delta''), \ v \in E^h_{G_F}   \}$$
is the graph $\hat{\mathcal{H}}_F$ of a $C^{1+Lip}$ function in $\mathcal{T}^1_0(G_F,\delta'',2\epsilon)$. By Proposition \ref{hip2} there is $w \in \mathbb{B}^u_{G_F}(\pi^u_{G_F}(F-G_F),\delta'')$ such that $w + \mathcal{H}_F(w)$  is a map such that all its critical points belong to the same periodic orbit. In particular every map close enough to $w + \mathcal{H}_F(w)$ is a hyperbolic map with an attracting periodic orbit that attracts all its critical points. In particular we can find  hyperbolic maps in $\hat{\mathcal{H}}$ arbitrarily close to $F$. Note that hyperbolic maps  do not belong to $W^s_{\delta_3}(\Omega_{n,p})$, since every map in $W^s_{\delta_3}(\Omega_{n,p})$ is infinitely renormalizable.
\end{proof}
 
\section{Families of multimodal maps}\label{multimodalfamilies} 

In the beginning of Section \ref{teisection} we saw that a version of Theorem A. for real-analytic families of extended maps  that {\it belong} to $\mathcal{W}^{\mathbb{R}}$ can be obtained from the hyperbolicity of $\Omega_{n,p}$, the Empty Interior Transversality property and \cite[Theorem  1]{sm5}. This is not enough  to our purposes, once Theorem A. deals with real-analytic families of  multimodal maps.  Indeed a multimodal map with more than a critical point is  {\it not} an extended map. 

To prove Theorem A. we will need a classic tool, {\it inducing}. We will associate to each real-analytic multimodal map $f$ that is close enough to an infinitely renormalizable multimodal map with bounded combinatorics a renormalization  $F$ of $f$, that is an extended map    in $\mathcal{W}^{\mathbb{R}}$. Indeed a renormalizable multimodal map can be renormalizable in many ways (it can have  distinct cycles of restrictive intervals with disjoint orbits) and many times (it can have deeper and deeper renormalizations), so we need to {\it mark} $f$ with a restrictive interval $P$ in such way to make this  association 
$$(f,P)\mapsto \mathcal{I}(f,P)=F$$
well-behaved. Indeed we are going to see that $\mathcal{I}$ can be defined  in such way that it is  a real-analytic map defined in an open set of a real Banach space with image in $\mathcal{W}^{\mathbb{R}}$. The derivative $D_{(f,P)}\mathcal{I}$ of this map  has dense image at every infinitely renormalizable marked multimodal maps $(f,P)$, which allows us to use Proposition 8.1 of \cite{sm5} to conclude that $\mathcal{I}^{-1}W_\delta^s(\Omega_{n,p})$ intersects a generic real-analytic family of multimodal maps on a set of parameters with zero Lebesgue measure. This is the main  argument of the proof of Theorem A. We provide the complete proof below.

 Let $V\subset \mathbb{C}$ be a connected open set,  symmetric with respect to the real line ($z \in V$ implies $\overline{z} \in V$)   such that  $[-1,1] \subset V$. In this section we will denote by $\mathcal{B}_\mathbb{C}$ by  affine subspace of  $\mathcal{B}(V)$ defined by the restrictions $f(-1)=f(-1)=-1$.  Denote by $\mathcal{B}_{\mathbb{R}}$ the real Banach space of all functions $f \in \mathcal{B}_\mathbb{C}$ that are real on the $V\cap \mathbb{R}$.  
 
 Given $m \in \mathbb{N}$, let $\Gamma^{\omega_\mathbb{R}}_m(\mathcal{B}_{\mathbb{R}})$ be the set of all  continuous functions
 $$\gamma\colon \overline{\mathbb{D}}^m \rightarrow \mathcal{B}_\mathbb{C}$$ 
that are complex analytic on $\mathbb{D}^m$ and such that $\gamma(\lambda)\in \mathcal{B}_{\mathbb{R}}$ for every $\lambda \in [-1,1]^m$. We can endow $\Gamma^{\omega_\mathbb{R}}_m(\mathcal{B}_{\mathbb{R}})$ with the sup norm. 

Let $\Gamma\subset \mathcal{B}_{\mathbb{R}}$ be  the open subset of multimodal maps $f\colon [-1, 1]\rightarrow [-1,1]$, where $-1$ is a repelling fixed point, $f'(1)\neq 0$,  with quadratic critical points,  negative schwarzian derivative and   $f(-1,1)\subset (-1,1)$. 

Denote by $\Gamma_n \subset \Gamma^{\omega_\mathbb{R}}_n(\mathcal{B}_{\mathbb{R}})$ the subset of all families $\gamma$ such that $\gamma(\lambda) \in \Gamma$ for every $\lambda \in [-1,1]^n$. Note that $\Gamma_n$ is an open subset of $\Gamma^{\omega_\mathbb{R}}_n(\mathcal{B}_{\mathbb{R}})$.

\subsection{Generic families} Our main result for generic families is

\begin{thm}[Theorem A] \label{thmA}For every $\gamma$ in a generic subset of $\Gamma_m$, the set $\Lambda$ of parameters $\lambda$ such that $\gamma(\lambda)$  has (at least) one  solenoidal attractor with bounded combinatorics on $(-1,1)$ has zero Lebesgue measure.
\end{thm}
\begin{proof} We divide the proof in several steps.\\ 

\noindent {\it Step I (Marking restrictive intervals)}.  It turns out that a multimodal map may have many disjoint cycles of restrictive intervals. To deal with that we need to "mark" one of those restrictive intervals. To this end fix $j \in \mathbb{N}^*$ and $q, n  \in \mathbb{N}\setminus\{0,1\}$.  Let $\mathcal{O}_{j,q,n}$  be the set of all pairs $(f_0,P_0)$, such that 
\begin{itemize}
\item[A.]  The map $f_0 \in \Gamma$  has $j$ critical points in $[-1,1]$.
\item[B.] $P_0$ is a restrictive interval of $f_0$ such that each $f_0^i(P_0)$ have at most one critical point for every $i$, $\cup_i f_0^i(P_0)$ contains  $n$ critical points and   $P_0$ has a repelling periodic point in its boundary, with period $q'< q$. In particular  $f^{q'}_0(\partial P_0)\subset \partial P_0$.
\item[C.] The $f_0$-forward orbit of any critical point on the orbit of such restrictive interval $P_0$ {\it does not } fall in the orbit of such periodic point. 

\end{itemize} 

Note that the image $\pi_1(\mathcal{O}_{j,q,n})$ of the  projection onto the first coordinate in  $\mathcal{O}_{j,q,n}$   is an open subset of $\Gamma$.  Of course the countable family 
$$\{  \pi_1(\mathcal{O}_{j,q,n})\}_{j,q,n}$$
covers all infinitely renormalizable multimodal maps.

Fix $(f_0,P_0)\in \mathcal{O}_{j,q,n}$. By the implicit function theorem  the repelling  periodic point of $f_0$  in the boundary of $P_0$ has an analytic continuation  that is also repelling and it defines a restrictive interval $P_g$ for  each map $g$ in an open connected neighborhood $\mathcal{V}_0$ of $f_0$ on $\Gamma$ and such restrictive interval also satisfies properties A., B. and C.  In particular the family $\mathcal{F}$ of pairs $(\mathcal{V},P)$ where
\begin{itemize}
\item[1.] $\mathcal{V}$ is an open and connected subset of $\Gamma$, with $f_0 \in \mathcal{V}$. 
\item[2.] The real analytic function
$$g \in \mathcal{V} \mapsto P(g)$$
associate with each map $g \in \mathcal{V}$ a restrictive interval $P(g)$ of $g$ satisfying $(g,P(g))\in \mathcal{O}_{j,q,n}$ and moreover $P(f_0)=P_0$.
\end{itemize}
is non empty and consequently by Zorn's Lemma $\mathcal{F}$ has a maximal element with respect to the order $(\mathcal{V}_1,P_1) < (\mathcal{V}_2,P_2)$ if and only if $\mathcal{V}_1\subset \mathcal{V}_2$ and $P_2(g)=P_1(g)$ for every $g \in \mathcal{V}_1$.  We claim  that such  maximal element is unique.

We claim that  if $(\mathcal{V}_0,P_0), (\mathcal{V}_1,P_1) \in \mathcal{F}$ then $P_0=P_1$ on $\mathcal{V}_0\cap \mathcal{V}_1$. Indeed  since $f \in \mathcal{V}_i$, $i=0,1$,  has always $j$  critical points (moving continously with respect to $f$, since they are quadratic) and  a point $b_{f,i} \in \partial P_i(f)$ is a repelling periodic point of $f$ that is analytic continuation of $b_{f_0,0}=b_{f_0,1}$, it follows that all those  periodic points have exactly the same combinatorics with respect to the symbolic dynamics defined by partition induced by the critical points. In particular if $f \in \mathcal{V}_0\cap \mathcal{V}_1$ then $b_{f,0}, b_{f,1}$ are repelling periodic points of $f$ with the same combinatorics. Since $f$ has negative schwarzian derivative, the minimal principle implies that $b_{f,0}=b_{f,1}$. This proves the claim. 

In particular  the maximal element of $\mathcal{F}$, denoted   by $(\mathcal{V}_{f_0,P_0}, P_{f_0,P_0})$, can be described by 
$$\mathcal{V}_{f_0,P_0} =\cup_{(\mathcal{V},P) \in \mathcal{F}}  \mathcal{V}$$
and $P_{f_0,P_0}(f)=P(f)$ for every  $f \in \mathcal{V}$ satisfying  $(\mathcal{V},P)\in \mathcal{F}$. Note that 

$$\mathcal{G}=\{\mathcal{V}_{f_0,P_0}\colon (f_0,P_0) \in \mathcal{O}_{j,q,n}\}$$
is a partition of $\mathcal{O}_{j,q,n}$. We claim that such  partition  has a countable  number of elements. Indeed   suppose that
$$\{ (f_\lambda,P_\lambda) \}_{\lambda\in \Lambda} $$ is an uncountable family such that 
$$\mathcal{V}_{f_\lambda,P_\lambda} \neq \mathcal{V}_{f_\mu,P_\mu} $$
for every $\lambda\neq \mu$.  Choose a complex neighborhood $W$ of $[-1,1]$ such that  $\overline{W}\subset U$. Then there exists a sequence $\lambda _k \in \Lambda$, $k\in \mathbb{N}$, such that 
\begin{itemize}
\item[P1.] There are  $n$ and $q'< q$ such that $(f_{\lambda_k},P_{\lambda_k})$ satisfies the conditions $A$ and $B$ for every $k$. 
\item[P2.] $\lim_k(f_{\lambda_k},P_{\lambda_k})= (f_\infty,P_\infty)$ on $\mathcal{B}_\mathbb{R}(W)$, where $(f_\infty,P_\infty)$ is a multimodal map with $j$ quadratic critical points, negative schwarzian derivative and $f_\infty(-1,1)\subset (-1,1)$, and that also satisfies $A.$, $B.$ and $C.$ for the very same $n$ and $q'$ as in P1.
\item[P3.] There is $\theta > 1$ such that if $b_{\lambda_k}$ is the  repelling periodic point in the  boundary of $P_{\lambda_k}$ then $|Df_{\lambda_k}^{q'}(b_{\lambda_k})| > \theta$ for every $k$.
\item[P4.] If $k\neq k'$ then $\lambda_k\neq \lambda_{k'}$. 
\end{itemize}

By the implicit function theorem there is a ball $Y$ of $\mathcal{B}_\mathbb{R}(W)$ around $f_\infty$  and a real-analytic function $P$ defined in $Y$ such that for every $f \in Y$ we have that $P(f)$ is a restrictive interval for $f$ satisfying $A$ and $B$, and additionally $P(f_{\lambda_k})=P_{\lambda_k}$ for every large $k$.  In particular, choose $k_0, k_1$ large enough and a small connected open subset $\tilde{W} \subset \Gamma$ around the segment $\{t f_{\lambda_{k_0}}+(1-t)f_{\lambda_{k_1}}, \ t \in [0,1]\}$. Then the function $P$ is defined in $\tilde{W}$, which implies that 
$$\mathcal{V}_{f_{\lambda_{k_0}},P_{\lambda_{k_0}}} =  \mathcal{V}_{f_{\lambda_{k_1}},P_{\lambda_{k_1}}} ,$$
which is a contradiction. This completes the proof of our claim. \\

\noindent {\it Step II. (Replacing multimodal maps by extended maps of type $n$)} A real analytic multimodal map does not have the  nice structure of a multimodal map of type $n$. Fix some open set $\mathcal{V}_{f_0,P_0} \subset \mathcal{O}_{j,q,n}$. We will replace every $g \in \pi_1(\mathcal{V}_{f_0,P_0})$ by a induced map that is an extended map of type n. Denote by   $I_g$  the extended map of type $n$ that is  the renormalization of $g$ associated with the restrictive interval $P_{f_0,P_0}(g)$. Of course $I_g$ is a real-analytic extended map with negative schwarzian derivative and quadratic critical points.\\

\noindent {\it Step III. (Compexification).} Fix $p\geq 2$ Let $\mathcal{W} \subset \mathcal{B}_{nor}(U)$ be  the domain of the complexification of the $p$-bounded renormalization operator $\mathcal{R}$  as defined Section \ref{no_re}. Note that if  $I_g$ is infinitely renormalizable with $p$-bounded combinatorics we {\it don't } necessarily have that $I_g \in \mathcal{W}$. It may be the case that $I_g$ is not defined on the domain $U$, for instance. However by  the beau complex bounds given by Proposition \ref{co_bo} and the universality result in Proposition \ref{contraction2}     there is $k$  such that $R^{k'}(I_g)\in \mathcal{W}$ for every  $k'\geq k$. Again by  the beau complex bounds and Proposition \ref{contraction2}  we can find  open subsets $\mathcal{V}_{f_0,P_0}^k\subset \mathcal{V}_{f_0,P_0}$ such that for 
\begin{itemize}
\item[E1.] For every $g \in \mathcal{V}_{f_0,P_0}^k$ we have that $\mathcal{I}_k(g)=R^{k}(I_g)$ is well defined and it belongs to $\mathcal{W}$.
\item[E2.] For every $g \in \mathcal{V}_{f_0,P_0}^k$ that is infinitely renormalizable with $p$-bounded combinatorics we have $R^{k'}(I_g)\in \mathcal{W}$ for every  $k'\geq k$.
\item[E3.] We have that
$$\cup_k \mathcal{V}_{f_0,P_0}^k$$
contains all  infinitely maps in $ \mathcal{V}_{f_0,P_0}$ which are renormalizable with $p$-bounded combinatorics.
\end{itemize}

The operator $\mathcal{I}_k$ has a complexification (it can the proven using {\it exactly}  the same argument as in the complexification of the renormalization operator in Section \ref{no_re}). From now on we restrict $\mathcal{I}_k$  to  real maps. Note that  the image of the operator $D_g\mathcal{I}_k$ is dense in $T\mathcal{B}_{nor}^{\mathbb{R}}(U)$ (again, the argument is the same as with the renormalization operator in Theorem  \ref{dense}).\\

\noindent {\it Step IV. (Applying the hyperbolicity of $\Omega_{n,p}$).} Due Theorem \ref{hip1} we have that  $\Omega_{n,p}$ is a hyperbolic invariant set of $\mathcal{R}$.  Moreover  Corollary \ref{tei} says that $W^s(\Omega_{n,p})$ has transversal empty interior.  Theorem  \ref{dense} tells us that 
$D_F\mathcal{R}(T_{F}\mathcal{B}_{nor}^{\mathbb{R}}(U))$ is dense in $T_{\mathcal{R}F}\mathcal{B}_{nor}^{\mathbb{R}}(U)$ for every $F \in \mathcal{W}^{\mathbb{R}}$. So we conclude that $\mathcal{R}$ (restricted to real maps) satisfies the assumptions of Theorem 1 in  \cite{sm5} (taking $k=\omega_{\mathbb{R}}$ there). Now  we can apply  Proposition 8.1 of \cite{sm5}  taking $\mathcal{M}=\mathcal{I}_k$ to conclude that for a generic $\gamma \in \Gamma^{\omega_\mathbb{R}}_n(\mathcal{B}_{\mathbb{R}})$ the set of parameters $\lambda \in [-1,1]^n$ where $\mathcal{I}_k(\gamma(\lambda))$  is infinitely  renormalizable with $p$-bounded combinatorics has zero Lebesgue measure. Since there is just a  countable number of choices for $k$, elements of $\mathcal{G}$,  $q$, $p$ and $j$, we concluded the proof.
\end{proof}

Indeed  Proposition 8.1 of \cite{sm5} implies an analogous result for finitely differentiable families of maps in $\Gamma$. We refer the reader to \cite{sm5} for additional statements and definitions for this setting.

\subsection{ Transversal families of polynomial-like maps} 

Recall the definition of $\hat{E}^h_{f}$ and $\hat{E}^v_{f}$ in Section \ref{vdir}. 

\begin{thm}[Transversal families]\label{sv} Let $\Lambda$ be an open subset of $\mathbb{C}^d$. Let 
$$\lambda \in \Lambda \mapsto f_\lambda\colon V^1_\lambda \rightarrow V^2_\lambda$$
be a complex analytic family of polynomial-like maps such that for every   $\lambda \in \Lambda\cap \mathbb{R}^d$ we have that $V^1_\lambda, V^2_\lambda$ are symmetric with respect with $\mathbb{R}$, its real trace is an interval, $f_\lambda(x) \in \mathbb{R}$ for every $x \in \mathbb{R}$, $f_\lambda$ has negative Schwarzian derivative and just quadratic critical points on the real line.  Suppose that for every $\lambda_0 \in \mathbb{R}^d$  such that $f_{\lambda_0}$ is infinitely renormalizable with bounded combinatorics we have that \\ \\
\noindent {\bf (Transversality assumption.)} Every holomorphic vector in a neighborhood of $K(f_{\lambda_0})$ can be written as  a sum of a vector in $\hat{E}^h_{f_{\lambda_0}}$ and a vector in  $$D_\lambda f_\lambda  |_{\lambda=\lambda_0} (\mathbb{R}^d).$$
\ \\
\noindent Then the set of parameters $\lambda \in \Lambda\cap \mathbb{R}^d$ where $f_\lambda$ is infinitely renormalizable with bounded combinatorics   has zero $d$-dimensional Lebesgue measure. 
\end{thm} 
\begin{proof} We can find a countable family  of domains $U_i\subset \mathbb{C}$, symmetric with respect to $\mathbb{R}$,  and open subsets $\Lambda_i\subset \Lambda \cap \mathbb{R}^d$ such that 
$\cup_ i \Lambda_i=\Lambda \cap \mathbb{R}^d$,  $f_\lambda\in \mathcal{B}(U_i)$, for $\lambda\in \Lambda_i$, with $K(f_\lambda)\cap \mathbb{R} \subset U_i\cap \mathbb{R}$, and 
$$\lambda \in \Lambda_i \mapsto f_\lambda \in \mathcal{B}(U_i)$$ is an real analytic family.  It is enough to prove the conclusion of Theorem \ref{sv} for each one of those families. So fix $i$.  The proof goes as the proof of Theorem \ref{thmA}. We can define the sets $\mathcal{O}_{j,q,n}$ replacing the pairs $(f,P)$ by pairs of the form $(\lambda, P)$, where $\lambda \in \Lambda_i$ and $P$ is a restrictive interval of $f_\lambda$. In a similar way we can define $(\mathcal{V}_{\lambda_0,P_{\lambda_0}},P_{\lambda_0,P_{\lambda_0}})$, the sets  $\mathcal{V}_{\lambda_0,P_{\lambda_0}}^k\subset \Lambda_i$ and the  real-analytic parametrized families
$$\lambda \in \mathcal{V}_{\lambda_0,P_{\lambda_0}}^k \rightarrow \mathcal{I}_k(f_\lambda)\in \mathcal{W}.$$
Suppose that  $f_\lambda$ is infinitely renormalizable with $p$-bounded combinatorics. Then 
$$D_{f_\lambda}  \mathcal{I}_k (E^h_{f_\lambda})  \subset E^h_{\mathcal{I}_k(f_\lambda)}.$$
This follows exactly as the proof of Proposition \ref{invariance}. On the other hand we know (see the proof of Theorem \ref{thmA}) that 
 $Im \ D_{f_\lambda} \mathcal{I}_k$ is dense in $T_{\mathcal{I}_k(f_\lambda)}\mathcal{B}_{nor}^{\mathbb{R}}(U).$
By the Transversality assumption  this  implies that there is a subspace $S_\lambda \subset \mathbb{R}^d$, with $\dim S_\lambda =n$,  such that
$$D_{f_\lambda} \mathcal{I}_k \cdot D_\lambda f_\lambda (S_\lambda) \pitchfork E^h_{\mathcal{I}_k(f_\lambda)}.$$
Suppose that $\lambda_0$ is such that $f_{\lambda_0}$ is infinitely renormalizable with $p$-bounded combinatorics. Let $v_1, \dots, v_n, v_{n+1}, \dots, v_{d}$ be a basis of $\mathbb{R}^d$ such that $v_1, \dots, v_n$ is a basis for $S_{\lambda_0}$. Then for every $\gamma=(\gamma_1,\dots,\gamma_{d-n}) \in \mathbb{R}^{n-d}$ that is small enough we have that the family
$$\theta=(\theta_1,\dots,\theta_n)\mapsto g_\theta= f_{\lambda_0+ \sum_{i=1}^{n}  \theta_i v_i+ \sum_{i=n+1}^{d}  \gamma_{i-n} v_i },$$
where $\theta$ is also small, satisfies
$$D_{g_\theta} \mathcal{I}_k \cdot D_\theta g_\theta (\mathbb{R}^n) \pitchfork E^h_{\mathcal{I}_k(g_\theta)}.$$
So by \cite[Corollary 10.2]{sm5} we have that for every small $\gamma$, the set of small parameters $\theta$ such that $g_\theta$ is infinitely renormalizable with bounded combinatorics has zero $n$-dimensional Lebesgue measure. By the Fubini's Theorem it follows that in a small neighborhood the parameter $\lambda_0$ the set of parameters $\lambda$ such that $f_\lambda$ is infinitely renormalizable with $p$-bounded combinatorics has zero $d$-dimensional Lebesgue measure. This completes the proof. 
\end{proof}

\begin{proof}[Proof of Theorem C] Let  $f_{\lambda_1,\lambda_2}(z)=z^3-3\lambda_1^2z +\lambda_2$. Note that if $\lambda_1=0$ then $f_{\lambda_1,\lambda_2}$ is not infinitely renormalizable, so we assume that $\lambda_1\neq 0$.  Let $\lambda_0=(a,b)$, $a\neq 0$.  Then 
$$\partial_\lambda f_\lambda|_{\lambda=\lambda_0}  (\mathbb{R}^2) =\{  cz+d , \  c,d \in \mathbb{R}\}$$
By Proposition \ref{v444} we have that $\hat{E}^v_{f_\lambda}$ is the space of cubic polynomials, so   $\dim \hat{E}^v_{f_\lambda} =4$ and 
$$\partial_\lambda f_\lambda|_{\lambda=\lambda_0}  (\mathbb{R}^2) \subset  \hat{E}^v_{f_\lambda}.$$
We claim  that 
\begin{equation}\label{hv}  \{ 2z^3-b, -3z^2-3a^2-1\} \subset  \hat{E}^h_{f_\lambda}\cap \hat{E}^v_{f_\lambda}\end{equation} 
Indeed, let $H_t(z)=z/t$. Then 
$$2z^3 -b = \partial_t    (H_t\circ f_{a,b}\circ H_t^{-1}(z))|_{t=1},$$
so $2z^3-b = \alpha_1(f_{a,b}(z))-Df_{a,b}(z)\alpha_1(z)$, where $\alpha_1(z)=\partial_t H_t(z)|_{t=1}= -z$.  Let $S_t(z)=z-t$. Then
$$3z^2-3a^2-1 = \partial_t    (S_t\circ f_{a,b}\circ S_t^{-1}(z))|_{t=0},$$
so $3z^2-3a^2-1= \alpha_2(f_{a,b}(z))-Df_{a,b}(z)\alpha_2(z)$, where $\alpha_2(z)=\partial_t S_t(z)|_{t=0}= -1$. 
Since $\{ 1,z, 2z^3 -b,3z^2-3a^2-1\}$ is a basis of $\hat{E}^v_{f_\lambda}$, by (\ref{somadi}) and  (\ref{hv}) we have that every holomorphic vector in a neighborhood of $K(f_{a,b})$ can be written as  a sum of a vector in $\hat{E}^h_{f_{a,b}}$ and a vector in  $\partial_\lambda f_\lambda|_{\lambda=\lambda_0}  (\mathbb{R}^2)$. Now we apply  Theorem  \ref{sv}. \end{proof}

\subsection{Compositions of quadratic maps} 

We can say something about a specific family of extended maps of type $n$. This family was  introduced in \cite{sm2}.  Let $\lambda=(\lambda_i)_{i\leq n},$ with $\lambda_i \in [0,1]$ and define
$$F_\lambda \colon \mathbb{C}\times \{ i\}_{i\leq n} \rightarrow \mathbb{C}\times \{ i\}_{i\leq n}$$
as $F_\lambda(z,i)=(-2\lambda_i z^2 + 2\lambda_i-1,i+1 \mod n).$

It follows from the study in \cite{sm2} that each  possible combinatorial type of an infinitely renormalizable extended map of type $n$ with combinatorics bounded by $p$ can be realized by a unique parameter in $[0,1]^n$ and  the set of such parameters $\Lambda_{p,n} \subset [0,1]^n$ is a Cantor set \cite[Theorem 2]{sm2}. The following result  answers a conjecture in \cite{sm2}. 

\begin{thm}We have that $m(\Lambda_{p,n})=0$, where $m$ is the $n$-dimensional Lebesgue measure. 
\end{thm}
\begin{proof} Due Corollary 10.2 in \cite{sm5},  it is enough to show that this family is transversal to the horizontal distribution $F\rightarrow E^h_F$.  We will give a proof similar  to the proof of the transversality of the quadratic family by  Lyubich \cite{lyu}. Indeed, suppose by contradiction that there exists $\lambda_0 \in \Lambda_{p,n}$ and $w \in \mathbb{R}^n\setminus \{0\}$ such that 
$$v=\partial_\lambda F_\lambda|_{\lambda=\lambda_0} \cdot w \in E^h_{F_{\lambda_0}}.$$
So there is a quasiconformal vector field $\alpha_0$, defined in a neighborhood of the postcritical set $P(F_{\lambda_0})$ 
satisfying (\ref{tcc}) on $P(F_{\lambda_0})$.  Since this is a family of polynomials, the conformal dynamics outside the Julia set of $F_{\lambda_0}$ is  
always the same, so $v$ is also a {\it vertical} direction, that is, there exists a conformal vector field $\alpha_1$ defined outside the Julia set such that (\ref{tcc}) holds outside the Julia set.  Using the infinitesimal pullback argument we can find a quasiconformal vector field solution $\alpha$ that satisfies (\ref{tcc}) everywhere and moreover it is conformal outside the Julia set. Since $F_{\lambda_0}$ does not support invariant line fields on its Julia set we conclude that $\alpha$ is conformal everywhere and indeed it is equal to zero, since it is zero at three points of $\overline{\mathbb{C}}\times \{i\}$, for each $i\leq n$. So $v=0$, which implies $w=0$. 
\end{proof}

\begin{rem} Note that $\Lambda_{p,n}$ only includes the parameters where each renormalization involves all $n$ critical points, that is, each cycle of intervals covers all critical points.    If we consider infinitely renormalizable maps where fewer points are involved then the set of parameters it is not a Cantor set anymore. However it is likely that this  larger subset of parameters also have zero Lebesgue measure. 
\end{rem}

\bibliographystyle{abbrv}
\bibliography{bibliografia}

\end{document}